\documentclass{amsart}
\usepackage{amsmath, accents, amsfonts, amscd, amssymb, amsthm, verbatim, psfrag, dsfont,enumerate,url}
\usepackage{color}
\newcommand{\R}{\mathbb{R}}

\newcommand{\la}{\langle}
\newcommand{\ra}{\rangle}
\newcommand{\nab}{\nabla}

\DeclareMathOperator{\tr}{tr}
\DeclareMathOperator{\dist}{dist}
\DeclareMathOperator{\Div}{div}

\DeclareMathOperator{\const}{const}
\DeclareMathOperator{\Area}{area}
\DeclareMathOperator{\Span}{span}
\DeclareMathOperator{\Hess}{Hess}
\DeclareMathOperator{\Hom}{Hom}
\DeclareMathOperator{\Rank}{Rank}
\DeclareMathOperator{\End}{End}

\DeclareMathOperator{\inj}{inj}
\DeclareMathOperator{\Int}{Int}
\DeclareMathOperator{\pr}{pr}

\DeclareMathOperator{\Genus}{Genus}
\newtheorem{lemma}{Lemma}[section]
\newtheorem{proposition}[lemma]{Proposition}
\newtheorem{theorem}{Theorem}
\newtheorem{corollary}[lemma]{Corollary}
\newtheorem{definition}[lemma]{Definition}
\newtheorem{remark}[lemma]{Remark}

\newtheorem*{introtheorem1}{Theorem A}
\newtheorem*{introtheorem2}{Theorem B}
\newtheorem*{introtheorem3}{Theorem C}
\newtheorem*{introtheorem4}{Theorem D}

\begin{document}
\title{Estimates for $J$-curves as submanifolds }
\author{Joel W. Fish}
\address{Department of Mathematics, Stanford University, Stanford, CA 94305}
\email{joelfish@math.stanford.edu}
\urladdr{http://www.stanford.edu/~joelfish}
\date{\today}
\thanks{Research partially supported by NSF grants DMS-0602191 and DMS-0802927}
\begin{abstract}
Here we develop some basic analytic tools to study compactness properties of $J$-curves (i.e. pseudo-holomorphic curves) when regarded as submanifolds. Incorporating techniques from the theory of minimal surfaces, we derive an inhomogeneous mean curvature equation for such curves, we establish an extrinsic monotonicity principle for non-negative functions $f$ satisfying $\Delta f\geq -c^2 f$, we show that curves locally parameterized as a graph over a coordinate tangent plane have all derivatives a priori bounded in terms of curvature and ambient geometry, and we establish $\epsilon$-regularity for the square length of their second fundamental forms. These results are all provided for $J$-curves either with or without Lagrangian boundary and hold in almost Hermitian manifolds of arbitrary even dimension (i.e. Riemannian manifolds for which the almost complex structure is an isometry).
\end{abstract}
\maketitle
\tableofcontents
\section{Introduction}\label{sec:Introduction}
In his seminal 1985 paper \cite{Gm85}, Gromov introduced the notion of a ``pseudo-holomorphic curve'' and established the fundamental notion of compactness for families of such $J$-curves.  Since then, the vast majority of modifications and generalizations to Gromov's compactness theorem have all followed the same basic recipe, namely to study $J$-curves as a type of special harmonic map.  This essentially reduces the compactness problem to applying Deligne-Mumford compactness to the underlying Riemann surfaces and then applying bubbling analysis.  However, there are a growing number of examples in which this approach badly breaks down -- for instance, $J$-curves in a family of symplectic manifolds which lacks a uniform energy threshold.  Such a case was considered in the author's Ph.D. thesis \cite{Fj07}, in which a compactness result was proved for $J$-curves in the connected sum of two contact manifolds for which the connecting handle collapsed to a point.  The key difficulty there and in several current research directions (particularly $J$-curves in symplectic cobordisms between non-compact and/or degenerate contact manifolds) is the lack of a uniform energy threshold.  Indeed, the lack of this quantity is so fundamental that it necessitates an alternate approach to the compactness problem: namely, to regard $J$-curves as submanifolds and to prove compactness by incorporating elements from minimal surface theory.  Such an approach was taken in \cite{Fj09a}, and the results proved below constitute the core analytic estimates on which those arguments rely, namely extrinsic monotonicity, graphical parameterizations with a priori bounds, and $\epsilon$-regularity of the square-length of the second fundamental form.  In Section \ref{sec:Applications} below, we restate the main result from \cite{Fj09a} and sketch a proof with an emphasis on applications of the estimates proved below.

We now make two important points.  First, the results that follow are established for $J$-curves either with or without a partial Lagrangian boundary condition.  Consequently, the target-local results proved in \cite {Fj09a} (which impose no such boundary condition) appear to be generalizable to the Lagrangian boundary case.

The second point is that (the image of) a $J$-curve satisfies an inhomogeneous mean curvature equation of a form which allows for a variety of minimal surface techniques to be adapted to prove the main results of this article.  While these arguments appear to be familiar (if not well known) to the Riemannian geometry community, they seem to be unknown to the pseudo-holomorphic curve community at large. Furthermore the main theorems proved in \cite{FjX10a} and \cite{Fj09a} critically rely on the estimates below, necessitating a precise formulation of these results.
Consequently, a distinct effort has been made to keep the Riemannian aspects of the analysis as clear as possible and reasonably elementary.  For instance, we attempt to either reference well-known results or else provide proofs of basic results for completeness. Although this approach lengthens the exposition, it has hopefully increased the paper's readability and accessibility for the target audience.

\subsection{Statement of results}\label{sec:StatementOfResults}
Here we collect slightly simplified versions of the main results established in this article; see Section \ref{sec:Preliminaries} for relevant definitions.  For convenience, we will only state here the results for $J$-curves which lack a partial Lagrangian boundary condition.  The first such result concerns extrinsic monotonicity, and is a modified statement of Theorem \ref{thm:MonotonicityType0} from Section \ref{sec:Monotonicity0} below.  The Lagrangian boundary case is stated as Theorem \ref{thm:MonotonicityType1and2} from Section \ref{sec:Monotonicty1and2} below.

\begin{introtheorem1}
Let $(M,J,g)$ be an almost Hermitian manifold.  Then there exists a constant $C$, depending only on the geometry of $(M,J,g)$, with the following significance.  For any compact $J$-curve $u:S\to M$ which lacks constant components and satisfies $\partial S = u^{-1}(\partial M)$, and for any point $p\in u(S\setminus \partial S)\subset M\setminus \partial M$, and positive constants $b$ and $\lambda$ where $b$ is sufficiently small relative to the geometry of $(M,J,g)$ near $p$, and function $f:S\to \R$ for which $f\geq 0$ and $\Delta_S f\geq -\lambda b^2 f$, the following holds for any $0<a<b$:
\begin{equation*}
e^{\frac{\lambda a}{2b}+2 Ca}a^{-2 }\int_{u^{-1}(\mathcal{B}_a(p))}f\leq e^{\frac{\lambda}{2}+2 Cb}b^{-2
}\int_{u^{-1}(\mathcal{B}_b(p))}f.
\end{equation*}
Here $\mathcal{B}_r(p)=\{q\in M: \dist(p,q)<r\}$, and integration is taken with respect to the Hausdorff measure associated to $u^*g$, or equivalently (since $u$ parameterizes a $J$-curve) the two-form $u^*\omega$ with $\omega:= g\circ(J\times \mathds{1})$.
\end{introtheorem1}

Although the above estimate is perhaps foreign in its current form, it is not too difficult to see that for particular choices of $f$ and constants, a number of results can quickly be deduced, such as the extrinsic mean value inequality for sub-harmonic functions and a version of the monotonicity of area lemma which is stronger than the version commonly cited for $J$-curves; see Corollaries \ref{cor:MeanValueInequality0}, \ref{cor:MonotonicityOfArea0}, \ref{cor:MeanValueInequality1}, \ref{cor:MonotonicityOfArea1} below.  It should also be pointed out that the above result holds regardless of the topology of $u^{-1}\big(\mathcal{B}_r(p)\big)$.  Also of possible interest is that a very similar result holds for $J$-curves which have two (totally geodesic) Lagrangian boundary conditions and the Lagrangians intersect transversally.  See Section \ref{sec:Monotonicity} below.

The next result, Theorem B below, should not be regarded as a stand-alone result, but rather as a very convenient computational tool when working in local coordinates in $M$.  Indeed, it is used in Section \ref{sec:CurvatureRegularity} to establish various regularity results for the second fundamental form of a $J$-curve, and it is critically used in \cite{Fj09a} to obtain compactness results for sequences of immersed $J$-curves with a priori bounded curvature but potentially unbounded topology.
It is a modified restatement of Theorem \ref{thm:AprioriDerivataiveBounds} from Section \ref{sec:InteriorCase} below. The case with partial Lagrangian boundary condition can be found in Section \ref{sec:BoundaryCase} and is stated there as Theorem \ref{thm:AprioriDerivataiveBounds2}.

\begin{introtheorem2}
Let $(M,J,g)$ be an almost Hermitian manifold, $\epsilon>0$, and $u:S\to M$  a compact immersed $J$-curve with $\partial S = u^{-1}(\partial M)$ which has second fundamental form $B$ that satisfies $\|B\|_{L^\infty}\leq 1$.  Then there exists constants $r_0>0$ and $C_0,C_1,\ldots$ which depend only on $\epsilon$ and the local geometry of $(M,J,g)$ with the following significance.  For each $\zeta\in S$ for which $\dist\big(u(\zeta),\partial M\big)\geq \epsilon$, there exist maps $\phi:\mathcal{D}_{r_0}\to S$ and geodesic normal coordinates $\Phi:\mathcal{B}_{2r_0}\big(u(\zeta)\big)\to \R^{2n}$ for which the following hold
\begin{enumerate}
\item $\tilde{u}(s,t):=\Phi\circ u\circ \phi (s,t) = \big(s,t,\tilde{u}^3(s,t),\ldots,\tilde{u}^{2n}(s,t)\big)$,
\item $\phi(0)=\zeta$, $\tilde{u}(0,0) = 0$, and $D_\alpha \tilde{u}^i (0,0) = 0$ for $|\alpha|=1$ and $i=3,\ldots,2n$,
\item $\|\tilde{u}\|_{C^k(\mathcal{D}_{r_0})} \leq C_k$ for $k\in \mathbb{N}$.
\end{enumerate}
\end{introtheorem2}

In short, Theorem B guarantees that any $J$-curve with a priori bounded curvature can be locally parameterized as a graph over a disc (of radius uniformly bounded away from zero) in a coordinate tangent plane in such a way that all derivatives of said parameterization are uniformly bounded independent of the $J$-curve.

The final consequence of this article is an $\epsilon$-regularity result for the square length of the second fundamental form.  These are stated below as Theorem \ref{thm:EpsRegularity0} and Theorem \ref{thm:EpsRegularity1},  however rather than restate them here, we instead state Theorem C, which follows as an immediate corollary and has more transparent applications.

\begin{introtheorem3}
Let $(M,J,g)$ be a compact almost Hermitian manifold, and $\epsilon>0$.  Then there exists $\hbar>0$ depending only on $\epsilon$ and the local geometry of $(M,J,g)$ with the following significance.  For each compact immersed $J$-curve $u:S\to M$ with $\partial S:= u^{-1}(\partial M)$ and each $\zeta\in S$ with $\dist\big(u(\zeta),\partial M\big)>\epsilon$, the following holds for every $0<r<\hbar$:
\begin{equation*}
\text{if}\qquad \|B_u(\zeta)\|_g \geq \frac{1}{r}\qquad \text{then}\qquad \int_{u^{-1}(\mathcal{B}_r(u(\zeta)))}\|B_u\|^2 \geq \hbar.
\end{equation*}
Here $B_u$ is the second fundamental form of the immersion $u$, and $\mathcal{B}_r(p):=\{q\in M:\dist_g(p,q)<r\}$.
\end{introtheorem3}

Roughly speaking, Theorem C guarantees that if the curvature $\|B\|$ of a $J$-curve is large at a point, then a small neighborhood of this point captures a threshold amount of total curvature $\int \|B\|^2$.  This is relevant since integrating the Gauss equations for immersed closed $J$-curves shows that the total curvature $\int_S \|B\|^2$ is bounded in terms of the genus of $S$, the area of $S$, and the geometry of $(M,J,g)$.  Indeed, this estimate and such an a priori bound on the total curvature form the foundation of the compactness results given in \cite{Fj09a}; a sketch of this argument is provided in the next section.

\subsection{Applications}\label{sec:Applications}
The above results are technical in nature, so we take a moment to discuss their application in the target-local compactness result proved in \cite{Fj09a}.  Indeed, the main result of that article is the following:

\begin{introtheorem4}
Let $(M,J,g)$ be a compact almost Hermitian\footnote{That is, for $(M,J,g)$ we require that $g$ is a Riemannian metric, and the almost complex structure $J$ is an isometry.} manifold with boundary.  Let $(J_k,g_k)$ be a sequence of almost Hermitian structures which converge to $(J,g)$ in $C^\infty(M)$, and let $(u_k,S_k,j_k,J_k)$ be a sequence of compact $J_k$-curves (possibly disconnected, but having no constant components) satisfying the following:
\begin{enumerate}
\item $u_k:\partial S_k \to \partial M$,
\item $\Area_{u_k^*g_k}(S_k)\leq C_{A}$,
\item $\Genus(S_k)\leq C_{G}$.
\end{enumerate}
Then there exists a subsequence (still denoted with subscripts $k$) of the $\mathbf{u}_k$, an $\epsilon>0$, and an open dense set $\mathcal{I}\subset [0,\epsilon)$ with the following significance.  For each $\delta \in \mathcal{I}$, define $\tilde{S}_k^\delta:=\{\zeta\in S_k: \dist_g\big(u_k(\zeta),\partial M)\geq \delta\}$; then the $J_k$-curves $(u_k,\tilde{S}_k^\delta,j_k,J_k)$ converge in a Gromov sense.
\end{introtheorem4}

There are a variety subtleties which make the proof of this theorem less than obvious.  For instance, one serious problem here is that we have not assumed that the number of connected components of the $\partial S_k$ is uniformly bounded.  Since the $S_k$ need not be diffeomorphic, we see that the standard arguments (i.e. Deligne-Mumford convergence of the underlying Riemann surfaces and then bubbling analysis) will be useless. On the other hand, if one could pass to a subsequence and find $0<\delta<\epsilon$ for arbitrarily small $\delta$ such that the curves $(u_k,\tilde{S}_k^{\epsilon,\delta},j_k,J_k)$ with
\begin{equation*}
\tilde{S}_k^{\epsilon,\delta}:=\{\zeta\in S_k: \epsilon\geq \dist_g\big(u_k(\zeta),\partial M)\geq \delta\}
\end{equation*}
converged in $C^\infty$ to an immersed curve, then one could use this convergence plus standard techniques to prove the above theorem. Indeed, this is precisely the tack taken in \cite{Fj09a}, so let us now see the relevance of the estimates proved below.

We begin by assuming the $u_k:S_k\to M$ are immersed; this is not an assumption in \cite{Fj09a}, but it will make the discussion in this section clearer.  The desired result is then proved in three steps. The first step is to prove Theorem D under the additional assumption that the second fundamental forms $B_{u_k}$ of the immersions $u_k$ are uniformly bounded in $L^\infty$, and that the limit curve is immersed and not nodal.  Since it is again the case that the domains $S_k$ (or trimmed domains $\tilde{S}_k^\delta$) are not diffeomorphic, one may still not employ standard arguments for Gromov compactness.  Instead, one employs Theorem B to cover the $\tilde{S}_k^\delta$ by open sets of the form $\mathcal{U}_{k,i}=\phi_{k,i}(\mathcal{D}_{r_0})$, where  $i\in \{1,\ldots, n_k\}$; a further consequence of Theorem B is that the maps $\tilde{u}_{k,i}:=\Phi_{k,i}\circ u_k\circ \phi_{k,i}$ have a priori bounds on all derivatives and $r_0$ is independent of $i$ and $k$.  Furthermore, by employing the uniform area bound, Theorem A (or other weaker versions of monotonicity), and a standard covering argument, one finds that the $\mathcal{U}_{k,i}$ can be chosen so that the $n_k$ are uniformly bounded.  Since the $\tilde{u}_{k,i}$ are graphical with uniformly bounded derivatives, it follows from Arzel\`{a}-Ascoli that after passing to a subsequence the $\tilde{u}_{k,i}\to \tilde{u}_{\infty,i}$ in $C^\infty$.  Using these limit maps, one passes to a further subsequence, constructs a surface $\tilde{S}$, an immersion $\tilde{u}:\tilde{S}\to M$, and diffeomorphisms $\psi_k:\tilde{S}\to\psi_k(\tilde{S})\subset S_k$ such that $S_\delta\subset \psi_k(\tilde{S})$ and $u_k\circ \psi_k\to \tilde{u}$ in $C^\infty$.  This shows the desired convergence under the additional $L^\infty$ curvature bound.

Of course, sequences of immersed $J$-curves often have unbounded curvature $\|B_{u_k}\|_{L^\infty}$.  Indeed, consider the formation of a critical point: $u_k(z)=(k^{-1} z, z^2)$, or the formation of the standard node: $\{(z_1,z_2)\in \mathbb{C}^2:z_1 z_2 = 1/k\}$ appropriately parameterized. Thus the second step is to prove Theorem D for immersed curves under the additional assumption that the total curvature $\int_{S_k}\|B_{u_k}\|^2$ is uniformly bounded.  Note that if the $S_k$ were each closed, then the total curvature would be bounded in terms of $C_G$, $C_A$, $\|K_{sec}\|_{L^\infty}$, and $\|\nabla J\|_{L^\infty}$.  Indeed this follows by integrating the Gauss equations, employing the Gauss-Bonnet theorem,  and using the fact that $J$-curves have mean curvature $H_\nu$ which is $L^\infty$-bounded in terms of $\|\nabla J\|_{L^\infty}$.  The difficulty with deriving a similar bound for non-closed $S_k$ once again arises from the lack of a priori control of boundary behavior.  Nevertheless, assuming uniformly bounded total curvature $\int_{S_k} \|B_{u_k}\|^2 \leq C_B$, we see that Theorem C guarantees that the curvature of our curves cannot blowup at arbitrarily many points in the interior of $M$.  More precisely, if $\{\zeta_{1,k},\ldots,\zeta_{n,k}\}\subset S_k$ is a sequence of finite sets for which $\lim_{k\to \infty} \min\{\|B_{u_k}(\zeta_{1,k})\|,\ldots,\|B_{u_k}(\zeta_{n,k})\|\} = \infty$, $\inf_{i,k} \{\dist_{g_k}\big(u_k(\zeta_{i,k}),\partial M\big)\} >0$, and $\inf_{k\in \mathbb{N}} \min_{i\neq j} \dist_{g_k}\big(u_k(\zeta_{i,k}),u_k(\zeta_{j,k})\big)>0$, then $n$ is bounded in terms of $C_B$, and the geometry of $(M,J,g)$.  Consequently, by passing to a subsequence one may find $\epsilon,\delta>0$ as small as we like such that the $\|B_{u_k}\|_{L^\infty(\tilde{S}^{\epsilon,\delta})}$ are uniformly bounded.  This reduces the problem to the one solved in our first step, and the desired result is immediate.

The third and final step in proving Theorem D for immersed curves is then to show the following.  For each $\delta>0$ there exists a $C_\delta>0$ such that $\int_{\tilde{S}_k^\delta}\|B_{u_k}\|^2 < C_\delta$.  Here $C_\delta$ will depend on $C_A$, $C_G$, and the geometry of $(M,J,g)$.  The proof is via contradiction, and employs the results of the previous two steps.  The key idea of the proof is the following. Since $J$-curves have mean curvature bounded in terms of $\|\nabla J\|$, it follows from the Gauss equations that the Gaussian curvature of such curves is bounded from above in terms of the sectional curvature of $M$ and $\nabla J$.  One can then write down a differential equation relating the area and curvature of intrinsic disks $\mathcal{D}_r:=\{\zeta\in S: \dist_{u^*g}(\zeta,\zeta_0)\}$, from which in can be concluded that if the total curvature on $\mathcal{D}_r$ is very large, then the area of $\mathcal{D}_{2r}$ is very large as well.  Since area is a priori bounded, so too must the total curvature.  The desired estimate is then obtained via a covering argument.  By the work in step two, the desired result is then immediate.  This completes the sketch proof the Theorem D for immersed curves, and it shows how the proof of that theorem critically relies on the estimates proved below.

\subsection{Acknowledgements}
The following is an extension of ideas developed in my Ph.D. thesis \cite{Fj07} at New York University, and as such I would like to thank my advisor, Helmut Hofer, for his encouragement, support, and for creating a vibrant symplectic and contact research group at the Courant Institute.

\section{Preliminaries}\label{sec:Preliminaries}
The purpose of this section is to establish some definitions, notation, and elementary results which will be used throughout the remainder of this article.  More precisely, in Section \ref{sec:DefAndNot} we define generally immersed $J$-curves of various Types in almost Hermitian manifolds, and in Section \ref{sec:MeanCurvatureEquation} we define the second fundamental forms $A$ and $B$ and the mean curvature vector $H$ of generally immersed curves, and we show that $J$-curves satisfy an inhomogeneous mean curvature equation (see Lemma \ref{lem:InhomogeneousMeanCurvatureEq} below).

\subsection{Definitions and notation}\label{sec:DefAndNot}

Let $M$ be a compact real $2n$-dimensional manifold (possibly with boundary) equipped with a smooth section $J\in \Gamma\big(\End(TM)\big)$ for which $J^2=-\mathds{1}$; we call $(M,J)$ an \emph{almost complex manifold}, and $J$ the \emph{almost complex structure}.  Note that $J$ need not be integrable; that is, it need not be induced from local complex coordinates.  Indeed, this will only be true if the Nijenhuis tensor $N_J$ associated to $J$ vanishes identically, and do not make such an assumption.

If $(M,J)$ is equipped with a smooth Riemannian metric $g$ for which $J$ is an isometry (i.e. $g(x,y)=g(Jx,Jy)$ for all $x,y\in TM$), then we call $(M,J,g)$ an \emph{almost Hermitian manifold}.  Observe that any almost complex manifold can be given an almost Hermitian structure $(J,g)$ by choosing an arbitrary Reimannian metric $\tilde{g}$, and defining $g(x,y):=\frac{1}{2}\big(\tilde{g}(x,y)+ \tilde{g}(Jx,Jy)\big)$.

To an almost Hermitian manifold $(M,J,g)$ one can associate a fundamental two form (c.f. \cite{KsNk96b}) $\omega \in \Gamma\big(\Lambda^2 TM \big)$ given by $\omega(x,y):=g(Jx,y)$.  We call $\omega$ the \emph{almost symplectic form} associated to $(J,g)$, where the ``almost'' refers to the fact that in general $d\omega\neq 0$.  Indeed, $\omega$ is non-degenerate by definition, so if $\omega$ is closed then it is a symplectic form, and in such case $J$ is an $\omega$-compatible almost complex structure.  Again, we do not make this additional assumption.

Suppose $(M,J,g)$ is an almost Hermitian manifold of dimension $2n$ with associated almost symplectic form $\omega$, and suppose $L\subset M$ is a compact embedded submanifold of dimension $n$, which may be disconnected and may have boundary. Letting $TL\subset TM\big|_L$ denote the tangent sub-bundle of $L$, we will say that $L$ is \emph{totally real} provided $JTL \cap TL \equiv 0$. We will say $L$ is \emph{Lagrangian} provided $JTL = TL^\bot$ where $TL^\bot$ denotes the normal bundle of $L\subset M$; equivalently $L$ is Lagrangian if the restriction of $\omega$ to $L$ vanishes identically.  There is an obvious extension of this definition to Lagrangian immersions, however since the following analysis is all local in nature, it will be sufficient here to assume that $L\subset M$ is embedded.

We will often be interested in Lagrangians which are totally geodesic; in other words when each geodesic in $(L,g\big|_L)$ is also a geodesic in $(M,g)$, or equivalently when the second fundamental form of $L$ vanishes identically.  While such sub-manifolds are difficult to find in a given almost Hermitian manifold, we note that the metric only plays an auxiliary role in what follows, so the following result is particularly useful.

\begin{lemma}[Frauenfelder]\label{lem:Frauenfelder}
Let $(M,J)$ be an almost complex manifold, and let $L\subset M$ be an embedded totally real submanifold.  Then there exists a metric $g$ on $M$ for which $(M,J,g)$ is almost Hermitian, $L$ is Lagrangian, and $L$ is totally geodesic.
\end{lemma}
\begin{proof}
See \cite{Fu08} or Chapter 4.3 of \cite{MdSd04}.
\end{proof}

Continuing with our definitions, we shall say a smooth map $u:S\to M$ between smooth manifolds (which may have boundary and corners, be disconnected, or be non-compact) is a \emph{generally immersed} provided that for each point $z\in S$ for which $T_z u\neq 0$ we have $\Rank(T_z u)=\dim S$, and the set of critical points, which we henceforth denote as $\mathcal{Z}_u:=\{z\in S: T_z u =0\}$, has no accumulation points.  Furthermore if $M$ is equipped with a Riemannian metric $g$, then we require that the conformal structure $[u^*g]$ on $S$ admits a smooth extension across $\mathcal{Z}_u$. If $S$ has real dimension two, then we call $(u,S)$ a generally immersed \emph{curve}.

\begin{remark}Given a generally immersed curve $u:S\to M$ into an Riemannian manifold $(M,g)$, we note that $u^*g$ is not a Riemannian metric, however it does induce the following metric:
\begin{equation*}
\dist_{u^*g}(\zeta_0,\zeta_1):=\inf\big\{{\textstyle \int_0^1} \la \dot{\gamma}(t),\dot{\gamma}(t)\ra_{u^*g}^{\frac{1}{2}}dt: \gamma\in C^1\big([0,1],S\big)\text{ and } \gamma(i)=\zeta_i\big\}.
\end{equation*}
In the case that $\zeta_0$ and $\zeta_1$ lie in different connected components, our convention will be to define $\dist(\zeta_0,\zeta_1):=\infty$.
Consequently, we may regard $(S,\dist_{u^*g})$ as a metric space, in which case it can be equipped with Hausdorff measures $d\mathcal{H}^k$. Note that if $\mathcal{O}\subset S$ is an open set on which $u$ is an immersion, then  $d\mathcal{H}^{2}(\mathcal{O})=\Area_{u^*g}(\mathcal{O})$.  As such, our convention will be to simply define the area of an arbitrary open set $\mathcal{U}\subset S$ to be $\Area_{u^*g}(\mathcal{U}):=d\mathcal{H}^{2}(\mathcal{U})$.  Furthermore, throughout the remainder of this article, all integration in $S$ will be with respect to these Hausdorff measures unless otherwise specified, and for convenience we shall suppress the measure $d\mathcal{H}$ in subsequent computations.
\end{remark}

Given a smooth immersion $u:S\to M$ into a Riemannian manifold, one can associate several important bundles over $S$:
\begin{align*}
u^*TM &:= \{ (z, X)\in S\times TM : X\in T_{u(z)} M\}\\
\mathcal{T}&:=\{X\in u^*TM: \text{there exists } x\in T_z S \text{ such that } X=u_*x\}\\
\mathcal{N}&:=\{X\in u^*TM: \la X, Y\ra =0 \text{ for all } Y \in \mathcal{T}\}.
\end{align*}
In the case that $u:S\to M$ is only generally immersed, then $\mathcal{T}$ and $\mathcal{N}$ will be considered only as bundles over the regular points $S\setminus \mathcal{Z}_u$.  In this manner we may still regard $\mathcal{T}$ and $\mathcal{N}$ as smooth bundles.

\begin{remark}[An abuse of notation]\label{rem:AbuseOfNotation}
In what follows, we aim analyze $J$-curves not as holomorphically parametrized maps, but rather as sub-manifolds which satisfy an inhomogeneous mean curvature equation.  Unfortunately, it will frequently be the case that $J$-curves need be neither embedded nor immersed, so to proceed with our extrinsic analysis we will employ the following abuse of notation.  If $u:S\to M$ is generally immersed, then on the set of regular points $S\setminus \mathcal{Z}_u$ we identify $TS$ and $\mathcal{T}$ via $Tu$; furthermore if $\mathcal{E}\to M$ is a bundle and $\sigma\in \Gamma (\mathcal{E})$ is a section, we denote the pulled-back section $u^*\sigma$ simply as $\sigma$ and refer to it as the \emph{restriction} to $S$.  For example, if $f:M\to \R$ is a smooth function, we will often write the restriction as $f:S\to \R$ instead of $f\circ u$.
\end{remark}

Given an almost Hermitian manifold $(M,J,g)$, and a generally immersed curve $u:S\to M$, one can consider the non-linear Cauchy-Riemann operator, $\bar{\partial}_J$, defined by $\bar{\partial}_J u:= (JTu)^\bot \in \Gamma\big(\Hom(\mathcal{T},\mathcal{N})\big)$.  We shall call any generally immersed curve $(u,S)$ which satisfies $\bar{\partial}_J u\equiv 0$ a \emph{pseudo-holomorphic curve}, or simply a \emph{$J$-curve}.  We are primarily interested in $J$-curves in three similar set-ups which we enumerate as Type 0 maps, Type 1 maps, and Type 2 maps.  Roughly speaking, the ``type'' measures the number of Lagrangian boundary conditions of the map.  We now make this precise.

\begin{definition}[Type 0 maps]\label{def:Type0Map}
Given a compact manifold $M$, a Type 0 map $u:S\to M$ is a smooth and proper generally immersed map from a compact two-dimensional manifold $S$ (possibly with boundary) which satisfies $\partial S = u^{-1}(\partial M)$.
\end{definition}

\begin{definition}[Type 1 maps]\label{def:Type1Map}
Given a compact manifold $M$ of dimension $2n$ and a compact embedded submanifold $L\subset M$ of dimension $n$ for which $\partial L = L\cap \partial M$, we define a Type 1 map $u:S\to M$ to be a smooth and proper generally immersed map from a compact two-dimensional manifold $S$ with piece-wise smooth boundary for which $\partial S=\partial_0 S\cup \partial_1 S$ where the subsets $\partial_0 S, \partial_1 S\subset \partial S$ satisfy
\begin{enumerate}
\item $\partial_0 S = u^{-1}(\partial M)$;
\item $u(\partial_1 S)\subset L$;
\item $\partial_0 S\cap \partial_1 S$ is finite;
\item $\partial_0 S\cap \partial_1 S$ is the set of non-smooth boundary points of $S$ (i.e. corners).
\end{enumerate}
\end{definition}

\begin{definition}[Type 2 maps]\label{def:Type2Map}
Given a compact \emph{Riemannian} manifold $(M,g)$ of dimension $2n$, two compact embedded submanifolds $L_1,L_2\subset M$ each of dimension $n$ which intersect each other transversely and satisfy $\partial L_i = L_i\cap \partial M$, we define a Type 2 map $u:S\to M$ to be a smooth generally immersed map from a two-dimensional manifold $S$ with piece-wise smooth boundary for which $\partial S=\partial_0 S\cup \partial_1 S\cup \partial_2 S$, and the following hold.
\begin{enumerate}
\item $\Area_{u^*g} S <\infty$.
\item $L_1\cap L_2\subset M\setminus \partial M$.
\item For each compact set $\mathcal{K}\subset M\setminus (L_1\cap L_2)$, the set $u^{-1}(\mathcal{K})$ is compact.
\item $\partial_0 S = u^{-1}(\partial M)$;
\item $u(\partial_i S)\subset L_i$;
\item The set  $\cup_{i<j}\big(\partial_i S\cap \partial_j S\big)$ is finite;
\item The set  $\cup_{i<j}\big(\partial_i S\cap \partial_j S\big)$ is the set of non-smooth boundary points of $S$.
\end{enumerate}
\end{definition}

\begin{remark}
We note that maps $u:S\to M$ of Type 0 and Type 1 are both proper with compact domains, whereas Type 2 maps are only proper when the target is restricted to $M\setminus (L_1\cap L_2)$ and may have non-compact domains.  Indeed, unlike the former cases, there are some non-trivial subtleties regarding the behavior of $J$-curves of Type 2 near the intersection points $L_1\cap L_2$.  A standard manner to analyze this behavior is to regard $M\setminus (L_1\cap L_2)$ as a manifold with finitely many negative cylindrical ends.  Since we wish to allow for such analysis, we allow $J$-curves  of a Type 2 to have non-compact domains, rather than a priori demanding some sort of corner structure near the intersection of the Lagrangians. It should also be noted that finite area condition is standard in all reasonable applications of $J$-curves, and here it is essentially a technical convenience.
\end{remark}

\subsection{A mean curvature equation}\label{sec:MeanCurvatureEquation}
In this section, we derive an inhomogeneous mean curvature equation for $J$-curves.  To that end, we start by recalling some important metric properties of almost complex structures.

\begin{lemma}[Computational]\label{Lem:ComputationalAlmostHermition}
Let $(M,J,g)$ be an almost Hermitian manifold, and let $\nab$ denote the Levi-Civita connection associated to $g$.  Then the following hold.
\begin{align}
0&=(\nabla_XJ)J+J(\nab_X J)\label{eq:nabJAntiCommutesWithJ}\\
0&=\left<(\nab J)X, Y\right> + \left<X,(\nab J)Y\right>\label{eq:nabJAdjoint}\\
0&=\left<(\nab_Z J)X,\alpha X + \beta JX\right>\quad\forall \alpha,\beta\in\mathbb{R}\label{eq:nabJImageComplexOrthogToDomain}
\end{align}
\end{lemma}

\begin{proof}
To prove equation (\ref{eq:nabJAntiCommutesWithJ}), we covariantly differentiate $J^2=-\mathds{1}$. To prove equation (\ref{eq:nabJAdjoint}), we fix a vector field $Z$ and covariantly differentiate the equation $0=\left<JX,Y\right>+\left<X,JY\right>$ to find that
\begin{align*}
0&=\left<\nab_Z (JX),Y\right>+\left<JX,\nab_Z Y\right> +\left<\nab_Z X,JY\right> + \left<X,\nab_Z(JY)\right>\\
&=\left<(\nab_Z J) X,Y\right> +\left<J\nab_Z X,Y\right> -\left<X,J\nab_Z Y\right>\\
&\qquad-\left<J\nab_Z X, Y\right> +\left<X,(\nab_Z J)Y\right> + \left<X,J\nab_Z Y\right>\\
&=\left<(\nab_Z J)X,Y\right>+\left<X,(\nab_Z J)Y\right>.
\end{align*}
By (\ref{eq:nabJAdjoint}), it follows that $\la (\nab_Z J)X,X\ra = -\la X,(\nab_Z J) X\ra$; by (\ref{eq:nabJAntiCommutesWithJ}) and (\ref{eq:nabJAdjoint}) it follows that $\la (\nab_Z J)X, JX\ra = - \la J  X, (\nab_Z J)X\ra $; (\ref{eq:nabJImageComplexOrthogToDomain}) then follows immediately.

\end{proof}

Next, recall that given an immersion $u:S\to M$ into a Riemannian manifold $(M,g)$, one may associate second fundamental forms $A$ and $B$ by the following:
\begin{equation*}
A^\nu(x):= -(\nab_x \nu)^\top\qquad \text{and}\qquad B(x,y):= (\nab_x y)^\bot,
\end{equation*}
where $x,y\in \Gamma(\mathcal{T})$, $\nu\in \Gamma(\mathcal{N})$, $\nabla$ is the Levi-Civita connection associated to $g$, and $x\mapsto x^\top$ and $x\mapsto x^\bot$ are the orthogonal projections from $u^*TM$ to $\mathcal{T}$ and $\mathcal{N}$ respectively.  Using fact that $\nab$ is torsion free, one easily verifies that $B(x,y)=B(y,x)$. Furthermore, defining the bundle $\mathcal{S}\to S$ of symmetric linear transformations of $\mathcal{T}$ by
\begin{equation}\label{eq:BundleS}
\mathcal{S}:=\{s\in \Hom (\mathcal{T},\mathcal{T}): \la s x, y \ra = \la x, sy\ra\text{ for all } x,y\in \mathcal{T}\},
\end{equation}
one can use the the Leibniz rule and the fact that $\nab$ also preserves the metric $g$ to show that $A$ and $B$ are sections of $\Gamma\big(\Hom(\mathcal{N},\mathcal{S})\big)$ and $\Gamma\big(\Hom(\mathcal{T}\otimes\mathcal{T},\mathcal{N})\big)$ respectively (in particular they are tensors), and they also are related by the following:
\begin{equation}\label{eq:secondFundamentalFormsRelation}
\la A^w(x),y\ra = \la w,B(x,y)\ra.
\end{equation}
As an immediate consequence, on finds
\begin{equation}\label{eq:NormAisNormB}
\|A\| = \|B\|.
\end{equation}
\begin{definition}[mean curvature $H$]\label{def:meanCurvature}
For immersion $u:S\to M$, let $\{e_i\}_{i=1}^{\dim S}$ be an orthonormal frame for $\mathcal{T}_\zeta$.  Define $H_\nu(\zeta)\in \mathcal{N}_\zeta$ by
\begin{equation*}
H_\nu(\zeta):=\tr_S B_\zeta = \sum_{i=1}^{\dim S} B_\zeta(e_i,e_i).
\end{equation*}
It is straight-forward to show that $H$ does not depend on the choices of $\{e_i\}$, and thus we define the mean curvature vector field $H_\nu\in\Gamma(\mathcal{N})$ by $H_\nu:= \tr_S B$.

\end{definition}

We are now ready to derive a mean curvature equation for $J$-curves.
\begin{lemma}\label{lem:InhomogeneousMeanCurvatureEq}
Let $(M,J,g)$ be an almost Hermitian manifold, and define $Q$, a $(1,2)$-tensor on $M$, by
\begin{equation}\label{eq:DefQ}
Q(X,Y):= J(\nabla_{X}J)Y.
\end{equation}
Then immersed $J$-curves satisfy the following inhomogeneous mean curvature equation:
\begin{equation*}
H_\nu=\tr_S Q:=Q(e_1,e_1) + Q(e_2,e_2),
\end{equation*}
where $\{e_1,e_2\}$ is any locally defined orthonormal frame field of $\mathcal{T}$.
\end{lemma}
\begin{proof}
Without loss of generality, we assume $e_2 = J e_1$, and for clarity we define $e:=e_1$.  Then we compute
\begin{align*}
H_\nu&:= \big(\nabla_{e_1} e_1 + \nabla_{e_2} {e_2}\big)^\bot\; = \;\big(\nabla_{e} e + \nabla_{Je} Je\big)^\bot \;=\; \big(\nabla_e e + (\nabla_{Je} J )e + J\nabla_{Je} e \big) ^\bot\\
&= \big(\nabla_e e + (\nabla_{Je} J )e + J\nabla_{e} (Je) + J[e,Je] \big) ^\bot\\
&= \big(\nabla_e e + (\nabla_{Je} J )e + J(\nabla_{e} J)e - \nabla_e e\big) ^\bot\;\;= \;\;\big( (\nabla_{Je} J )e + J(\nabla_{e} J)e\big) ^\bot\\
&=(\nabla_{Je} J )e + J(\nabla_{e} J)e = \; J(\nabla_{Je} J )Je + J(\nabla_{e} J)e\\
&=\tr_S Q,
\end{align*}
where on the second to last line we have made use of both  (\ref{eq:nabJAntiCommutesWithJ}) and (\ref{eq:nabJImageComplexOrthogToDomain}).
\end{proof}

Recall that if the almost symplectic form $\omega$ is closed, then $\nab_{JX}J = -J\nab_X J$ (c.f. \cite{MdSd04} Appendix C.7) in which case $J$-curves have vanishing mean curvature; i.e. they are minimal surfaces.  Furthermore, it is straight-forward to show that $\|Q\|\leq \sup_M \|\nab J\|$, so that $J$-curves have a priori bounded mean curvature in any compact almost Hermitian manifold.

\section{Monotonicity}\label{sec:Monotonicity}
The purpose of this section is to establish extrinsic monotonicity results for non-negative functions $f:S\to \R$ which satisfy $\Delta_S f\geq - \lambda b^2 f$ on $J$-curves $u:S\to M$.  The approach below follows the basic argument for minimal surfaces in three-dimensional Riemannian manifolds given in \cite{CtMw99}, but with necessary generalizations for $J$-curves in almost Hermitian manifolds of arbitrary (even) codimension. The interior case is established in Section \ref{sec:Monotonicity0}, and Lagrangian boundary case is established in Section \ref{sec:Monotonicty1and2}.  Before deriving these results however, we recall the definition and properties of the Laplace-Beltrami operator $\Delta$ in Section \ref{sec:LaplaceBeltrami}.  The most important result in that section is Lemma \ref{lem:LaplacianEstimate}, in which establishes a useful estimate for $\Delta \beta^2$ where $\beta:\mathcal{O}\subset M \to \R$ is the distance to a point; that is $\beta(q):=\dist(p,q)$ for some fixed $p\in M$.

\subsection{The Laplace-Beltrami operator}\label{sec:LaplaceBeltrami}
Here we recall $\Delta_M$, and provide some elementary estimates which will be useful later on.  Indeed, recall that for an $m$-dimensional Riemannian manifold $M$, one can define the Laplace-Beltrami operator $\Delta_M$ by the following:
\begin{equation*}
\Delta_M f:= \Div_M \nab_M f,
\end{equation*}
where the gradient and divergence are given respectively by
\begin{equation*}
\nab_M f:= \sum_{i=1}^m(\nab_{e_i} f)e_i\qquad\text{and}\qquad\Div_M X := \sum_{i=1}^m\la\nab_{e_i} X,e_i\ra;
\end{equation*}
here $e_1,\ldots,e_m$ is any locally defined orthonormal frame field of $TM$.  We take a moment to point out a subtlety in our notation, namely that $\nab_M$ or $\nab_S$ denote gradients on Riemannian manifolds $M$ or $S$, whereas $\nab$ denotes a covariant derivative with respect to a Levi-Civita connection.

Given an immersion $u:S\to M$ with $\dim S = n$, and a function $f:M\to \mathbb{R}$, one can compute the $S$-gradient of the restriction\footnote{Recall our abuse of notation discussed in Remark \ref{rem:AbuseOfNotation}.} of $f$ to $S$:
\begin{equation*}
\nab_S f = {\textstyle \sum_{i=1}^n}(\nab_{e_i} f)e_i = (\nab_M f)^\top;
\end{equation*}
Here $e_1,\ldots,e_n$ is an orthonormal frame of the tangent plane $\mathcal{T}$. Similarly, for a vector field $X\subset \Gamma(TM)$ along $M$, one can also compute the divergence of $X$ along a submanifold $\Sigma\subset M$ by the following:
\begin{align*}
\Div_\Sigma(X)&:={\textstyle\sum_{i=1}^n} \la \nab_{e_i} X, e_i\ra\\
&={\textstyle \sum_{i=1}^n} \la \nab_{e_i} X^\top, e_i\ra + {\textstyle \sum_{i=1}^n} \la \nab_{e_i} X^\bot, e_i\ra\\
&=\Div_S(X^\top) - \la X^\bot, {\textstyle \sum_{i=1}^n} \nab_{e_i}e_i\ra\\
&=\Div_S(X^\top) - \la X^\bot, H_\nu\ra.
\end{align*}
Given a function $f:M\to \mathbb{R}$, and a smooth immersion $u:S\to M$, it will be convenient to have a formula for $\Delta_S f$, where we abuse notation as above by letting $f$ denote the restriction $f\circ u: S\to\mathbb{R}$.  To that end, we fix a point $\zeta\in S$, and let $e_1, \ldots,e_n$ be a locally defined orthonormal frame field of $\mathcal{T}$ along $S$ for which $\nabla_{e_i} e_j\big|_{\zeta}=(\nabla_{e_i}e_j)^\bot\big|_{\zeta}$ for $i,j\in \{1,\ldots, n\}$. We then extend $e_1,\ldots,e_n$ to a locally defined orthonormal frame field $e_1,\ldots,e_m$ of $TM$ in a neighborhood of $u(\zeta)$.  Letting $\Sigma$ be the image of $S$ by $u$, we then compute as follows.
\begin{align*}
\Delta_S f &= \Div_S \nab_S f\\
&= \Div_S(\nab_M f^\top)\\
&=\Div_{\Sigma} \nab_M f + \la \nab_M f ^\bot, H_\nu\ra\\
&=\la \nab_M f ^\bot, H_\nu \ra + {\textstyle \sum_{j=1}^n \sum_{i=1}^m} \la \nab_{e_j}\big((\nab_{e_i}f)e_i\big), e_j\ra\\
&=\la \nab_M f ^\bot, H_\nu \ra + {\textstyle\sum_{i=1}^n}  \nab_{e_i}(\nab_{e_i}f) +{\textstyle \sum_{j=1}^n \sum_{i=1}^m} (\nab_{e_i}f)\la \nab_{e_j}e_i, e_j\ra.
\end{align*}
Evaluating at $\zeta\in S$ we find that
\begin{align*}
(\Delta_S f)(\zeta)&={\textstyle\sum_{i=1}^n}  \nab_{e_i}(\nab_{e_i}f)\\
&=\tr_S \Hess_M^f+{\textstyle \sum_{i=1}^n} \nabla_{(\nabla_{e_j e_j})^\bot}f\\
&=\tr_S \Hess_M^f + \la \nabla_M f, H_\nu\ra;
\end{align*}
here $\Hess_M^f$ denotes the Hessian of $f:M\to \R$ and is given by
\begin{align*}
\Hess_M^f(X,Y):&= X(Y(f))-(\nab_X Y)f\\
&=\nab_X(\nab_Y f) - \nab_{\nab_X Y}f.
\end{align*}
Consequently, for any immersed $J$-curve $u:S\to M$ with image $\Sigma$ and vector field $X\subset \Gamma(TM)$, we have
\begin{align}
\Div_\Sigma (X) &= \Div_S(X^\top) - \la X^\bot, \tr_S Q\ra\label{eq:JCurveDivergence}\\
\Delta_S f &= \tr_S \Hess_M^f + \la \nabla_M f, \tr_S  Q\ra.\label{eq:JCurveLaplacian}
\end{align}

Next, we estimate $\Delta_S \beta^2$ along a $J$-curve where $\beta:M\to \R$ is the distance from a specified point.  To that end, it will be convenient to define the injectivity radius of $M$ at the point $p$ by $\inj(p)$, and to denote the sectional curvature of a Riemannian manifold $(M,g)$ at the point $p$ as
\begin{equation*}
K_{sec}(X,Y):=\frac{\la R(X,Y)Y,X\ra}{\|X\|^2\|Y\|^2 - \la X, Y\ra^2},
\end{equation*}
where $R$ is the Riemann curvature tensor given by
\begin{equation}\label{eq:RiemannCurvatureTensor}
R(X,Y)Z:=\nab_X \nab_Y Z -\nab_Y \nab_X Z - \nab_{\nab_X Y} Z.
\end{equation}
Given $p\in M$, it will be convenient to define the continuous function $p\mapsto |K_{sec}(p)|$ by the following
\begin{equation}\label{eq:DefOfAbsKsec}
|K_{sec}(p)|:=\sup \{|K_{sec}(X,Y)| : X,Y\in T_p M\text{ and } X\wedge Y\neq 0 \}.
\end{equation}
We now provide the desired estimate.

\begin{lemma}\label{lem:LaplacianEstimate}
Fix a constant $C>0$, and let $(M,J,g)$ be a  compact almost Hermitian manifold, possibly with boundary, for which
\begin{equation*}
\sup_{p\in M} |K_{sec}(p)| \leq {\textstyle\frac{1}{4}}C^2\qquad\text{and}\qquad \sup_{\substack{e\in TM\\\|e\|=1}}\|Q(e,e)+Q(Je,Je) \| \leq {\textstyle\frac{1}{2}}C;
\end{equation*}
Fix $p\in M$ and define the function $\beta:M\to\R$ by $\beta(q)=\dist(p,q)$.  Then for any $J$-curve $u:S\to \mathcal{B}_r(p)$ where
$r\leq {\textstyle \frac{1}{2}}\min\big(\inj(p), C^{-1}\big)$, the following point-wise inequality holds:
\begin{equation}\label{eq:LaplaceEstimate}
\big|\Delta_S \beta^2 -4\big| \leq 4 C\beta;
\end{equation}
here we have abused notation in inequality (\ref{eq:LaplaceEstimate}) by letting $\beta$ denote the restriction, $\beta\circ u$, to the given $J$-curve.
\end{lemma}

\begin{proof}
First observe that for any function $f:S\to\R$, it follows from the definition of $\Delta_S$ that $\Delta_S f^2 = 2 f \Delta_S f + 2\|\nab_S f\|^2$. Thus making use of (\ref{eq:JCurveLaplacian}), we find
\begin{align*}
\big|\Delta_S \beta^2 - 4\big|&=\big| 2\|\nab_S \beta\|^2 + 2 \beta\Delta_S \beta -4 \big|\\
&=\big| 2\|\nab_S \beta\|^2 + 2 \beta \tr_S \Hess_M^\beta +2\beta\la \nab_M \beta,\tr_S Q\ra -4 \big|.
\end{align*}
To further massage this equation, we observe that $\|\nab_M \beta\|\equiv 1$, and thus we may define the orthogonal projection
\begin{equation*}
\pi_\beta^\bot: TM\to TM\qquad\text{by}\qquad \pi_\beta^\bot (X) := X - \la X, \nab_M \beta\ra\nab_M\beta,
\end{equation*}
so that
\begin{equation*}
{\textstyle \sum_{i=1}^2} \|\pi_\beta^\bot  (e_i)\|^2 = {\textstyle \sum_{i=1}^2} \big(1-\la e_i,\nab_M \beta\ra^2\big)=2-\|\nab_S \beta\|^2.
\end{equation*}
Consequently,
\begin{align*}
|\Delta_S \beta^2 - 4|&=\big|2\beta \la \nab_M \beta, \tr_S Q\ra + 2\big(\beta\tr_S \Hess_M^\beta -  {\textstyle \sum_{i=1}^2}
\| \pi_\beta^\bot(e_i)\|^2\big) \big|\\
&\leq 2|\beta|\; \|\tr_S Q\| + 2{\textstyle \sum_{i=1}^2}\big|\beta\Hess_M^\beta(e_i,e_i)-\|\pi_\beta^\bot(e_i)\|^2\big|\\
&\leq \beta C + 2{\textstyle \sum_{i=1}^2}\big|\beta\Hess_M^\beta(e_i,e_i)-\|\pi_\beta^\bot(e_i)\|^2\big|.
\end{align*}
The proof of Lemma \ref{lem:LaplacianEstimate} then follows immediately from Lemma \ref{lem:HessianEstimate} below.
\end{proof}

\begin{lemma}\label{lem:HessianEstimate}
Let $(M,g)$ be a Riemannian manifold, $p\in M$, and let $\beta$, $C$, and $r$ be defined as in Lemma \ref{lem:LaplacianEstimate}. Then for each unit vector $e\in T_q M$ with $\beta(q)\leq r$, we have
\begin{equation}\label{eq:HessianEstimate}
\big|\beta \Hess_M^\beta(e,e)-\|\pi_\beta^\bot(e) \|^2 \big| \leq \beta C/2
\end{equation}
\end{lemma}
\begin{proof}
We begin by noting that $\Hess_M^\beta(\nab_M \beta, \cdot )\equiv 0$, so it is sufficient to prove the above result for the $e\in T_qM$ which are perpendicular to $\nab_M \beta$. Next we recall that the Hessian comparison theorem (c.f. Theorem 1.1 in \cite{SrYst94}) guarantees that if for $i=1,2$, the $(M_i,g_i)$ are $m$-dimensional Riemannian manifolds, and $p_i\in M_i$, and $\beta_i(q_i):=\dist_{g_i}(q_i,p_i)$, and
\begin{equation*}
\inf_{q_1\in M_1}K_{sec}^{g_1}(q_1)\geq \sup_{q_2 \in M_2} K_{sec}^{g_2}(q_2),
\end{equation*}
then for any unit vectors $e_i\in T_{q_i} M_i$ with $\beta_1(q_1)=\beta_2(q_2)< \min\big(\inj(p_1),\inj(p_2)\big)$ and $\la X_i, \nab_{M_i} \beta_i\ra=0$, the following holds:
\begin{equation*}
\Hess_{M_1}^{\beta_1}(e_1,e_1) \leq \Hess_{M_2}^{\beta_2}(e_2,e_2).
\end{equation*}
Consequently
\begin{equation*}
\Hess_{M_1}^{\beta_1}(e_1,e_1)  \leq \Hess_M^\beta(e,e)  \leq \Hess_{M_2}^{\beta_2}(e_2,e_2) ,
\end{equation*}
where $(M_i,g_i)$ are $m$-dimensional Riemannian manifolds with constant sectional curvatures $K_{sec}^{g_1}\equiv C^2/4$ and $K_{sec}^{g_2}\equiv -C^2/4$ (i.e. space forms of curvature $\pm C^2/4$ respectively), and $\beta_i$ and $e_i$ are as in the hypotheses of the Hessian comparison theorem.  Thus to finish the proof of Lemma \ref{lem:HessianEstimate}, it is sufficient to prove
\begin{equation}\label{eq:ConstSecCurveHessianEstimates}
-\beta_1 C/2\leq\beta_1 \Hess_{M_1}^{\beta_1}(e_1,e_1) - 1 \qquad\text{and}\qquad\beta_2 \Hess_{M_2}^{\beta_2}(e_2,e_2) - 1 \leq \beta_2 C/2.
\end{equation}
To that end, recall (c.f. \cite{SrYst94}, Chapter 1) that for any Riemannian manifold $(M,g)$ with $p,q\in M$, $\dist(p,q) < \inj(p)$, $\gamma:[0,s]\to M$ a unit speed geodesic with $\gamma(0)=p$ and $\gamma(s)=q$, $X\in T_q M$ and $\la X, \dot{\gamma}\ra =0 $ we have
\begin{equation}\label{eq:HessianAsIntegral}
\Hess_M^\beta(X,X)=\int_0^s \big( \|\nab_{\dot{\gamma}} \tilde{X}\|^2-\la R(\tilde{X},\dot{\gamma})\dot\gamma,\tilde{X}\ra \big) dt,
\end{equation}
where $R$ is defined as in (\ref{eq:RiemannCurvatureTensor}) and $\tilde{X}$ is a Jacobi field (i.e. a solution to the following differential equation)
\begin{equation}\label{eq:JacobiFields}
\nab_{\dot{\gamma}}\nab_{\dot{\gamma}} \tilde{X} + R(\tilde{X},\dot{\gamma})\dot{\gamma} = 0\qquad\text{and}\qquad \tilde{X}(p) = 0\;\;,\;\;\tilde{X}(q)=X.
\end{equation}
Next, recall (c.f. \cite{KsNk96a}) that for a space form $(M_\kappa, g_\kappa)$ of curvature $\kappa$, we have
$R(X,Y)Z=\kappa\big( \la Y,Z\ra X - \la Z, X\ra Y\big)$,
in which case solutions to the Jacobi equation (\ref{eq:JacobiFields}) have the form $\tilde{X}_{\gamma(t)}= f(t) X_{\gamma(t)}$, where $\nab_{\dot{\gamma}} X_{\gamma(\cdot)}\equiv 0$, $X_{\gamma(s)}=X$, $f''(t)+\kappa f(t)=0$, $f(0)=0$, and $f(s)=1$.  Observe that
\begin{equation*}
f(t) = \frac{\sin(\sqrt{\kappa} t)}{\sin(\sqrt{\kappa} s)}\;\;\text{if}\;\; \kappa>0\qquad\text{and}\qquad f(t)=\frac{\sinh(\sqrt{|\kappa|} t)}{\sinh(\sqrt{|\kappa|} s)}\;\;\text{if}\;\; \kappa<0.
\end{equation*}
Combining this with (\ref{eq:HessianAsIntegral}) for the space forms $(M_i,g_i)$ with curvatures $\pm C^2/4$ as defined above, we have
\begin{equation*}
\Hess_{M_1}^{\beta_1}(e_1,e_1)= {\textstyle \frac{1}{2}}C \cot (C \beta_1/2)\qquad\text{and}\qquad \Hess_{M_2}^{\beta_2}(e_2,e_2)= {\textstyle \frac{1}{2}}C \coth(C \beta_2/2),
\end{equation*}
and thus for $C\beta_i \leq 1$ we find
\begin{equation*}
\beta_2\Hess_{M_2}^{\beta_2}(e_2,e_2) - 1 \leq {\textstyle \frac{1}{2}}C \beta_2 \coth(C\beta_2/2) -1 \leq \cosh(C \beta_2/2) - 1 \leq {\textstyle \frac{1}{2}}C\beta_2,
\end{equation*}
and
\begin{equation*}
\beta_1 \Hess_{M_1}^{\beta_1}(e_1,e_1) - 1 \geq {\textstyle \frac{1}{2}}C \beta_1 \cot(C \beta_1/2) -1 \geq -{\textstyle \frac{1}{2}}C\beta_1.
\end{equation*}
Thus we have proved (\ref{eq:ConstSecCurveHessianEstimates}), which also completes the proof of Lemma \ref{lem:HessianEstimate}.

\end{proof}

\subsection{Interior case}\label{sec:Monotonicity0}
Here we prove the aforementioned monotonicity results for $J$-curves of Type 0.  We begin with a computational lemma, which will be convenient later.
\begin{lemma}\label{lem:IntegrationByParts0}
Let $(M,J,g)$ be a compact almost Hermitian manifold of dimension $2n$, and possibly with boundary.  Let $u:S\to M$ be a $J$-curve of Type 0, fix $\zeta_0\in S\setminus \partial S$, let $f\in C^2(S, \R)$,  let $\beta:S\to \R$ be defined by
\begin{equation}\label{eq:DefOfBeta}
\beta(\zeta):=\dist_g\big(u(\zeta),u(\zeta_0)\big),
\end{equation}
and define the set
$S_r:=\beta^{-1}\big([0,r]\big)$.  Then for every regular value $r\leq {\textstyle \frac{1}{2}}\inj(p)$ of $\beta$, the following holds
\begin{equation*}
\int_{S_r}\!\!f\Delta_S \beta^{2}
=\int_{S_r}\!\! (\beta^2-r^2)\Delta_S f
+2r\int_{\partial S_r}\!\!\!\!  f\| \nabla_S \beta\|.
\end{equation*}
\end{lemma}
\begin{proof}
Let $\nu\in \Gamma(\mathcal{T})\big|_{\partial S_r}$ be an outward pointing unit normal vector.  Then
\begin{align*}
\int_{S_r}\!\!f\Delta_S \beta^{2}&=\int_{S_r}\!\! \beta^2\Delta_S f-\int_{\partial S_{r}}\!\!\!\!\beta^{2}\la
\nabla_S f,\nu\ra
+\int_{\partial S_{r}}\!\!\!\! 2\beta f\la \nabla_S \beta,\nu \ra\\
&=\int_{S_r}\!\! \beta^2\Delta_S f-r^2\int_{ S_r}\!\!\!\!
\Delta_S f
+2r\int_{\partial S_r}\!\!\!\!  f\| \nabla_S \beta\|,
\end{align*}
where to obtain equality of the second terms on the right hand side we have made use of the fact that $\beta\big|_{\partial S_r}\equiv r$ so that $\beta^2$ pulls out of the integrand, and we then integrated by parts once more. Equality of the third terms follows since $\nab_S \beta\big|_{\partial S_r} = h \nu$ for some positive function $h:\partial S_r \to \R$.  The desired result is then immediate.
\end{proof}

\begin{theorem}[Monotonicity - Type 0]\label{thm:MonotonicityType0}
Fix a constant $C>0$, and let $(M,J,g)$ be a  compact almost Hermitian manifold, possibly with boundary, for which
\begin{equation*}
\sup_{p\in M} |K_{sec}(p)| \leq {\textstyle\frac{1}{4}}C^2\qquad\text{and}\qquad \sup_{\substack{e\in TM\\\|e\|=1}}\|Q(e,e)+Q(Je,Je) \| \leq {\textstyle\frac{1}{2}}C;
\end{equation*}
here $|K_{sec}(p)|$ is defined as in (\ref{eq:DefOfAbsKsec}), and $Q$ is the tensor defined in Lemma \ref{lem:InhomogeneousMeanCurvatureEq}.  Let $u:S\to M$ be a $J$-curve of Type 0, fix $\zeta_0\in S\setminus \partial S$, define $\beta$ as above so that $\beta(\zeta_0)=0$, fix positive constants
$\lambda$ and $b$ so that $b<\frac{1}{2}\min\big(\inj\big(u(\zeta_0)\big),C^{-1}\big)$, and let $f\in C^2(S,\R)$  satisfy $f\geq 0$ and $\Delta_S f \geq -\lambda b^2 f$. Then for any $a\in [0,b]$ the following holds
\begin{equation*}
e^{\frac{\lambda a}{2b}+2 Ca}a^{-2 }\int_{S_a}f\leq e^{\frac{\lambda}{2}+2 Cb}b^{-2
}\int_{S_b}f,
\end{equation*}
where $S_r := \beta^{-1}\big([0,r]\big)$.
\end{theorem}
\begin{proof}
We begin by defining the functions
\begin{equation*}
F(\tau)=\int_{S_\tau}f\qquad \text{and}\qquad G(\tau)=e^{\frac{\lambda \tau}{2b}} e^{2C\tau}\tau^{-2}.
\end{equation*}
Our goal will be to show that the map $\tau\mapsto G(\tau)F(\tau)$ is monotone increasing on $(0,b)$, and we accomplish this in two steps.  The first step is to show that $(GF)'\geq 0$ on its set of regular values (which is open and dense), and the second step is to show that this is sufficient to conclude that $GF$ is monotone increasing.  Recall the smooth coarea formula  states that if $f:S\to \R$ is an integrable function, and $\beta:S\to \R$ is as above and has no critical values in the interval $[a,b]$, then
\begin{equation}\label{eq:coareaFormula0}
\int_{S_b\setminus S_a} f = \int_a^b \Big(\int_{\partial S_\tau}f  \|\nab_{S} \beta\|^{-1}\Big) d\tau,
\end{equation}
where the integrals over $S_b\setminus S_a$ and $\partial S_\tau$ are respectively taken with respect to the Hausdorff measures associated to the metric $u^*g$.  Note that proof of (\ref{eq:coareaFormula0}) follows from an elementary change of variables formula.  Also note that as an immediate consequence of (\ref{eq:coareaFormula0}), we have
\begin{equation}\label{eq:coareaConsequence0}
\frac{d}{d\tau}\Big|_{\tau=\tau_0} \int_{S_\tau} f = \int_{\partial S_{\tau_0}}\frac{f}{\|\nab_S \beta\|}
\end{equation}
for every regular value $\tau_0$ of $\beta$.

We are now ready to show that $(GF)'\geq 0$ on the set of regular values $\mathcal{I}_{reg}$ of $\beta$.  Indeed, for $\tau\in \mathcal{I}_{reg}$ we compute as follows:

{\allowdisplaybreaks
\begin{align*}
\big(e^{-\frac{\lambda \tau}{2b}}e^{-2 C\tau}\tau^{3}\big)(GF)'(\tau)
&= e^{-\frac{\lambda\tau}{2b}}e^{-2 C\tau}\tau^{3}\frac{d}{d\tau}\left( e^{\frac{\lambda \tau}{2b}}e^{2
C\tau}\tau^{-2 }\int_{S_\tau}f\right)\notag\\
&=\frac{\lambda \tau}{2b}\int_{S_\tau}f+\frac{1}{2}\int_{S_\tau} \big(4C\tau +\Delta_S \beta^2-4\big)f \\ &\qquad+\tau\int_{\partial S_\tau}f\frac{1}{\|
\nabla_S
\beta\|}
-\frac{1}{2}
\int_{S_\tau}f\Delta_{S }\beta^{2}\notag\\
&=\frac{\lambda \tau}{2b}\int_{S_\tau}f+\frac{1}{2}\int_{S_\tau}\left( 4 C
\tau+\Delta_{S }\beta^{2}-4 \right) f\\
&\qquad\qquad +\frac{1}{2}\int_{S_\tau}\left(
\tau^{2}-\beta^{2}\right)
\Delta_{S }f
+\tau\int_{\partial S_\tau}f\frac{
\| \nabla_M \beta^\bot\| ^{2}}{\| \nabla_S
\beta\|}\notag\\
&\geq\frac{\lambda \tau}{2b}\int_{S_\tau}f+\frac{1}{2}\int_{S_\tau}\left(
\tau^{2}-\beta^{2}\right)
\Delta_{S}f\notag\\
&\geq\frac{\lambda \tau}{2b}\int_{S_\tau}f -\frac{1}{2}\int_{S_\tau}\big(\tau^2-\beta^2\big) \lambda b^{-2} f\\
&= \frac{\lambda\tau}{2 b}\int_{S_\tau}\Big(1-\frac{\tau}{b}+\frac{\beta^2}{b\tau}\Big)  f\\
& \geq 0,
\end{align*}}
where the second equality follows from differentiating and applying equation (\ref{eq:coareaConsequence0}); the third equality follows from Lemma \ref{lem:IntegrationByParts0} together with the fact that $1-\|\nab_S \beta \|^2 = \|\nab_M \beta^\bot \|^2$; finally, the first inequality is obtained by employing Lemma \ref{lem:LaplacianEstimate}. Consequently $(GF)'(\tau_0)\geq 0$ for every point $\tau_0\in \mathcal{I}_{reg}$.  Theorem \ref{thm:MonotonicityType0} now follows immediately from Lemma \ref{lem:AnalysisLemma1} below.
\end{proof}

\begin{lemma}\label{lem:AnalysisLemma1}
Consider two functions $f,g:[a,b]\to \R$ for which $f,g\geq0$, $g$ is continuously differentiable, and $f$ is monotone increasing.  Suppose further that the set of regular\footnote{In other words, $\mathcal{I}_{reg}$ is the set of points at which $f$ is differentiable.} points of $f$, $\mathcal{I}_{reg}\subset [a,b]$, is open, and that at every point $\tau_0\in \mathcal{I}_{reg}$ we have $(gf)'(\tau_0)\geq 0$.  Then $\tau\mapsto g(\tau)f(\tau)$ is monotone increasing.
\end{lemma}

Before providing the proof of Lemma \ref{lem:AnalysisLemma1}, let us first make a few remarks regarding Theorem \ref{thm:MonotonicityType0} and state a few immediate corollaries.  Indeed, the first observation is rather elementary in that the above result holds not just for embedded or immersed curves, but also for generally immersed curves. The second observation is that the the above result holds \emph{regardless of the topology of the $S_r$.} Indeed, for small $b$ the manifold  $S_b$ may be a disk, and for larger $b$ the manifold might be a cylinder or pair of pants or a much more complicated surface, and yet Theorem \ref{thm:MonotonicityType0} holds.  Still postponing the proof, let us deduce some well known corollaries from this single estimate.

\begin{corollary}\label{cor:AlternateSr}
Theorem \ref{thm:MonotonicityType0} also holds when the $S_r$ are alternatively defined to be the closure of the connected component of $\beta^{-1}\big([0,r)\big)$ which contains $\zeta_0$.
\end{corollary}
\begin{proof}
The proof is identical to that of the theorem: show the newly defined map $\tau\mapsto F(\tau)G(\tau)$ has non-negative derivative on its set of regular values, which are identical to those of the old map $\tau\mapsto F(\tau)G(\tau)$, then apply Lemma \ref{lem:AnalysisLemma1} to finish the proof.
\end{proof}

\begin{corollary}[Mean Value Inequality]\label{cor:MeanValueInequality0} Assume the same hypotheses as in Theorem \ref{thm:MonotonicityType0}. Then for each $\epsilon>0$ there exists $r_0< \frac{1}{2}\inj(p)$ which depends only on the constant $C$ with the property that whenever $0<b<r_0$, the following holds:
\begin{equation}\label{eq:MeanValueInequality}
\sum_{\zeta\in u^{-1}\big(u(\zeta_0)\big)}f(\zeta)\mu(\zeta)\leq\frac{(1+\epsilon)e^{\frac{\lambda }{2}}}{\pi
b^{2 }}\int_{S_b}f,
\end{equation}
where $\mu(\zeta):=\lim_{a\to 0} \frac{1}{\pi a^2} \Area\big(S_a(\zeta)\big)$ and $S_a(\zeta)$ is the connected component of $u^{-1}\big(\mathcal{B}_a(u(\zeta))\big)$ which contains $\zeta$.  In particular, if $\zeta_0 = u^{-1}\big(u(\zeta_0)\big)$ and $\zeta_0$ is not a critical point of $u$, then
\begin{equation*}
f(\zeta_0)\leq \frac{(1+\epsilon)e^{\frac{\lambda }{2}}}{\pi
b^{2 }}\int_{S_b}f.
\end{equation*}
\end{corollary}

\begin{proof}
Observe that for each $\epsilon>0$ there exists $r_0>0$ depending only on $C$ such that $e^{2Cb}<1+\epsilon$ for all $0<b<r_0$. Define $r_0$ as such, and then let $a\to 0$ in Corollary \ref{cor:AlternateSr}.
\end{proof}

Note in the above, the mass $\mu(\zeta)$ can alternatively be defined as the unique $k\in \mathbb{N}$ such that there exist coordinates centered at $\zeta\in S$ and $u(\zeta)\in M$ so that in these coordinates $u$ has the form
\begin{equation*}
u(z)=(z^k,0,\ldots,0) + O(|z|^{k+1}).
\end{equation*}
The existence of such a $k$ is well known (c.f. \cite{MdSd04}  or \cite{MmWb94}).

\begin{corollary}[Monotonicity of Area]  \label{cor:MonotonicityOfArea0}
Assume the same hypotheses as in Theorem \ref{thm:MonotonicityType0}. Then for each $\epsilon>0$ there exists $r_0< \frac{1}{2}\inj(p)$ which depends only on the constant $C$ with the property that whenever $0<b<r_0$, the following holds:
\begin{equation*}
\frac{ 1}{a^{2}}\Area_{u^*g}(S_a(\zeta_0))\leq \frac{1+\epsilon}{b^{2 }}\Area_{u^*g}\big(S_b(\zeta_0)\big).
\end{equation*}
Or more commonly,
\begin{equation*}
(1+\epsilon)^{-1}\pi b^2 \leq \Area_{u^*g}\big(S_b(\zeta_0)\big).
\end{equation*}
\end{corollary}

\begin{proof}
Let $f=1$ and observe that $\Delta_S f = 0 \geq -\lambda b^2 f$ for all $\lambda>0$.  Apply Theorem \ref{thm:MonotonicityType0}, and let $\lambda\to 0$.The result is immediate.
\end{proof}

We now complete the proof of Theorem \ref{thm:MonotonicityType0} by proving Lemma \ref{lem:AnalysisLemma1}.

\begin{proof}[Proof of Lemma \ref{lem:AnalysisLemma1}]
We begin by recalling that any function $h:[a,b]\to\mathbb{R}$ of bounded variation may be written as $h=h_{cont}+h_{sing}$, where $h_{cont}$ is absolutely continuous, and $h_{sing}$ is differentiable almost everywhere and satisfies
\begin{equation}\label{eq:AnalysisLem1}
h_{sing}'(\tau_0)= 0
\end{equation}
for every regular point $\tau_0$. In particular, an explicit decomposition can be written as
\begin{equation}\label{eq:AnalysisLem3}
h_{cont}(\tau)=\int_a^\tau h'(\sigma)d\sigma\qquad\text{and}\qquad h_{sing}=h-h_{cont}.
\end{equation}
We furthermore recall (c.f. Chapter 5 in \cite{Rhl88}) that if $h$ is monotone increasing then $h(t)-h(s)\geq \int_s^t h'(\sigma)d\sigma,$
from which it immediately follows that both $h_{cont}$ and $h_{sing}$ are monotone increasing.

We now recall that since $g$ is continuously differentiable and $f$ is monotone, they are each of bounded variation, and so too is their product.  Thus we may write
\begin{equation}\label{eq:AnalysisLem2}
gf = (gf)_{cont} + (gf_{cont})_{sing} + (gf_{sing})_{sing},
\end{equation}
and observe that it is sufficient to prove that each of the three terms on the right hand side of (\ref{eq:AnalysisLem2}) is monotone increasing.  To that end, we first consider the first term $(gf)_{cont}$.  Recall that $f$ is differentiable on the set of full measure $\mathcal{I}_{reg}$, and $g$ is continuously differentiable by assumption, so $gf$ is also differentiable on $\mathcal{I}_{reg}$.  Furthermore by our definition of $(gf)_{cont}$ it follows that for all $\tau \in \mathcal{I}_{reg}$ we have $(gf)_{cont}'(\tau)=(gf)'(\tau)$ which is non-negative by assumption.  Consequently
\begin{equation*}
(gf)_{cont}(t)- (gf)_{cont}(s)= \int_s^t \big((gf)_{cont}\big)'(\sigma) d\sigma =  \int_s^t (gf)'(\sigma) d\sigma \geq 0,
\end{equation*}
which proves that $(gf)_{cont}$ is monotone increasing.

Next, we consider the second term $(gf_{cont})_{sing}$. We again note that $g$ is continuously differentiable by assumption, and $f_{cont}$ is absolutely continuous by definition, so $gf_{cont}$ is also absolutely continuous, and hence $(gf_{cont})_{sing}\equiv 0$ by (\ref{eq:AnalysisLem3}). Consequently $(gf_{cont})_{sing}$ is trivially monotonic.

Thus to finish the proof of Lemma \ref{lem:AnalysisLemma1}, it is now sufficient to prove that $(gf_{sing})_{sing}$ is monotone increasing.  For clarity, we will write $\tilde{f}:=f_{sing}$, and we fix $s,t$ such that $a\leq s< t\leq b$. Then we have
\begin{align}
(g\tilde{f})_{sing}(t) - (g\tilde{f})_{sing}(s) &= (g\tilde{f})(t) - (g\tilde{f})(s) - \int_s^t (g\tilde{f})'(\sigma)d\sigma\label{eq:AnalysisLem4}\\
&=(g\tilde{f})(t) - (g\tilde{f})(s) - \int_s^t (g'\tilde{f})(\sigma)d\sigma,\notag
\end{align}
where we have made use of the fact that $\tilde{f}' = f_{sing}' =0 $ on $\mathcal{I}_{reg}$ -- a set of full measure.  Furthermore, $\mathcal{I}_{reg}$ is open by assumption, so there exist a countable collection of disjoint non-empty open intervals $\mathcal{I}_k$ for which $\mathcal{I}_{reg}\cap (s,t) = \cup_{k\in\mathbb{N}} \mathcal{I}_k$, and $\tilde{f}\big|_{\mathcal{I}_k}=\const$.  For each $k\in \mathbb{N}$, define
\begin{equation*}
x_k^+:=\sup \mathcal{I}_k\qquad x_k^-:=\inf \mathcal{I}_k\qquad x_k^0:= {\textstyle \frac{1}{2}}(x_k^+ + x_k^-)\in \mathcal{I}_k.
\end{equation*}
Then by construction
\begin{equation*}
\tilde{f}\big|_{(s,t)} = \sum_{k\in\mathbb{N}}\tilde{f}(x_k^0)\chi_{\mathcal{I}_k}\quad\text{almost everywhere},
\end{equation*}
where $\chi_{\mathcal{I}_k}$ is the characteristic function on the interval $\mathcal{I}_k$.  For each $n\in \mathbb{N}$, it will be convenient to define $x_{k;n}^{\pm,0}$ by $\{x_{1;n}^{\pm,0},\ldots,x_{n;n}^{\pm,0}\}:=\{x_1^{\pm,0},\ldots,x_n^{\pm,0}\} $ so that $x_{k-1;n}^{\pm,0} < x_{k;n}^{\pm,0}$ for each $k=2,\ldots, n$; similarly, define $\mathcal{I}_{k;n}$, and for notational convenience also define $x_{n+1;n}^-=t=x_{n+1;n}^0$ and $x_{0;n}^+=s=x_{0;n}^0$.  Next, observe that since $g$ is continuously differentiable, there exists a constant $M$ such that $\sup_{[a,b]}|g'|\leq M$, so by the dominated convergence theorem we have

{\allowdisplaybreaks
\begin{align*}
g(t)\tilde{f}(t)-g(s)\tilde{f}(s)&-\int_s^t
(g'\tilde{f})(\sigma)d\sigma\notag\\
&=g(t)\tilde{f}(t)-g(s)\tilde{f}(s)-\lim_{n\to\infty}\sum_{k=1}^n
\tilde{f}(x_{k;n}^0)\int_s^t g'(\sigma)\chi_{\mathcal{I}_{k;n}}(\sigma)d\sigma\notag\\
&=\lim_{n\to \infty}\sum_{k=1}^{n+1}\big(g(x_{k;n}^-)\tilde{f}(x_{k;n}^0)-g(x_{k-1;n}^+)\tilde{f}(x_{k-1;n}^0) \big)\\
&=\lim_{n\to\infty}\Big(\sum_{k=1}^{n+1} g(x_{k;n}^-)\big(\tilde{f}(x_{k;n}^0) -\tilde{f}(x_{k-1;n}^0) \big)\\
&\qquad\qquad\qquad+\sum_{k=1}^{n+1}\tilde{f}(x_{k-1;n}^0)\big(g(x_{k;n}^-)-g(x_{k-1;n}^+)\big)\Big),\\
&\geq - \tilde{f}(t) M \lim_{n\to\infty}\sum_{k=1}^{n+1} (x_{k;n}^- - x_{k-1;n}^+)\\
&= -\tilde{f}(t)M \lim_{n\to\infty}\mu \big( (s,t)\setminus \cup_{k=1}^{n} \mathcal{I}_{k;n}\big)\\
&= 0
\end{align*}
}
where the first inequality makes use of the fact that $g\geq 0$, and (as previously discussed) the fact that $f$ is monotone increasing implies that $\tilde{f}=f_{sing}$ is also; in the final equality we have let $\mu$ denote the Lebesgue measure.  Combining this estimate with (\ref{eq:AnalysisLem4}) proves that the final term in (\ref{eq:AnalysisLem2}) is monotone increasing, and thus Lemma \ref{lem:AnalysisLemma1} is proved.

\end{proof}

\subsection{Lagrangian boundary case}\label{sec:Monotonicty1and2}
Here we extend the results of Section \ref{sec:Monotonicity0} to the case in which the $J$-curves of interest have Lagrangian boundary condition.
For such results, we will need to take some additional care near the boundaries, so we begin with some observations in that direction.

Here and throughout, $(M,J,g)$ will be a compact almost Hermitian manifold, possibly with boundary.  We will also consider either one or two compact embedded totally geodesic Lagrangians $L_1,L_2\subset M$ which satisfy $\partial L_i = L_i\cap \partial M$.  When considering the two Lagranigans, we will also demand that and $L_1$ has finite transverse intersections with $L_2$.

Consider $(M,J,g,L)$ as above, and fix $p\in L$, and fix $r_0$ so that $0< r_0 < \frac{1}{2}\inj(p)$ and so that $L\cap \mathcal{B}_{r_0}(p)$ is connected.  As above, define the distance function  $\beta:\mathcal{B}_{r_0}(p)\to \R$ by $\beta(q):=\dist(p,q)$. Observe that by construction, the gradient field $\nab_M \beta$ is a smooth vector field on $\mathcal{B}_{r_0}(p)\setminus \{p\}$, with $\|\nab_M \beta\|\equiv 1$, and the integral curves of $\nab_M \beta$ are unit-speed geodesics in $M$ which emanate radially from from $p$.  Furthermore, $L$ is totally geodesic, connected, and contains $p$.  Consequently we have $\nab_M \beta \big|_{L} \subset TL$, or in other words the gradient field $\nab_M \beta$ is parallel to $L$.  As a further consequence, we see that for any vector normal to $L$ given by $X\in TL^\bot$, we have $\la X, \nab_M \beta\ra=0$.  This is of particular interest, since if $u:S\to \mathcal{B}_{r_0}(p)$ is a $J$-curve of Type 1, and we define $\nu \in \Gamma(TS)\big|_{\partial_1 S}$ to be an outward pointing $u^*g$-unit normal vector along the Lagrangian-type boundary $\partial_1 S$, then $J\nu\in T\partial_1 S\subset TL$ so that
\begin{equation}\label{eq:GradBetaPerpToOutNormal}
\la\nu,\nab_M\beta\ra = \la\nu,\nab_S\beta\ra =0.
\end{equation}
To make use of this fact, it will be convenient to define the sets
\begin{equation}\label{eq:DefOfSr}
S_r:=(\beta\circ u)^{-1}\big([0,r]\big)=u^{-1}\big(\mathcal{B}_r(p)\big),
\end{equation}
for $r\in [0,r_0]$. It then follows from (\ref{eq:GradBetaPerpToOutNormal}) that the set of critical values of $\beta\circ u\big|_{\partial_1 S}$ is contained in the set of critical values of $\beta\circ u \big|_{S}$; we conclude from this observation, and the fact that $u$ is a proper map into $M$,  that for each regular value $r$ of $\beta\circ u\big|_{S}$ the set $S_{r}$ has the structure of a smooth manifold with boundary and a finite number of corners.  In fact, we can conclude even more, namely that the boundary portions $\partial_0 S$ and $\partial_1 S$ intersect orthogonally.
We make this precise with Lemma \ref{lem:OrthoBoundary} below.
\begin{lemma}\label{lem:OrthoBoundary}
Letting $u:S\to\mathcal{B}_{r}(p)$ be a $J$-curve of Type 1 as above, with $r$ a regular value of $\beta\circ u$, then then $\partial_0 S_r \bot_{u^*g} \partial_1 S_r$.
\end{lemma}
\begin{proof}
Fix $\zeta\in \partial_1 S_r$ where $\beta\big(u(\zeta)\big)=r$ and $r$ is a regular value of $\beta\circ u$.  Let $\tau_1\in T_{\zeta} (\partial_1 S_r)$ and $\tau_0\in T_\zeta (\partial_0 S_r)$ be unit vectors.  We now abuse notation by letting each $\tau_i$ denote the push-forward $u_* \tau_i$. Our goal is to show that $\la \tau_0,\tau_1\ra_{g}=0$.  To that end we choose an orthonormal basis $\{e_1,\ldots, e_n,f_1,\ldots,f_n\}\in T_{u(\zeta)} M$ so that $J e_i = f_i$ for all $i=1,\ldots, n$, and so that $T_{u(\zeta)} L = \Span (e_1,\ldots,e_n)$, $\nab_M \beta = e_1$, and $\tau_1 = a_1^1 e_1 + a_1^2 e_2$.  Then $\tau_0$ can be written as $\tau_0= (\sum_{j\geq 2} a_0^j e_j)+(\sum_{j\geq 1} b_0^j f_j)$. Since $r$ is a regular value of $\beta$, it follows that $a_1^1\neq 0$, and since $\tau_0$ and $\tau_1$ span a $J$-complex line, it follows that there exist $x,y\in \R$ such that $\tau_0 = x \tau_1 + y J\tau_1$, which can be restated in our given basis as
\begin{equation*}
\big(\sum_{j\geq 2} a_0^j e_j\big)+\big(\sum_{j\geq 1} b_0^j f_j\big) = xa_1^1 e_1 + xa_1^2 e_2 + ya_1^1 f_1 + ya_1^2 f_2.
\end{equation*}
Since $a_1^1\neq 0$ it follows that $x=0$, and thus $\tau_0\in \Span(f_1,\ldots,f_n)=T_{u(\zeta)}L^\bot$ while $\tau_1\in \Span(e_1,\ldots,e_n)=T_{u(\zeta)} L$.  The result is then immediate.
\end{proof}
Next, consider the case in which $L_1$ and $L_2$ are transversely intersecting Lagrangians as above, and suppose $p\in L_1\cap L_2$.  Also fix $r_0$ such that $0<r<\frac{1}{2}\inj(p)$ and so that each of $L_i\cap \mathcal{B}_{r_0}(p)$ is connected.  Then arguing precisely as above, one concludes that for $J$-curves $u:S\to \mathcal{B}_{r_0}(p)$ of Type 2,
the set of critical values of each of $\beta\circ u\big|_{\partial_1 S}$ and $\beta\circ u\big|_{\partial_2 S}$ is contained in the set of critical values of $\beta\circ u \big|_{S}$, so that again
for each regular value $r$ of $\beta\circ u\big|_{S}$ the set $S_{r}$ has the structure of a (not necessarily compact) smooth manifold with boundary and a finite number of corners, and the proof of Lemma \ref{lem:OrthoBoundary} guarantees that for $i\in \{1,2\}$ we again have $\partial_0 S_r \bot_{u^*g} \partial_i S_r$.

Recall that $J$-curves $u:S\to M$ of Type 2 need not be proper near the finite set $L_1\cap L_2$. In order to proceed, we will need the following technical result.

\begin{lemma}\label{lem:SmallTermsVanish}
Let $(M,J,g, L_1,L_2)$ be as above, and let $u:S\to M$ be a $J$-curve of Type 2. Then there exists a decreasing sequence $\epsilon_k\to 0$ of positive numbers for which $\partial S_{\epsilon_k}^\beta$ has the structure of a smooth compact one-dimensional manifold with boundary, and
\begin{equation*}
\epsilon_k\int_{\partial S_{\epsilon_k}^\beta} 1 \to 0.
\end{equation*}
\end{lemma}
\begin{proof}
We begin by recalling that $\|\nab_M\beta\|\equiv 1$, $\nab_S \beta = (\nab_M \beta)^\top$, and therefore  $\int_{\partial S_\epsilon} \|\nab_S  \beta\|^{-1} \geq \int_{\partial S_\epsilon} 1 $, whenever $\epsilon$ is a regular value of $\beta\circ u$.  Consequently, to prove Lemma \ref{lem:SmallTermsVanish}, it is sufficient to prove the existence of a
sequence $\epsilon_k\to 0$ of positive numbers for which $\partial S_{\epsilon_k}^\beta$ has the structure of a smooth compact one-dimensional manifold with boundary, and
\begin{equation}\label{eq:SequenceWithDecay}
\epsilon_k\int_{\partial S_{\epsilon_k}^\beta} \|\nab_S \beta\|^{-1} \to 0.
\end{equation}
To that end, we suppose not.  Observe that the regular values of $\beta\circ u$ are open and dense in $[0,r]$, so if the result does not hold, then it must be the case that there exist $\delta,\epsilon_0>0$ such that for almost every $\epsilon\in [0,\epsilon_0]$ we have
\begin{equation}\label{eq:ElementaryCoercive}
\epsilon\int_{\partial S_\epsilon^\beta} \|\nab_S \beta\|^{-1} \geq \delta>0.
\end{equation}
By applying the coarea formula and inequality (\ref{eq:ElementaryCoercive}), which holds for almost every $\epsilon$, we find
\begin{align*}
\Area_{u^*g}(S_{\epsilon_0}^\beta)&=
\int_{S_{\epsilon_0}^\beta}1 \;\; =\;\; \int_0^{\epsilon_0}\Big(\int_{\partial S_\epsilon^\beta}\|\nab_S \beta\|^{-1}\Big) d\epsilon \\
&\geq \int_0^{\epsilon_0} \frac{\delta}{\epsilon}d\epsilon\;\;=\;\;\infty.
\end{align*}
However, $S$ has finite $u^*g$-area, which provides the desired contradiction.  This completes the proof of Lemma \ref{lem:SmallTermsVanish}.
\end{proof}

We now prove the analogue of Lemma \ref{lem:IntegrationByParts0} for the Lagrangian boundary case. Again, it is essentially just integration by parts while carefully managing the boundary terms.

\begin{lemma}\label{lem:IntegrationByParts1and2}
Let $u:S\to M$ be a $J$-curve  either of Type 1 into $(M,J,g,L_1)$  or else of Type 2 into $(M,J,g,L_1,L_2)$ where $L_1$ and $L_2$ transversely intersecting totally geodesic Lagrangians as above. Next suppose $f\in C^2\cap W^{2,\infty}(S,\R)$.  Furthermore, letting $\nu \in \Gamma( TS)\big|_{\partial S_r}$ denote an outward pointing $u^*g$-unit normal vector, and using the decompositions $\partial S_r = \partial_0 S_r \cup\partial_1 S_r$ in the Type 1 case and $\partial S_r = \partial_0 S_r \cup\partial_1 S_r \cup\partial_2 S_r$ in the Type 2 case,
we assume that $\la \nab_S f, \nu\ra\big|_{\partial_i S_r} =0$ for $i=1,2$.  Then for every regular value $r$ of $\beta=\beta\circ u$ we have
\begin{equation*}
\int_{S_r}\!\!f\Delta_S \beta^{2}
=\int_{S_r}\!\! (\beta^2-r^2)\Delta_S f
+2r\int_{\partial_0 S_r}\!\!\!\!  f\| \nabla_S \beta\|.
\end{equation*}
\end{lemma}
\begin{proof}
We prove the case in which $u:S\to M$ is Type 2; the Type 1 case is simpler since the domain is compact, and the result in that case is easily deduced.  We begin by letting $\epsilon_k\to 0$ be the sequence guaranteed by Lemma \ref{lem:SmallTermsVanish}, we define $S_{r,\epsilon_k}$ to be the closure of $S_{r}\setminus S_{\epsilon_k}$ and integrate by parts twice to find
\begin{align}
\int_{S_r}\!\!f\Delta_S \beta^{2}
&=\lim_{k\to\infty}\int_{S_{r,\epsilon_k}}\!\!f\Delta_S \beta^{2}\notag\\
&=\lim_{k\to\infty}\Big(\int_{S_{r,\epsilon_k}}\!\! \beta^2\Delta_S f-\int_{\partial S_{r,\epsilon_k}}\!\!\!\!\beta^{2}\la
\nabla_S f,\nu\ra
+\int_{\partial S_{r,\epsilon_k}}\!\!\!\! 2\beta f\la \nabla_S \beta,\nu \ra\Big)\notag \\
&=\int_{S_r}\!\! \beta^2\Delta_S f-\int_{\partial_0 S_r}\!\!\!\!\beta^{2}\la
\nabla_S f,\nu\ra
+\int_{\partial_0 S_r}\!\!\!\! 2\beta f\la \nabla_S \beta,\nu \ra\notag\\
&\qquad-\lim_{k\to\infty} \int_{\partial_0 S_{\epsilon_k}}\!\!\!\!\beta^{2}\la
\nabla_S f,\nu\ra
+\lim_{k\to\infty}\int_{\partial_0 S_{\epsilon_k}}\!\!\!\! 2\beta f\la \nabla_S \beta,\nu \ra,\notag\\
&=\int_{S_r}\!\! \beta^2\Delta_S f-\int_{\partial_0 S_r}\!\!\!\!\beta^{2}\la
\nabla_S f,\nu\ra
+\int_{\partial_0 S_r}\!\!\!\! 2\beta f\la \nabla_S \beta,\nu \ra\label{eq:IntByPartsComputation}
\end{align}
where to obtain the first inequality, we recall that $S$ is a finite measure space and $f\in L^\infty$ and $\Delta_S \beta^2 \in L^\infty$ (this latter statement follows from Lemma \ref{lem:LaplacianEstimate}). To obtain the third equality, we made use of the fact that by assumption $\la \nab_S f, \nu\ra\big|_{\partial_i S_r^\beta} =0$ for $i=1,2$, and by the discussion at the beginning of this section, in particular equation (\ref{eq:GradBetaPerpToOutNormal}), we have $\la \nab_S \beta, \nu\ra\big|_{\partial S_{r,\epsilon_k}^\beta}=0$ for $i=1,2$; to obtain the final equality, we have made use of the fact that the functions $\beta \|\nab_S f\|$ and $f$ are in $C^0\cap L^\infty (S,\R)$, $\beta\big|_{\partial_0 S_{\epsilon_k}^\beta}\equiv \epsilon_k$, and Lemma \ref{lem:SmallTermsVanish}.  Finally, pulling the $\beta$ terms outside the boundary integrals, observing that $\nu\big|_{\partial_0 S_r^\beta}=\|\nab_S \beta\|^{-1}\nab_S \beta$, and integrating by parts once more on the second term in (\ref{eq:IntByPartsComputation}) (and again arguing via limits and Lemma \ref{lem:SmallTermsVanish}) yields
\begin{align*}
\int_{S_r}\!\!f\Delta_S \beta^{2}
=\int_{S_r}\!\! \beta^2\Delta_S f-r^2\int_{ S_r}\!\!\!\!
\Delta_S f
+2r\int_{\partial_0 S_r}\!\!\!\!  f\| \nabla_S \beta\|,
\end{align*}
which is precisely the desired result.
\end{proof}

%%%%%%%%%%%%%%%%%%%%%%%%%%%%%%%%%%%%%%%%%%%%%%%%%%%%%%%%%%%%%%%%%%%%%%%%%%%%%%%%%%%%%%%%%%%%%%%%%%%%%%%%%%%%%%%%%%%%%%%%%%%%%
%%%%  MAIN MONOTONICITY THEOREM  %%%%%%%%%%%%%%%%%%%%%%%%%%%%%%%%%%%%%%%%%%%%%%%%%%%%%%%%%%%%%%%%%%%%%%%%%%%%%%%%%%%%%%%%%%%%
%%%%%%%%%%%%%%%%%%%%%%%%%%%%%%%%%%%%%%%%%%%%%%%%%%%%%%%%%%%%%%%%%%%%%%%%%%%%%%%%%%%%%%%%%%%%%%%%%%%%%%%%%%%%%%%%%%%%%%%%%%%%%
\begin{theorem}[Monotonicity]\label{thm:MonotonicityType1and2}
Fix a constant $C>0$, and let $u:S\to M$ be a $J$-curve  either of Type 1 into $(M,J,g,L_1)$  or else of Type 2 into $(M,J,g,L_1,L_2)$, where $L_1$ and $L_2$ are compact embedded totally geodesic Lagrangians for which $\partial L_i=L_i\cap \partial M$.  If there are two such Lagrangians, we also assume they have finite transverse intersections on the interior of $M$. Suppose further that
\begin{equation*}
\sup_{p\in M} |K_{sec}(p)| \leq {\textstyle\frac{1}{4}}C^2\qquad\text{and}\qquad \sup_{\substack{e\in TM\\\|e\|=1}}\|Q(e,e)+Q(Je,Je) \| \leq {\textstyle\frac{1}{2}}C;
\end{equation*}
here $|K_{sec}(p)|$ is defined as in (\ref{eq:DefOfAbsKsec}) and $Q$ is the tensor defined in Lemma \ref{lem:InhomogeneousMeanCurvatureEq}.
Fix $p\in L_1$ or $p\in L_1\cap L_2$ as appropriate, and fix $b\in \R$ such that $0< b <\frac{1}{2}\min\big(\inj(p),C^{-1}\big)$ and such that $\mathcal{B}_b(p)\cap L_i$ is connected for each $i$.  Lastly, suppose $f\in C^2\cap W^{2,\infty}(S,\R)$ with $f\geq 0$ and $\Delta_S f\geq -\lambda b^2 f$ for some $\lambda>0$. Then for any $a\in [0,b]$ the following holds
\begin{equation*}
e^{\frac{\lambda a}{2b}+2 Ca}a^{-2 }\int_{S_a}f\leq e^{\frac{\lambda}{2}+2 Cb}b^{-2
}\int_{S_b}f,
\end{equation*}
where $S_r := u^{-1}\big(\mathcal{B}_r(p)\big)$.
\end{theorem}
\begin{proof}The proof of Theorem \ref{thm:MonotonicityType1and2} is identical to that of Theorem \ref{thm:MonotonicityType0} after the following two modifications are made: first, all references to $\partial S_r$ must be replaced with references to $\partial_0 S_r$, and second, references to Lemma \ref{lem:IntegrationByParts0} must be replaced with references to Lemma \ref{lem:IntegrationByParts1and2}.  Also note that by Lemma \ref{lem:OrthoBoundary} and the discussion proceeding it, the coarea formula (\ref{eq:coareaFormula0}) and its consequence (\ref{eq:coareaConsequence0}) again hold for these curves.
\end{proof}

As in Section \ref{sec:Monotonicity0} we deduce some immediate corollaries.

\begin{corollary}\label{cor:AlternateSr1}
Theorem \ref{thm:MonotonicityType1and2} also holds in the Type 1 case when $\zeta_0\in \partial_1 S$ is fixed and the $S_r$ are alternatively defined to be the connected component of $u^{-1}\big(\mathcal{B}_{r}(p)\big)$ which contains $\zeta_0$.
\end{corollary}

\begin{corollary}[Mean Value Inequality]\label{cor:MeanValueInequality1} Assume the same hypotheses as in Theorem \ref{thm:MonotonicityType1and2}, with $u:S\to M$ a $J$-curve of Type 1.  Then for each $\epsilon>0$ there exists a positive $r_0$ depending only on $C$ such that, the following also holds:
\begin{equation}\label{eq:MeanValueInequality1}
\sum_{\zeta\in u^{-1}\big(u(\zeta_0)\big)}f(\zeta)\mu(\zeta)\leq\frac{(1+\epsilon)e^{\frac{\lambda }{2}}}{\pi
b^{2 }}\int_{S_b}f,
\end{equation}
where $\mu(\zeta):=\lim_{a\to 0} \frac{1}{\pi a^2} \Area\big(S_a(\zeta)\big)$ and $S_a(\zeta)$ is the connected component of $u^{-1}\big(\mathcal{B}_a(u(\zeta))\big)$ which contains $\zeta$.  In particular, if $\zeta_0 = u^{-1}\big(u(\zeta_0)\big)\in \partial_1 S$ and $\zeta_0$ is not a critical point of $u$, then
\begin{equation*}
f(\zeta_0)\leq \frac{2(1+\epsilon)e^{\frac{\lambda }{2}}}{\pi
b^{2 }}\int_{S_b}f.
\end{equation*}
\end{corollary}

\begin{corollary}[Monotonicity of Area]  \label{cor:MonotonicityOfArea1}
Fix the same assumptions as in Theorem \ref{thm:MonotonicityType1and2} with  a $J$-curve $u:S\to \R$ of Type 1, fix  $\zeta_0\in \partial_1 S$ a regular point of $u$, and define $S_r(\zeta_0)$ as in Corollary \ref{cor:MeanValueInequality1}.  Then for each $\epsilon>0$ there exists an $r_0>0$ depending on $C$ such that whenever $0<b<r_0$ the following holds:
\begin{equation*}
\frac{ 1}{a^{2}}\Area_{u^*g}(S_a(\zeta_0))\leq \frac{(1+\epsilon)}{b^{2 }}\Area_{u^*g}(S_b(\zeta_0)).
\end{equation*}
Or in more common form:
\begin{equation*}
(1+\epsilon)^{-1}\frac{\pi b^2}{2}\leq \Area_{u^*g}(S_b(\zeta_0)).
\end{equation*}
\end{corollary}

\section{Graphs with a priori bounds}\label{sec:GraphsWithAPrioriBounds}
The purpose of this section is to prove Theorems \ref{thm:AprioriDerivataiveBounds} and \ref{thm:AprioriDerivataiveBounds2}, which essentially establish that $J$-curves (with or without Lagrangian boundary) which have $L^\infty$-bounded second fundamental forms, can locally be graphed over a coordinate tangent plane with a priori bounds on all derivatives.  Also of importance here is that the domain of these graphical parameterizations are disks or half-disks with radii bounded away from $0$ depending only on the curvature bound and properties of the geometry of the ambient manifold $(M,J,g)$ or $(M,J,g,L)$ as appropriate. Such parameterizations with a priori bounds are necessary for the analysis in Section
\ref{sec:CurvatureRegularity}, as well as in the analysis in \cite{Fj09a}.  In Section \ref{sec:InteriorCase} we consider the case without Lagrangian boundary, and in Section \ref{sec:BoundaryCase} we consider the case with Lagrangian boundary.

\subsection{Interior case}\label{sec:InteriorCase}
Here and throughout this Section \ref{sec:GraphsWithAPrioriBounds}, $\mathcal{K}$ will denote a compact set contained in a manifold $M$, and we will say $\mathcal{K}\subset\subset M$ provided $\mathcal{K}\subset \Int(M)$. Also, given a Riemannian metric $g$ on $M$, we will let $\nab$ denote the associated Levi-Civita connection, and for any immersion $u:S\to M$, we let $B$ denote the associated second
fundamental form as defined in Section \ref{sec:MeanCurvatureEquation}.  It will also be convenient to use the notation
\begin{equation*}
\mathcal{D}_r:=\{(s,t)\in \mathbb{R}^2 : s^2+t^2<r^2\}
\end{equation*}
\begin{equation*}
\mathcal{B}_\epsilon(p):=\{q\in M:\dist_g(p,q)<\epsilon\}.
\end{equation*}

\begin{definition}[$\|T\|_{C_g^k(\mathcal{K})}$ and $\|T\|_{C_g^{k,\alpha}(\mathcal{K})}$]\label{def:C1BoundOnG}
Let $(M,g)$ be a Riemannian manifold  and let $T$ be a tensor on $M$. For each point $p\in \mathcal{K}$ we define the ball $\mathcal{B}(p):=\mathcal{B}_{\frac{3}{4}\inj(p)}(p)$, and equip it with geodesic normal coordinates $(x^1,\ldots,x^m)$ centered at $p$. For multi-indices $I_1$ and $I_2$, we let $T_{I_1}^{I_2}$ denote the components of $T$ in the coordinates $(x^1,\ldots,x^m)$, and then define
\begin{equation*}
[T]_{C_g^k(p)}:=\sup_{q\in\mathcal{B}(p)} \Big(\sum_{|\alpha|=k}\sum_{I_1,I_2}|D_\alpha T_{I_1}^{I_2}(q)|^2\Big)^{1/2}\quad\text{and}\quad \|T\|_{C_g^k(p)}:= \sum_{j=0}^k [T]_{C_g^j(p)}.
\end{equation*}
Similarly for each $\alpha\in (0,1]$ we define
\begin{equation*}
\|T\|_{C_g^{k,\alpha}(p)}:= \|T\|_{C_g^k(p)} + \sup_{\substack{q_1,q_2\in\mathcal{B}(p) \\ q_1\neq q_2}}\Big(\sum_{|\alpha|=k}\sum_{I_1,I_2}\frac{|D_\alpha T_{I_1}^{I_2}(q_1)-D_\alpha T_{I_1}^{I_2}(q_2)|^2}{|q_1-q_2|^{2\alpha}}\Big)^{1/2},
\end{equation*}
where $|q_1-q_2|$ is the distance between $q_1$ and $ q_2$ computed with respect to the flat metric $dx^i\otimes dx^i$. More generally, if $\mathcal{K}\subset M\setminus \partial M$ is a set, then by
repeating the above construction for each $p\in \mathcal{K}$, we also define
\begin{equation*}
[T]_{C_g^k(\mathcal{K})}:= \sup_{p\in \mathcal{K}} [T]_{C_g^k(p)},\qquad \|T\|_{C_g^k(\mathcal{K})} := \sup_{p\in \mathcal{K}} \|T\|_{C_g^k(p)},\;\;\text{and}
\end{equation*}
\begin{equation*}
\|T\|_{C_g^{k,\alpha}(\mathcal{K})}:=\sup_{p\in \mathcal{K}} \|T\|_{C_g^{k,\alpha}(p)}.
\end{equation*}
\end{definition}

\begin{definition}[admissible]\label{def:admissible}
The triple $(M,g,\mathcal{K})$, where $(M,g)$ is a compact Riemannian manifold of dimension $m$ (possibly with boundary) and $\mathcal{K}\subset M$ is a compact subset, is said to be \emph{admissible} provided the following hold.
\begin{equation*}
10m^2[g]_{C_g^2(\mathcal{K})} \leq 1\qquad\text{and}\qquad\inf_{q\in \mathcal{K}}\inj(q)\geq 2.
\end{equation*}
\end{definition}

\begin{remark}
The notion of the admissibility of a triple $(M,g,\mathcal{K})$ is purely a computational convenience.  Indeed, if $\mathcal{K}\subset M$ is compact and contained in the interior of $M$, and $g$ is any twice differentiable Riemannian metric on $M$, then there exists a $C_0>0$ with the property that for all $C\geq C_0$ the triple $(M,Cg,\mathcal{K})$ is admissible. The existence of such a $C_0$ can be deduced from Lemma \ref{lem:RescalingProperties} below which describes the behavior of some geometric quantities under such a conformal rescaling of $g$.
\end{remark}

\begin{lemma}\label{lem:RescalingProperties}
Let $(M,g)$ be a Riemannian manifold, let $\mathcal{K}\subset M$ be a compact set contained in the interior of $M$, and define another metric on $M$ by $g^c:= c^2 g$ where $c>0$.  Also suppose $u:S\to M$ is an immersion. Then the following hold.
\begin{enumerate}[({R}1)]
\item $\inj^{g^c}(p)=c\inj^g(p) $;\label{en.R1}
\item $[g^c]_{C_{g^c}^k(\mathcal{K})}=c^{-k} [g]_{C_g^k(\mathcal{K})} $;\label{en.R2}
\item $[J]_{C_{g^c}^k(\mathcal{K})}=c^{-k} [J]_{C_g^k(\mathcal{K})} $;\label{en.R2p5}
\item $K_{sec}^{g^c}(X,Y)=c^{-2}K_{sec}^g(X,Y)$;\label{en.R3}
\item $\|B_u(\zeta)^{g^c}\|_{g^c} = c^{-1}\|B_u(\zeta)^g\|_{g}$;\label{en.R4}
\end{enumerate}
where $B_u^g(\zeta)$ and $B_u^{g^c}(\zeta)$ denote the second fundamental forms of the immersion $u:S\to M$ evaluated at $\zeta\in S$ and computed with respect to the metrics $g$ and $g^c$ respectively.
\end{lemma}
\begin{proof}
We begin by observing that if $\nab$ and $\nab^c$ are the Levi-Civita connections associated to $g$ and $g^c$ respectively, then $\nab=\nab^c$.  This can be proved by observing that the Christoffel symbols of $\nab$ are invariant under rescaling the metric.  Consequently, it follows that paths which are constant speed $g$-geodesics are also constant speed $g^c$-geodesics, and thus
\begin{equation*}
\exp_p^g (X) = \exp_p^{g^c}(X).
\end{equation*}
It is then straightforward to deduce (R\ref{en.R1}).

Next observe that if $e_1,\ldots, e_m$ is a $g$-orthonormal frame for $T_p M$, then the vectors $c^{-1} e_1,\ldots, c^{-1} e_m$ form a $g^c$-orthonormal frame for $T_p M$.  Consequently, if $\mathcal{O}\subset M$ is some small neighborhood of $p$, and $\Phi:\mathcal{O}\to \R^{2n}$ are $g$-geodesic normal coordinates, and $\Phi^c:\mathcal{O}\to \R^{2n}$ are $g^c$-geodesic normal coordinates, then $\Phi^c = c \Phi$.  However, for $x\in \R^{2n}$ and $Y,Z\in T_x\R^{2n}$, we can then compute
\begin{align*}
\big(\Phi_*^c g^c\big)_x(Y,Z) &= g_{(\Phi^c)^{-1}(x)}^c\big( \Phi^{c*} Y, \Phi^{c*}  Z\big) =c^2 g_{\Phi^{-1}(x/c)}\big(c^{-1} \Phi^* Y,c^{-1} \Phi^* Z \big)\\
&= g_{\Phi^{-1}(x/c)}\big( \Phi^* Y,\Phi^* Z\big) =\big(\Phi_* g)_{(x/c)} (Y,Z).
\end{align*}
Thus letting $g_{ij}^c$ and $g_{ij}$ denote the components of $g^c$ and $g$ respectively in $g^c$ and $g$ geodesic normal coordinates, we find that $g_{ij}^c(x) = g_{ij}(x/c)$, and then (R\ref{en.R2}) is easily deduced.  Arguing similarly, one finds that $(\Phi_*^c J)_{ij}(x)=(\Phi_* J)_{ij}(x/c)$, and (R\ref{en.R2p5}) follows immediately.

Next, recall that for smooth sections $Y,Z\in \Gamma(TM)$ which are point-wise linearly independent, the sectional curvature of the plane spanned by $Y$ and $Z$ is given by
\begin{equation*}
K_{sec}^g(Y,Z) = \frac{\la\nab_Y \nab_Z Z -\nab_Z \nab_Y Z - \nab_{[Y,Z]}Z, Y\ra_g}{\|Y\|_g^2\|Z\|_g^2 - \la Y,Z\ra_g^2}.
\end{equation*}
As observed above, $\nab=\nab^c$, so (R\ref{en.R3}) is then quickly deduced.  Similarly, since $B_u^g(Y,Z) = (\nab_Y Z)^\bot$ for sections $Y,Z\in \Gamma(\mathcal{T})$ of the tangent bundle of the immersion $u:S\to M$, we see that (R\ref{en.R4}) also follows immediately.
\end{proof}

It will be convenient for later computations to have the following basic estimates at our disposal. The proof of  Lemma \ref{lem:GammaEst} below is elementary.
\begin{lemma}\label{lem:GammaEst}
Let $(M,g,\mathcal{K})$ be an admissible triple, $p\in \mathcal{K}$ a point, $(x^1,\ldots,x^m)$ geodesic normal coordinates centered at $p$, $g_{ij}$ the components of $g$ in these coordinates, $g^{ij}$ which satisfy $g_{i\ell} g^{\ell j} = \delta_i^j$, and $\Gamma_{ij}^k$ the Christoffel symbols of the Levi-Civita connection associated to $g$ in these coordinates.  Then for all points $q\in M$ for which $|q|:=\dist(p,q)\leq 1$, the following inequalities hold.
\begin{align*}
&{\textstyle\sum_{ij=1}^m} |g_{ij}(q)-\delta_{ij}|^2\leq |q|^4,
&&{\textstyle\sum_{ijk=1}^m} |\Gamma_{ij}^k(q)|^2 \leq |q|^2\\
&{\textstyle\sum_{ij=1}^m} |g^{ij}(q)-\delta^{ij}|^2\leq |q|^4
&&\sup_{p\in \mathcal{K}} |K_{sec}(p)| \leq 1.
\end{align*}
Furthermore, for $\bar{g}:=dx^i\otimes dx^i$, the following inequalities hold.
\begin{align*}
\big| \la Y,Z\ra_g - \la Y,Z\ra_{\bar{g}} \big|  &\leq |q|^2 \|Y\|_{\bar{g}}\|Z\|_{\bar{g}}\\
\big(1-|q|^2\big)\|Y\|_{\bar{g}}^2\;\;\leq\;\;  &\|Y\|_g^2 \;\;\leq\;\; \big(1+|q|^2\big)\|Y\|_{\bar{g}}^2.
\end{align*}
Lastly, for $\mathcal{V}\subset T_q M$, a vector space of dimension $k$, we let $\pi_\mathcal{V}^g$ and $\pi_\mathcal{V}^{\bar{g}}$ respectively denote the $g$ and $\bar{g}$ orthogonal projections onto $\mathcal{V}$. Then for $q$ as above, and
$Y,Z\in T_qM$, the following inequalities also hold.
\begin{equation*}
\big\| \pi_\mathcal{V}^g (Y) - \pi_\mathcal{V}^{\bar{g}}(Y)\big\|_{\bar{g}}\leq |q|^2\|Y\|_{\bar{g}}.
\end{equation*}
\end{lemma}

We are now prepared to prove a a key result about graphical parameterizations.

\begin{lemma}[Uniform Local Graphs]\label{lem:uniformLocalGraphs}
There exists a constant $r_0>0$ with the following significance.  Let $(M,g,\mathcal{K})$ be an admissible triple with $\dim M = m$, and suppose $u:S\to M$ is a Type 0 immersion for which
\begin{equation*}
\sup_{\zeta\in S} \|B_u(\zeta)\|_g \leq 1.
\end{equation*}
Then for each $\zeta\in u^{-1}(\mathcal{K})$, there exists a map $\phi:\mathcal{D}_{r_0}\to S$ and geodesic normal coordinates $\Phi:\mathcal{B}_{2r_0}\big(u(\zeta)\big)\to\mathbb{R}^m$ with the following properties.
\begin{enumerate}[({P}1)]
\item $\phi(0)=\zeta$\label{en.P1}
\item $\tilde{u}(s,t):=\Phi\circ u \circ \phi(s,t)=\big(s,t,\tilde{u}^3(s,t),\ldots,\tilde{u}^m(s,t)\big)$.\label{en.P2}
\item $\tilde{u}^i(0,0)=0$ for $i=1,\ldots, m$\label{en.P3}
\item $D_\alpha \tilde{u}^i(0,0)=0$ whenever $|\alpha|=1$ and $i=3,\ldots,m$\label{en.P4}
\item $\sum_{|\alpha|=1}\sum_{i=3}^m\|D_\alpha \tilde{u}^i\|_{C^0(\mathcal{D}_{r_0})}^2 \leq 10^{-20}$\label{en.P5}
\item For Euclidian coordinates $\rho=(s,t)$, on $\mathcal{D}_{r_0}$, we have\label{en.P6}
\begin{equation*}
{\textstyle \frac{1}{2}}|\rho|\leq \dist_{(u\circ \phi)^*g}(0,\rho) \leq 2 |\rho|
\end{equation*}
\item For $|\alpha|=2$ and $i=1,\ldots,m$, $\|D_\alpha \tilde{u}^i\|_{C^0(\mathcal{D}_{r_0})} \leq 10 $ \label{en.P7}
\item With $\rho$ as above, we have\label{en.P8}
\begin{equation*}
    {\textstyle\frac{1}{2}}|\rho|\leq \dist_{g}\big(u(\phi(\rho)),u(\phi(0))\big)\leq 2|\rho|.
\end{equation*}
\end{enumerate}
\noindent Furthermore, letting subscripts $s$ and $t$ denote partial derivatives, and denoting $\gamma_{11}=\la u_s,u_s\ra_g$, $\gamma_{12}=\gamma_{21}=\la u_s,u_t\ra_g$, $\gamma_{22}=\la u_t,u_t\ra_g$, and $\gamma_{ik}\gamma^{k j}=\delta_i^j$, the following inequalities also hold.
\begin{enumerate}[({P}1)]
\setcounter{enumi}{8}
\item $\frac{1}{4}|\xi|^2\leq \inf_{\rho\in\mathcal{D}_{r_0}} \gamma^{ij}\xi_i\xi_j$\label{en.P9}
\item $\|\gamma^{ij}\|_{C^{0,\alpha}(\mathcal{D}_{r_0})}\leq 2\cdot 10^3$. \label{en.P10}
\end{enumerate}
\end{lemma}
\begin{proof}
Begin by defining a constant $\delta:= 10^{-10}$. Assume that $\Phi:\mathcal{B}_{2\delta}\big(u(\zeta)\big)\to\mathbb{R}^m$ has been chosen so that $\Phi\big(u(\zeta)\big)=0$ and the image of $T_\zeta (\Phi\circ u)$ is $\mathbb{R}^2\times \{(0,\ldots,0)\}$.
We next claim (and shall prove) that
\begin{equation}\label{eq:exponentialMap}
\exp_\zeta :\{X\in T_\zeta S: \|X\|_{u^*g}\leq \delta\}\to S
\end{equation}
is well defined and is an immersion.  Indeed, the map is well defined since
\begin{equation*}
\dist_{u^*g}(\zeta,\partial S)\geq \dist_g\big(u(\zeta),\partial M)\geq\dist_g(\mathcal{K},\partial M)\geq 2 \geq \delta,
\end{equation*}
where the second-to-last inequality follows from the admissibility of $(M,g,\mathcal{K})$. If the map were not an immersion, then the Riemannian manifold $(S,u^*g)$ would have conjugate points $\zeta,\zeta'$ for which $\dist_{u^*g}(\zeta,\zeta')\leq \delta$.  By considering the Jacobi equation along the shortest connecting geodesic, one can show that there must be some point $\zeta''\in S$ at which the Gaussian curvature satisfies $K_{u^*g}(\zeta'')\geq \delta^{-2}$.  However recall the Gauss equations for immersions:
\begin{equation*}
K_{sec}(e_1,e_2)=K_{u^*g}(\zeta'') - \la B(e_1,e_1),B(e_2,e_2)\ra_g + \| B(e_1,e_2)\|_g^2
\end{equation*}
where $e_1,e_2\in \mathcal{T}_{\zeta''}$ is an orthonormal basis; consequently
\begin{equation}\label{eq:CurvatureContradiction}
10^{20}=\delta^{-2}\leq K_{u^*g}(\zeta'') \leq 2,
\end{equation}
where we have made use of the fact that the length of the second fundamental form is bounded by assumption, and that $\sup_{\mathcal{K}}|K_{sec}|\leq 1$.  Indeed, this latter statement follows from the fact that $(M,g,\mathcal{K})$ is an admissible triple, together with Lemma \ref{lem:GammaEst}. The contradiction (\ref{eq:CurvatureContradiction}) then shows that indeed the exponential map in (\ref{eq:exponentialMap}) is an immersion.

Next observe that for any $X\in T_\zeta S$ with $\|X\|_{u^*g}=1$, the path given by $\alpha(t):=u\circ \exp_\zeta^{u^*g }(tX)$ is the image by $u$ of a $u^*g$-unit speed geodesic in $S$ emanating from $\zeta$, and it is defined for $t\in[0,\delta]$.  Furthermore, for geodesic normal  coordinates $\Phi=(x^1,\ldots,x^{m})$, we can locally define the flat metric $\bar{g}:=dx^i\otimes dx^i$; we also locally define the $(2,1)$ tensor $\Gamma$ by
\begin{equation*}
\Gamma(X,Y):= \nab_X Y - \bar{\nab}_X Y,
\end{equation*}
where $\nab$ and $\bar{\nab}$ are  the  Levi-Civita connections  associated to $g$ and $\bar{g}$ respectively.  Note that conveniently the components of $\Gamma$ are precisely the Christoffel symbols of $\nab$ in the specified geodesic normal coordinates. We then estimate as follows:
\begin{align*}
\frac{d}{dt}dx^i\big(\dot{\alpha}(t)\big) &= dx^i\big(\nabla_{\dot{\alpha}}\dot{\alpha}\big) - dx^i\big(\Gamma(\dot{\alpha},\dot{\alpha})\big)\\
&=dx^i\big((\nabla_{\dot{\alpha}}\dot{\alpha})^\bot\big) - dx^i\big(\Gamma(\dot{\alpha},\dot{\alpha})\big)\\
&=dx^i\big(B(\dot{\alpha},\dot{\alpha})\big) - dx^i\big(\Gamma(\dot{\alpha},\dot{\alpha})\big);
\end{align*}
here $X\mapsto X^\bot$ is the $g$-orthogonal projection to the normal bundle of $\mathcal{N}\to S$. Making use of the fact that $\|B\|_g\leq 1$, we integrate up to find that
\begin{align}
\big|dx^i\big(\dot{\alpha}(t)\big) - dx^i\big(\dot{\alpha}(0)\big) \big| &\leq t\big(\|B(\dot{\alpha},\dot{\alpha})\|_g\|dx^i\|_g + \|\Gamma\|_{\bar{g}}\|dx^i\|_{\bar{g}}\big)\label{eq:smallDeviation}\\
&\leq t(\|dx^i\|_g + \|\Gamma\|_{\bar{g}})\notag\\
&\leq 3t,\notag
\end{align}
where we have made use of the following facts: $\|dx^i\|_g\leq 2$ and $\|\Gamma\|_{\bar{g}}\leq 1$; these estimates are easily deduced from Lemma \ref{lem:GammaEst}. Similarly for a continuous unit vector field $\eta$ along $\alpha$ which is locally tangent to the image of $u$ but orthogonal to $\dot{\alpha}$, one finds
\begin{equation}\label{eq:goodParam2}
\big|dx^i\big(\eta(t)\big) - dx^i\big(\eta(0)\big) \big| \leq   3t.
\end{equation}
Let $\pr:\mathbb{R}^{m}\to\mathbb{R}^2$ be the coordinate projection $\pr(x^1,\ldots,x^{m})=(x^1,x^2)$. It follows from inequalities (\ref{eq:smallDeviation}) and (\ref{eq:goodParam2}) that there exists a $r_0\in (0,\delta)$  such that the map
\begin{equation*}
\varphi:=\pr\circ \Phi\circ u\circ \exp_{\zeta}^{u^*g} : \{X\in T_\zeta S: \|X\|_{u^*g} < 2r_0\}\to \mathbb{R}^2
\end{equation*}
is a diffeomorphism with its image, its image contains $\mathcal{D}_{r_0}$, and the map
\begin{equation*}
\phi:=\exp_\zeta^{u^*g}\circ \varphi^{-1}:\mathcal{D}_{r_0}\to S
\end{equation*}
is well defined and satisfies properties (P\ref{en.P1}) - (P\ref{en.P6}). We postpone the proof of (P\ref{en.P7}) for the moment, and instead prove (P\ref{en.P8}). To that end, we note that by integrating (P\ref{en.P5}) it follows that for $\rho=(s,t)$
\begin{equation}\label{eq:goodParam3}
\sum_{i=3}^m \big(\tilde{u}^i(\rho)\big)^2\leq \sum_{i=3}^m \big(|s|\| \tilde{u}_s^i\|_{C^0(\mathcal{D}_{r_0})} + |t|\|\tilde{u}_t^i \|_{C^0(\mathcal{D}_{r_0})}  \big) ^2 \leq 2\cdot10^{-20}|\rho|^2\leq {\textstyle \frac{1}{4}}|\rho|^2.
\end{equation}
Making use of the fact that $\dist_{\Phi_*g}\big(\tilde{u}(\rho),0\big)^2=\sum_{i=1}^m \big(\tilde{u}^i(\rho)\big)^2$, we see that the right-most inequality of property (P\ref{en.P8}) follows from inequality (\ref{eq:goodParam3}).  However the left-most inequality of property (P\ref{en.P8}) also holds, since
\begin{equation*}
|\rho|^2 \leq |\rho|^2 + \sum_{i=3}^m \big(\tilde{u}^i(\rho)\big)^2 = \dist_{\Phi_*g}\big(\tilde{u}(\rho),0\big)^2.
\end{equation*}

To prove property (P\ref{en.P7}) it will be convenient to work in $\R^m$ rather than $\mathcal{B}_{2\delta}(p)$, so we will abuse notation by allowing $g$ and $\bar{g}$ to respectively denote $\Phi_*g$ and $\Phi_*\bar{g}$.  Roughly speaking, property (P\ref{en.P7}) will follow from the fact that the length of the second fundamental form $B_u$ is a priori bounded; the proof would be obvious if second derivatives of $\tilde{u}$ (when regarded as vector fields along $\tilde{u}(\mathcal{D}_{r_0})$) were always $g$-orthogonal to the surface $\tilde{u}(\mathcal{D}_{r_0})$, however this is not the case.  Thus we need the following computational lemma to bound the second derivatives by their normal components.
\begin{lemma}\label{lem:ProjectionLemma}
Let $\tilde{u}:\mathcal{D}_{r_0}\to \R^m$ be as above.  Suppose $q\in \tilde{u}(\mathcal{D}_{r_0})$, and $X\in T_q \R^m$ with $dx^1(X)=0=dx^2(X)$, and let $\pi_{\mathcal{N}}^g: T_q \R^m \to \mathcal{N}_q$ denote the $g$-orthogonal projection to the fiber $\mathcal{N}_q$ of the $g$-normal bundle $\mathcal{N}\to \tilde{u}(\mathcal{D}_{r_0})$.  Then
\begin{equation}
\|X\|_{g}\leq 2\|\pi_{\mathcal{N}}^g(X)\|_{g}.
\end{equation}
\end{lemma}
\begin{proof}
Here and throughout, we will regard $\tilde{u}_s=\tilde{u}_s(\rho)$ and $\tilde{u}_t=\tilde{u}_t(\rho)$ as vectors in $T_{\tilde{u}(\rho)}\R^m$.  We then observe that the graphical parametrization $\tilde{u}$ and the definition of $\bar{g}$ yield
\begin{equation*}
\| \sin(\theta) \tilde{u}_s+\cos(\theta) \tilde{u}_t\|_{\bar{g}}^2 =1 + {\textstyle \sum_{i=3}^m} \big(\sin(\theta) \tilde{u}_s^i + \cos(\theta)\tilde{u}_t^i \big)^2;
\end{equation*}
combining this with (P\ref{en.P5}) then yields
\begin{equation*}
\|\sin(\theta) \tilde{u}_s + \cos(\theta) \tilde{u}_t\|_{\bar{g}}\geq 1 - \sqrt{2} \cdot 10^{-10}.
\end{equation*}
Thus for any $X\in T_q \R^m$ with $dx^1(X)=0=dx^2(X)$, we find
\begin{align}
\|\pi_{\mathcal{T}}^{\bar{g}} (X)\|_{\bar{g}}&=\sup_{\theta\in [0,2\pi]} \frac{\big| \la X, \sin(\theta) \tilde{u}_s +\cos(\theta) \tilde{u}_t\ra_{\bar{g}}\big|}{\| \sin(\theta) \tilde{u}_s +\cos(\theta) \tilde{u}_t \|_{\bar{g}}}\label{eq:ProjEst}\\
&\leq (1-\sqrt{2}\cdot 10^{-10})^{-1}\big( \big| \la X,  \tilde{u}_s\ra_{\bar{g}}\big| + \big| \la X,  \tilde{u}_t\ra_{\bar{g}}\big|\big)\notag\\
&\leq 10^{-9} \|X\|_{\bar{g}}.\notag
\end{align}
Next, recall that for each $q\in \tilde{u}(\mathcal{D}_{r_0})$ we have $\dist_{g}(0,q)\leq \delta$, so by
so by Lemma \ref{lem:GammaEst}, it follows that for each $Y\in T_q \R^m$, the following estimates hold:
\begin{align*}
(1-10^{-10})\|Y\|_{\bar{g}}\leq \|Y\|_g& \leq (1+10^{-10})\|Y\|_{\bar{g}}\\
\|\pi_{\mathcal{T}}^{g}(X)-\pi_{\mathcal{T}}^{\bar{g}}(X)\|_{\bar{g}}&\leq 10^{-9} \|X\|_{\bar{g}},
\end{align*}
where $\mathcal{T}$ is the plane tangent to $\tilde{u}(\mathcal{D}_{r_0})$ at $q$. Combining these estimates with inequality (\ref{eq:ProjEst}), we find that for each $X\in T_q\R^m$ with $dx^1(X)=0= dx^2(X)$, the following holds:
\begin{align*}
\|\pi_{\mathcal{T}}^{g}(X)\|_{g}&\leq (1+10^{-10})\|\pi_{\mathcal{T}}^{g}(X)\|_{\bar{g}}\notag\\
&\leq (1+10^{-10})\big(\|\pi_{\mathcal{T}}^{g}(X)-\pi_{\mathcal{T}}^{\bar{g}}(X)\|_{\bar{g}} + \|\pi_{\mathcal{T}}^{\bar{g}}(X)\|_{\bar{g}}\big)\\
&\leq (1+10^{-10})\cdot2\cdot 10^{-9}\|X\|_{\bar{g}}\\
&\leq {\textstyle \frac{1}{2}}\|X\|_{g},
\end{align*}
from which we conclude that
\begin{equation}
\|X\|_{g}\leq 2\|\pi_{\mathcal{N}}^g(X)\|_{g},
\end{equation}
which is precisely the desired inequality. This completes the proof of Lemma \ref{lem:ProjectionLemma}.
\end{proof}

We are now ready to prove (P\ref{en.P7}).  Indeed, observe that for any multi-index $\alpha$ with $|\alpha|=2$, the first two components of $D_\alpha\tilde{u}$ vanish; thus regarding $D_{\alpha}\tilde{u}(\rho)\in T_{\tilde{u}(\rho)} \R^m$, we see that Lemma \ref{lem:ProjectionLemma} applies.  For instance, we then have
\begin{equation*}
\|\tilde{u}_{st}\|_{\bar{g}}\leq 2\|\pi_{\mathcal{N}}^g(\tilde{u}_{st})\|_{g} =2\big\|\pi_{\mathcal{N}}^g(\nabla_{\tilde{u}_s}\tilde{u}_t)-\pi_{\mathcal{N}}^g\big(\Gamma(\tilde{u}_s,\tilde{u}_t)\big)\big\|_{g}\leq 10
\end{equation*}
The same estimate holds for the other second derivatives, and thus property (P\ref{en.P7}) follows immediately.

To prove property (P\ref{en.P9}), recall that $\rho\in \mathcal{D}_{r_0}$ guarantees $\dist_g \big(\tilde{u}(\rho),\tilde{u}(0)\big)\leq \delta=10^{-10}$, so by Lemma \ref{lem:GammaEst} it follows that
\begin{equation*}
\big|\la X,Y\ra_g - \la X,Y\ra_{\bar{g}} \big|\leq 10^{-10} \|X\|_{\bar{g}} \|Y\|_{\bar{g}};
\end{equation*}
combining this inequality with (P\ref{en.P5}) then guarantees that $|\gamma_{ij}-\delta_{ij}|\leq 10^{-9}$.  Letting $[\gamma_{ij}]$ and $[\gamma^{ij}]$ denote the symmetric positive definite $2\times 2$ matrices with entries $\gamma_{ij}$ and $\gamma^{ij}$ respectively, it then follows that
\begin{equation}\label{eq:goodParam6}
\|[\gamma_{ij}]\| \leq 1+10^{-8}, \qquad |1-\det[\gamma_{ij}]| \leq 10^{-8},
\quad\text{and}\quad
|\gamma^{ij}-\delta^{ij}| \leq 10^{-8}.
\end{equation}
The first inequality of (\ref{eq:goodParam6}) together with the fact that $[\gamma_{ij}]$ is a positive symmetric definite matrix and $[\gamma_{ij}]^{-1}=[\gamma^{ij}]$ proves (P\ref{en.P9}).

To prove (P\ref{en.P10}), we first note that the third inequality of (\ref{eq:goodParam6}) guarantees that  $\|\gamma^{ij}\|_{C^0(\mathcal{D}_{r_0})}\leq 2$, so it remains to estimate the H\"{o}lder seminorm $[\cdot]_{0,\alpha}$ given by
\begin{equation}\label{eq:HolderSeminorm}
[f]_{0,\alpha}:=\sup_{\substack{\rho,\rho'\in\mathcal{D}_{r_0}\\\rho\neq \rho'}} \frac{|f(\rho)-f(\rho')|}{|\rho-\rho'|^{\alpha}}.
\end{equation}
Recall that $[f]_{0,\alpha}\leq (2r_0)^{1-\alpha} \|D f\|_{C^0(\mathcal{D}_{r_0})}$,  and
$\gamma^{ij}=\pm \gamma_{i'j'}/|\gamma|$ for some $i'$ and $j'$ and $|\gamma|:=\det[\gamma_{ij}]$, so
\begin{align*}
[\gamma^{ij}]_{0,\alpha}&\leq \| D \gamma_{i' j'} \| \cdot \| 1/|\gamma| \| + 2 \|D \gamma_{i' j'}\| \cdot \|\gamma_{i'j'}\| \cdot \| 1/|\gamma|^2 \|\\
&\leq 10 \| D \gamma_{i' j'} \|,
\end{align*}
where $\|\cdot\|$ denotes the $C^0(\mathcal{D}_{r_0})$ norm.  To estimate the right-most term we compute
\begin{align*}
\big|\partial_s (\gamma_{11}) \big| &= \big|\partial_s \la \tilde{u}_s,\tilde{u}_s\ra_g\big| \\
&\leq  \big| g_{ij,k}\tilde{u}_s^i \tilde{u}_s^j \tilde{u}_s^k\big| + \big| (g_{ij}-\delta_{ij}) \tilde{u}_{ss}^i \tilde{u}_{s}^j \big| + \big| (g_{ij}-\delta_{ij})
\tilde{u}_s^i \tilde{u}_{ss}^j \big| + 2\big|\tilde{u}_{ss}^i\tilde{u}_{s}^i\big|\\
&\leq 20,
\end{align*}
where we have made use of properties (P\ref{en.P5}) and (P\ref{en.P7}). The same argument holds for the partial of the $\gamma_{ij}$ with respect to $s$ and $t$, so conclude that $\|D \gamma_{i'j'}\|\leq 20$, and thus
\begin{equation*}
\|\gamma^{ij}\|_{C^{0,\alpha}(\mathcal{D}_{r_0})} \leq  2\cdot 10^3
\end{equation*}
This verifies property (P\ref{en.P10}), and completes the proof of Lemma \ref{lem:uniformLocalGraphs}.
\end{proof}

We are now prepared to state and prove the main result of Section \ref{sec:InteriorCase}, namely Theorem \ref{thm:AprioriDerivataiveBounds}, which guarantees that $J$-curves with a priori bounded curvature, can be graphically parameterized in such a way that all derivatives of the parameterization are bounded a priori.  We now make this precise.

\begin{theorem}[A priori derivative bounds]\label{thm:AprioriDerivataiveBounds}
Fix constants $m\in \mathbb{N}$ with $m\geq 2$, $C>0$, and $\alpha\in (0,1)$.  Then there exists a constant $C'=C'(m,C,\alpha)$ with the following significance.
Let $(M,J,g)$ be a compact almost Hermitian manifold of dimension $2n$, possibly with boundary, let $\mathcal{K}\subset M$ be a compact set for which $(M,g,\mathcal{K})$ is an admissible triple and $\|g\|_{C_g^{m-1,\alpha}(\mathcal{K})}+\|J\|_{C_g^{m-1,\alpha}(\mathcal{K})}\leq C$,  and let $u:S\to M$ be an immersed $J$-curve of Type 0 for which
\begin{equation*}
\sup_{\zeta\in S} \|B_u(\zeta)\|_g \leq 1.
\end{equation*}
Letting $r_0$ be the constant guaranteed by Lemma \ref{lem:uniformLocalGraphs}, we fix $\zeta\in u^{-1}(\mathcal{K})$ and let $\phi:\mathcal{D}_{r_0}\to S$ and $\Phi:\mathcal{B}_{2r_0}\to \R^{2n}$ denote the maps also guaranteed by that
Lemma.  \emph{Then} for $\tilde{u}:=\Phi\circ u \circ \phi$, the following estimate holds.
\begin{equation}
\|\tilde{u}\|_{C^{m,\alpha}(\mathcal{D}_{r_0/2})}\leq C'.\label{eq:EllipticRegularity}
\end{equation}
\end{theorem}
\begin{proof}
We begin by defining the constants $r_k:=\frac{1}{2}(1+\frac{1}{2^k}) r_0$, and note that $r_k\in (\frac{1}{2}r_0,r_0)$ for all positive $k\in \mathbb{N}$.  Note that a consequence of Proposition \ref{prop:GraphicalMinimalSurfaceSystem} is that for
$\tilde{u}:=\Phi\circ u\circ \phi$ and $\mu=3,\ldots,2n$ the $\tilde{u}^\mu$ satisfy the equation $\gamma^{ij}D_{ij}\tilde{u}^\mu = \mathcal{F}^\mu$, where
\begin{equation*}
\gamma_{ij} = \la D_i\tilde{u},D_j\tilde{u}\ra_{\Phi_* g},\qquad \gamma_{i\ell}\gamma^{\ell j} = \delta_i^j,
\end{equation*}
and
\begin{equation*}
\mathcal{F}^\mu:=\gamma^{ij} (D_i \tilde{u}^\alpha) (D_j \tilde{u}^\beta) \big((D_\ell \tilde{u}^\mu)\Gamma_{\alpha \beta}^\ell-(D_\ell \tilde{u}^\mu) Q_{\alpha \beta}^\ell +Q_{\alpha \beta}^\mu-\Gamma_{\alpha \beta}^\mu \big);
\end{equation*}
here the $\Gamma_{\alpha\beta}^\mu$ are the Christoffel symbols associated to geodesic normal coordinates, the component functions $Q_{\alpha\beta}^\mu$ are defined by $Q_{\alpha\beta}^\mu\partial_{x^\mu} = J(\nab_{\partial_{x^\alpha}} J ) \partial_{x^\beta}$, repeated Roman indices indicate a summation from $1$ to $2$, and repeated Greek indices indicate summation from $1$ to $2n$.   Also recall that properties (P\ref{en.P9}) and (P\ref{en.P10}) of Lemma \ref{lem:uniformLocalGraphs} guarantee that
\begin{equation*}
\|\gamma^{ij}\|_{C^{0,\alpha}(\mathcal{D}_{r_0})} \leq 2\cdot 10^3 \qquad\text{and}\qquad \inf_{\substack{\rho\in\mathcal{D}_{r_0}\\\xi\in \R^2}}\gamma^{ij}(\rho)\xi_i\xi_j \geq {\textstyle \frac{1}{4}}|\xi|^2.
\end{equation*}
Furthermore  property (P\ref{en.P5}) guarantees $\|D_i \tilde{u}\|_{C^0(\mathcal{D}_{r_0})}\leq 2$, and property (P\ref{en.P7}) guarantees $\|D_{ij} \tilde{u}\|_{C^0(\mathcal{D}_{r_0})}\leq 10$.  We thus observe that the $C^{0,\alpha}(\mathcal{D}_{r_0})$-norms of the $\mathcal{F}^\mu$ are bounded in terms of $r_0$, $\alpha$, $\|\tilde{u}\|_{C^{1,\alpha}(\mathcal{D}_{r_0})}$, $\|g\|_{C_g^{1,\alpha}(\mathcal{K})}$ and $\|J\|_{C_g^{1,\alpha}(\mathcal{K})}$.  However, $r_0$ is a universal constant, $\alpha$ is fixed by assumption, $\|\tilde{u}\|_{C^{1,\alpha}(\mathcal{D}_{r_0})}$ is bounded in terms of $\alpha$ and $\|\tilde{u}\|_{C^{2}(\mathcal{D}_{r_0})}$, and  $\|g\|_{C_g^{1,\alpha}(\mathcal{K})}+\|J\|_{C_g^{1,\alpha}(\mathcal{K})}<C$  by assumption.
Thus by elliptic regularity (c.f. \cite{GdTn01}, Chapter 6), it follows that there exists a $C_2=C_{2}(\alpha,C)$ such that
\begin{equation}
\|\tilde{u}\|_{C^{2,\alpha}(\mathcal{D}_{r_2})}\leq C_{2}.
\end{equation}
By the usual elliptic bootstrapping argument (i.e. differentiating the equation $\gamma^{ij}D_{ij}\tilde{u}^\mu = \mathcal{F}^\mu$, to show that higher order derivatives of $\tilde{u}^\mu$ solve the same partial differential equation with a different inhomogeneous term), there exist constants $C_{k} = C_{k}\big(\alpha,k,C\big)$ for $k=2,\ldots,m$ so that
\begin{equation*}
\|\tilde{u}\|_{C^{k,\alpha}(\mathcal{D}_{r_k})}\leq C_{k}.
\end{equation*}
Since this estimate holds for $k=m$, we let $C':= C_m$, and the proof of Theorem \ref{thm:AprioriDerivataiveBounds} is complete.
\end{proof}

\subsection{Lagrangian boundary case}\label{sec:BoundaryCase} In this section we prove results analogous to those of Section \ref{sec:InteriorCase} in the case that the $J$-curves of interest have a partial Lagrangian boundary condition, or more precisely the $J$-curves are Type 1 immersions.

Here and throughout, the $4$-tuple $(M,J,g,L)$ will consist of a smooth compact almost Hermitian manifold $(M,J,g)$ of dimension $2n$ (possibly with boundary), and a compact embedded totally geodesic Lagrangian submanifold $L\subset M$ with $\partial L = L \cap \partial M$.

\begin{definition}[$L$-adapted geodesic normal coordinates]\label{def:LAdaptGeoNormalCoords}
Let $(M,J,g,L)$ be as above, and fix $p\in L\setminus \partial L$.  Define the map
\begin{equation*}
\exp_p^L:T_p L \times T_p L^\bot\simeq T_p M \to M\qquad\text{by}\qquad \exp_p^L(\ell,\nu):=\exp_{\exp_p(\ell)}(\tilde{\nu}),
\end{equation*}
where $\tilde{\nu}_{\exp_p(\ell)}\in T_{\exp_p(\ell)} M$ is the vector obtained by parallel transport of $\nu_p$ along the path $t\mapsto \exp_p(t\ell)$.
Extend any orthonormal frame $\{f_1^*,\ldots,f_{n}^*\}\subset T_p^* L$ to an orthonormal frame $\{e_1^*,\ldots,e_{2n}^*\}\subset T_p^* M$ by defining $e_{2k-1}^*:=f_k^*$ and $e_{2k}^*:= -e_{2k-1}^*\circ J$ for $k=1,\ldots,n$, and we define the associated \emph{$L$-adapted geodesic normal coordinates}  $(x^1,\ldots,x^{2n})$ by the following:
\begin{equation*}
x^k(q):=e_k^*\big((\exp_p^L)^{-1}(q) \big).
\end{equation*}
\end{definition}

Just as the injectivity radius is associated to the $\exp$ map, we also define an $L$-adapted injectivity radius associated to the $\exp^L$ map by the following.

\begin{definition}[$\inj^L$]\label{def:injL}
For each $p\in L$, $\epsilon>0$, define $\mathcal{B}_\epsilon:=\{X\in T_p L: \|X\|< \epsilon\}$ and $\mathcal{B}_\epsilon^\bot:=\{X\in T_p L^\bot: \|X\|< \epsilon\}$.  We then define the $L$-adapted injectivity radius at $p$ to be the following:
\begin{equation*}
\inj^L(p):=\sup \{\epsilon : \;conditions  \;\;(C\ref{en.C1}) \;\;\text{and}\;\; (C\ref{en.C2})\;\;\text{hold}  \};
\end{equation*}
where the conditions (C\ref{en.C1}) and (C\ref{en.C2}) are given by
\begin{enumerate}[({C}1)]
\item $\exp_p^L:\mathcal{B}_\epsilon\times\mathcal{B}_\epsilon^\bot \to M\text{ is a diffeomorphism with its image}$ \label{en.C1}
\item $\mathcal{B}_\epsilon\times \{0\}\supset (\exp_p^L)^{-1}\big(  L\cap \exp_p^L\big(\mathcal{B}_\epsilon\times\mathcal{B}_\epsilon^\bot\big)\big)$ \label{en.C2}.
\end{enumerate}
Similarly, for a set $\mathcal{K}\subset L$ define
\begin{equation}\label{eq:LAdaptInjRad}
\inj^L(\mathcal{K}):=\inf_{p\in \mathcal{K}}\inj^L(p).
\end{equation}
\end{definition}
Note that when comparing the $L$-adapted injectivity radius to the usual injectivity radius, the first condition (C\ref{en.C1}) is an obvious adaptation from the usual definition, whereas the second condition (C\ref{en.C2}) guarantees that if $p\in L$ and $\epsilon< \inj^L(p)$, then $L\cap \exp_p^L(\mathcal{B}_\epsilon\times \mathcal{B}_\epsilon^\bot)$ is connected. Moreover, consider the torus $\mathbb{S}^1\times \mathbb{S}^1$ with the standard structures, and consider the Lagrangians $L_\delta: = \{0\}\times \mathbb{S}^1\cup \{\delta\}\times \mathbb{S}^1$ for each $\delta>0$;  for cases such as this in what follows,  it will be important that $\inj^{L_\delta}(L_\delta)\to 0$ as $\delta \to 0$.

Also note that for $\mathcal{B}_\epsilon\subset T_p L$ and $\mathcal{B}_\epsilon^\bot\subset T_pL^\bot$ as above, and for $\mathcal{B}_\epsilon^M(p)\subset M$ the metric ball of radius $\epsilon>0$, and for each sufficiently small $\epsilon>0$, it is straightforward to verify that $\mathcal{B}_\epsilon^M(p)\subset \exp_p^L\big(\mathcal{B}_\epsilon \times \mathcal{B}_\epsilon^\bot\big)$.

Since $L$-adapted geodesic normal coordinates are a bit non-standard, we take a moment to discuss properties of the metric components $g_{ij}$ in these coordinates.  We make this precise with Lemma \ref{lem:LAdaptedGeoCoords} below.

\begin{lemma}\label{lem:LAdaptedGeoCoords}
Let $(M,J,g,L)$ be as above.  Then in local $L$-adapted geodesic normal coordinates $\Phi=(x^1,\ldots,x^{2n})$ centered at $p\in L\setminus \partial L$, the following hold locally.
\begin{enumerate}[({L}1)]
\item $L=\cap_{k=1}^n \{x^{2k}=0\}$.\label{en.L1}
\item $TL = \Span\big(\{\partial_{x^{2k-1}}\}_{k=1}^n\big)\big|_{L}$ and $TL^\bot = \Span\big(\{\partial_{x^{2k}}\}_{k=1}^n\big)\big|_L$.\label{en.L2}
\item For any \emph{odd} $i,j\in 1, \ldots, 2n-1$ the components of the metric satisfy the following along $L$\label{en.L3}
\begin{equation*}
\left(
\begin{array}{cc}
g_{ij} & g_{i,j+1}\\
g_{i+1, j} & g_{i+1, j+1}
\end{array}
\right)
=
\left(
\begin{array}{cc}
g_{ij} & 0\\
0 & 1
\end{array}
\right).
\end{equation*}
\item At $p=\Phi^{-1}(0)$, $g_{ij}=\delta_{ij}$, $\partial_{x^k}g_{ij}=0$, and hence $\Gamma_{ij}^k(p)=0$, for all $i,j,k\in \{1,\ldots,2n\}$.\label{en.L4}
\item The Christoffel symbols satisfy $\Gamma_{o_i o_j}^{e_k}(q) = 0= \Gamma_{o_i e_k}^{o_j}(q)$, whenever $q\in L$, $o_i,o_j\in 1,\ldots,2n-1$ are odd, and $e_k\in 2,\ldots,2n$ is even.\label{en.L5}
\end{enumerate}
\end{lemma}
\begin{proof}
The validity of (L\ref{en.L1}) follows from the definition. Consequently, along $L$, the vector fields $\{\partial_{x^{2k-1}}\}_{k=1}^n$ form a basis for the fibers of $TL$, which proves the first half of (L\ref{en.L2}).  Next observe that since $L$ is totally geodesic, it follows that for any sections $\tau_1,\tau_2\in \Gamma(TL)$ and $\nu\in \Gamma(TL^\bot)$ we have $\nab_{\tau_1} \tau_2 \in \Gamma(TL)$ and $\nab_{\tau_1}\nu\in \Gamma(TL^\bot)$. Consequently for $(\tau,\nu)\in T_p L\times T_p L^\bot$, we have
\begin{equation*}
\exp_p^L(\tau,\nu) = \exp_{\exp_p(\tau)}(\tilde{\nu})
\end{equation*}
where $\tilde{\nu}\in T_{\exp_p(\tau)}L^\bot$ (instead of just $\tilde{\nu}\in T_{\exp_p(\tau)}M$).  We conclude that (L\ref{en.L3}) and the second half of  (L\ref{en.L2}) hold.  It is clear from the definition that $g_{ij}(p)=\delta_{ij}$.  We now claim that $\Gamma_{ij}^k(p)=0$ whenever $i$ and $j$ are both even or both odd. To prove this, we will consider the former case; the latter follows by the same argument.  Recall that in $L$-adapted geodesic normal coordinates,  the paths
\begin{equation*}
\gamma(t) = (0,a_2 t, 0, a_4 t,\ldots, 0,a_{2n} t)
\end{equation*}
are geodesics.  In other words,
\begin{equation*}
0 = \nab_{\dot{\gamma}}\dot{\gamma} = a_i a_j\Gamma_{ij}^k\big(\gamma(t)\big) \partial_{x^k}.
\end{equation*}
Thus $a_i a_j \Gamma_{ij}^k(0)=0$ whenever $a_{2k-1}=0$ for all $k=1,\ldots,n$.  By considering $a_\ell = \delta_{i\ell}$ with $i$ even, one quickly concludes that $\Gamma_{ij}^k(0)=0$ whenever $i=j$ and $i$ is even.  By considering $a_\ell = \delta_{i\ell} + \delta_{j\ell}$ with $i$ and $j$ even, and recalling that $\nab$ is torsion free, one deduces that $\Gamma_{ij}^k(0)=0$ whenever $i$ and $j$ are even; the same argument holds when $i$ and $j$ are both odd. Observe that we have just verified that $\nab_{\partial_{x^i}}\partial_{x^j}\big|_p=0$ whenever $i$ and $j$ are both even or both odd. Also by construction, the $\partial_{x^{2k}}$ are parallel along geodesics in $L$ emanating from $p$, so $\Gamma_{ij}^k(0)\partial_{x^k}=\nab_{\partial_{x^i}}\partial_{x^j}\big|_p = 0$ whenever $i$ is odd and $j$ is even.  Since $\nab$ is torsion free, one concludes that $\Gamma_{ij}^k(0)\partial_{x^k}=\nab_{\partial_{x^i}}\partial_{x^j}\big|_p = 0$ for all $i,j=1,\ldots,2n$. Property (L\ref{en.L4}) then follows from the computation
\begin{equation*}
\partial_{x^k}g_{ij}(p)= \la \nab_{\partial_x{^k}}\partial_{x^i}\big|_p, \partial_{x^j}\ra_g+ \la \partial_{x^i},\nab_{\partial_x{^k}} \partial_{x^j}\big|_p\ra_g=0.
\end{equation*}
Lastly, note that property (L\ref{en.L3}) and the fact that $L$ is totally geodesic guarantee property (L\ref{en.L5}).
\end{proof}

\begin{definition}[$\|T\|_{C_{g,L}^k(\mathcal{K})}$]\label{def:C1BoundOnG2}
Let $(M,J,g,L)$ be as above and let $T$ be a tensor on $M$. For each point $p\in L\setminus \partial L$ we define the bi-ball
\begin{equation*}
\mathcal{B}^2:=\exp_p^L\big(\mathcal{B}_{3\inj^L(p)/4}\times\mathcal{B}_{3\inj^L(p)/4}^\bot\big),
\end{equation*}
and equip it with $L$-adapted geodesic normal coordinates $(x^1,\ldots,x^{2n})$ centered at $p$. For multi-indices $I_1$ and $I_2$, we let $T_{I_1}^{I_2}$ denote the components of $T$ in the coordinates $(x^1,\ldots,x^{2n})$, and then define
\begin{equation*}
[T]_{C_{g,L}^k(p)}:=\sup_{q\in\mathcal{B}^2} \Big(\sum_{|\alpha|=k}\sum_{I_1,I_2}|D_\alpha T_{I_1}^{I_2}(q)|^2\Big)^{1/2}\quad\text{and}\quad \|T\|_{C_{g,L}^k(p)}:= \sum_{j=0}^k [T]_{C_{g,L}^j(p)} .
\end{equation*}
Similarly for each $\alpha\in (0,1]$ we define
\begin{equation*}
\|T\|_{C_{g,L}^{k,\alpha}(p)}:= \|T\|_{C_{g,L}^k(p)} + \sup_{\substack{q_1,q_2\in\mathcal{B}^2 \\ q_1\neq q_2}}\Big(\sum_{|\alpha|=k}\sum_{I_1,I_2}\frac{|D_\alpha T_{I_1}^{I_2}(q_1)-D_\alpha T_{I_1}^{I_2}(q_2)|^2}{|q_1-q_2|^{2\alpha}}\Big)^{1/2},
\end{equation*}
where $|q_1-q_2|$ is the distance between $q_1$ and $ q_2$ computed with respect to the flat metric $dx^i\otimes dx^i$. More generally, for any subset $\mathcal{K} \subset L\setminus \partial L$, we repeat the above construction for each $p\in \mathcal{K}$, to define
\begin{equation*}
[T]_{C_{g,L}^k(\mathcal{K})}:= \sup_{p\in \mathcal{K}} [T]_{C_{g,L}^k(p)}\qquad \|T\|_{C_{g,L}^k(\mathcal{K})} := \sup_{p\in \mathcal{K}} \|T\|_{C_{g,L}^k(p)},\quad\text{and}
\end{equation*}
\begin{equation*}
\|T\|_{C_{g,L}^{k,\alpha}(\mathcal{K})}:=\sup_{p\in \mathcal{K}} \|T\|_{C_{g,L}^{k,\alpha}(p)}.
\end{equation*}
\end{definition}

\begin{definition}[$L$-admissible]\label{def:LAdmissible}
The $5$-tuple $(M,J,g,L,\mathcal{K})$ is said to be \emph{$L$-admissible} provided the following hold.
The triple $(M,J,g)$ is a compact almost Hermitian manifold of dimension $2n$ (possibly with boundary), $L\subset M$ is a compact embedded totally geodesic Lagrangian submanifold with $\partial L = L \cap \partial M$, the subset  $\mathcal{K}\subset M$ is compact, $(M,g,\mathcal{K})$ is an admissible triple, and
\begin{equation*}
10(2n)^2[g]_{C_{g,L}^2(L\cap \mathcal{K})} \leq 1\qquad\text{and}\qquad \inf_{q\in L\cap \mathcal{K}}\inj^L(q)\geq 2.
\end{equation*}
\end{definition}

As in Section \ref{sec:InteriorCase}, it will be convenient for later computations to have the following elementary estimates at our disposal.  The proof of Lemma \ref{lem:GammaEst2} is both elementary and identical to that of Lemma \ref{lem:GammaEst}.

\begin{lemma}\label{lem:GammaEst2}
Let $(M,J,g,L,\mathcal{K})$ be an $L$-admissible $5$-tuple, $p\in L\cap \mathcal{K}$ a point, $(x^1,\ldots,x^{2n})$ $L$-adapted geodesic normal coordinates centered at $p$, $g_{ij}$ the components of $g$ in these coordinates which satisfy  $g_{i\ell} g^{\ell j} = \delta_i^j$, and $\Gamma_{ij}^k$ be the Christoffel symbols of the Levi-Civita connection associated to $g$ in these coordinates.  Then for all points $q\in M$ for which $|q|:=\dist(p,q)\leq 1$, the following inequalities hold.
\begin{align*}
&{\textstyle\sum_{ij=1}^{2n}} |g_{ij}(q)-\delta_{ij}|^2\leq |q|^4,
&&{\textstyle\sum_{ijk=1}^{2n}} |\Gamma_{ij}^k(q)|^2 \leq |q|^2\\
&{\textstyle\sum_{ij=1}^{2n}} |g^{ij}(q)-\delta^{ij}|^2\leq |q|^4
&&\sup_{p\in \mathcal{K}} |K_{sec}(p)| \leq 1.
\end{align*}
Furthermore, for $\bar{g}:=dx^i\otimes dx^i$, the following inequalities hold.
\begin{align*}
\big| \la Y,Z\ra_g - \la Y,Z\ra_{\bar{g}} \big|  &\leq |q|^2 \|Y\|_{\bar{g}}\|Z\|_{\bar{g}}\\
\big(1-|q|^2\big)\|Y\|_{\bar{g}}^2\;\;\leq\;\;  &\|Y\|_g^2 \;\;\leq\;\; \big(1+|q|^2\big)\|Y\|_{\bar{g}}^2.
\end{align*}
Lastly, for $\mathcal{V}\subset T_q M$, a vector space of dimension $k$, we let $\pi_\mathcal{V}^g$ and $\pi_\mathcal{V}^{\bar{g}}$ respectively denote the $g$ and $\bar{g}$ orthogonal projections onto $\mathcal{V}$. Then for $q$ as above, and
$Y,Z\in T_qM$, the following inequalities also hold.
\begin{equation*}
\big\| \pi_\mathcal{V}^g (Y) - \pi_\mathcal{V}^{\bar{g}}(Y)\big\|_{\bar{g}}\leq |q|^2\|Y\|_{\bar{g}}.
\end{equation*}
\end{lemma}

In what follows, it will be convenient to have the following result at our disposal.

\begin{lemma}\label{lem:ShortPathsHaveLargeCurvature}
Let $(M,J,g,L,\mathcal{K})$ be an $L$-admissible $5$-tuple, $p\in L\cap \mathcal{K}$ a point, and $\gamma:[0,\ell]\to \mathcal{B}_{1/2}(p)\subset M$ a $C^2$ path for which $\gamma(0),\gamma(\ell)\in L$,  $\|\dot{\gamma}\|\equiv 1$, and $\dot{\gamma}(0)\in TL^\bot$.  If $\ell< 1/2$, then
\begin{equation*}
\sup_{t\in [0,\ell]}\kappa_g(t) \geq \frac{1}{2\ell}-1,
\end{equation*}
where $\kappa_g(t)$ is the geodesic curvature of $\gamma$ at the point $t$.
\end{lemma}
\begin{proof}
We begin by recalling that the geodesic curvature is given by $\kappa_g(t) = \|\big(\nab_{\dot{\gamma}(t)}\dot{\gamma}(t)\big)^\bot\|$, and $\big(\nab_{\dot{\gamma}(t)}\dot{\gamma}(t)\big)^\bot=\nab_{\dot{\gamma}(t)}\dot{\gamma}(t)$; this follows since $\|\dot{\gamma}\|\equiv 1$.  Next we let $(x^1,\ldots,x^{2n})$ denote $L$-adapted geodesic normal coordinates centered at $p$.  Since $\dist\big(p,\gamma(t)\big)< 1/2$ for all $t\in[0,\ell]$, we see that $\gamma$ has image in the neighborhood on which the coordinates area defined.  In local coordinates, we then find
\begin{equation*}
\nab_{\dot{\gamma}(t)}\dot{\gamma}(t) = \ddot{\gamma}(t) + \Gamma(\dot{\gamma},\dot{\gamma})
\end{equation*}
where $\Gamma$ is the locally defined bilinear operator associated the Christoffel symbols of the Levi-Civita connection.  As a consequence of Lemma \ref{lem:GammaEst2} and the fact that $\dot{\gamma}\equiv 1$, we see that $\|\Gamma(\dot{\gamma},\dot{\gamma})\|\leq 1$.

As a consequence of Lemma \ref{lem:LAdaptedGeoCoords}, we may assume with out loss of generality that the coordinates are chosen so that $\frac{d}{dt}\big|_{t=0}\gamma_2(t)=\gamma_2'(0)=1$, where $\gamma_2(t):=x^2\big(\gamma(t)\big)$.  Since $\gamma(\ell)\in L$, and since $\inj^L(p)\geq 2$ (this follows from admissibility of the given 5-tuple) we see that $\gamma_2(\ell)=0$.  Furthermore, there exist $\ell'\in (0,\ell)$ such that $\gamma_2'(\ell')\leq 0$.  The mean value inequality then guarantees that there exists a $\ell''\in (0,\ell')$ such that $|\gamma_2''(\ell'')|\geq \ell^{-1}$.  The desired estimate then follows from this inequality and another application of Lemma \ref{lem:GammaEst2}.
\end{proof}

%%%%%%%%%%%%%%%%%%%%%%%%%%%%%%%%%%%%%%%%%%%%%%%%%%%%%%%%%%%%%%%%%%%%%%%%%%%%%%%%%%%%
%%%% Similar result with Lagrangian boundary
%%%%%%%%%%%%%%%%%%%%%%%%%%%%%%%%%%%%%%%%%%%%%%%%%%%%%%%%%%%%%%%%%%%%%%%%%%%%%%%%%%%%

In what follows, it will be convenient to have a notion of an immersion being ``adapted'' to totally geodesic Lagrangian.  We make this precise with Definition \ref{def:TotGeoAdapt} below.
\begin{definition}\label{def:TotGeoAdapt}
Let $(M,J,g,L)$ be as above, and let $u:S\to M$ be smooth immersion of Type 1 (see Definition \ref{def:Type1Map}).  We say that $(u,S)$ is \emph{adapted to $(M,J,g,L)$} provided that any locally defined, continuous, and outward pointing $u^*g$-normal vector field $\nu\in \Gamma(TS)\big|_{\partial_1 S}$ satisfies $u_*\nu \bot_g TL$, or equivalently $u_*\nu \in \Gamma (TL^\bot)\big|_{\partial_1 S}$.
\end{definition}

Immersions $(u,S)$ which are adapted to totally geodesic Lagrangian submanifolds as above have a particularly useful property which we state and prove at present.

\begin{lemma}[geodesic $\partial_1$-boundary]\label{lem:geodesicBoundary}
Let $(M,J,g,L)$ and $(u,S)$ be as in Definition \ref{def:TotGeoAdapt}.  Suppose further that $\alpha:(-\epsilon,\epsilon)\to \partial_1 S$ is a $u^*g$-unit speed path.  Then $\alpha$ is a $u^*g$-geodesic.
\end{lemma}
\begin{proof}
Let $\nu$ be continuous unit vector field along $\alpha$ for which $\la \nu,\dot{\alpha}\ra_{u^*g}\equiv 0$.  To show that $\alpha$ is a  geodesic it is sufficient to show that
$\la \nab_{\dot{\alpha}}^S \dot{\alpha},\nu\ra_{u^*g}\equiv0$, where $\nab^S$ is the Levi-Civita connection on $S$ associated to $u^*g$ (in other words $\nab^S = \pi_{\mathcal{T}}^g\circ \nab$ ).  Since $(u,S)$ is adapted to $L$ it follows that $u_*\nu\bot L$.
Thus to show $\la \nab_{\dot{\alpha}}^S \dot{\alpha},\nu\ra_{u^*g}=0$, it is sufficient to show that $\pi_{TL^\bot}^g(\nab_{\dot{\beta}} \dot{\beta})=0$ where $\beta = u\circ \alpha$, and $TL^\bot$ is the normal bundle to $L\subset M$.  However, the second fundamental form $B_{L}:TL\otimes TL \to TL^\bot$ of $L\subset M$ is defined to be $\pi_{TL^\bot}^g( \nab_X Y)$, and it vanishes identically because $L$ is totally geodesic. Consequently $\pi_{TL^\bot}^g(\nab_{\dot{\beta}} \dot{\beta})=B_{L}(\dot{\beta},\dot{\beta})=0$, and we immediately conclude that $\partial_1 S$ is a $u^*g$-geodesic.  This completes the proof of Lemma \ref{lem:geodesicBoundary}.
\end{proof}

For the following lemma, it will be convenient to define
\begin{equation}\label{eq:DefOfHr}
\mathcal{H}_r:=\{(s,t)\in \R^2: |s|<r\;\;\text{and}\;\; 0\leq t< r\}.
\end{equation}

\begin{lemma}[Uniform Local Graphs with Lagrangian Boundary]\label{lem:uniformLocalGraphs2}
There exists a constant $r_0>0$ with the following significance.  Let $(M,J,g,L,\mathcal{K})$ be an admissible 5-tuple with $\dim M = 2n$, and suppose $u:S\to M$ is a Type 1 immersion which is adapted to $(M,J,g,L)$ and which satisfies
\begin{equation*}
\sup_{\zeta\in S} \|B_u(\zeta)\|_g \leq 1.
\end{equation*}
Then for each $\zeta\in u^{-1}(L\cap \mathcal{K})\cap \partial_1 S$, there exists a map $\phi:\mathcal{H}_{r_0}\to S$ and $L$-adapted geodesic normal coordinates\footnote{See Definition \ref{def:LAdaptGeoNormalCoords}.} $\Phi:\exp_{u(\zeta)}^L\big(\mathcal{B}_{2r_0}\times \mathcal{B}_{2r_0}^\bot\big)\to\mathbb{R}^{2n}$  with the following properties.
\begin{enumerate}[({PL.}1)]
\item $\Phi(L) \subset \big(\R\times \{0\}\big)^n$\label{en.PL1}
\item $\phi(0)=\zeta$ and $\phi(s,0)\subset \partial_1 S$\label{en.PL2}
\item $\tilde{u}(s,t):=\Phi\circ u \circ \phi(s,t)=\big(s,t,\tilde{u}^3(s,t),\ldots,\tilde{u}^{2n}(s,t)\big)$\label{en.PL3}
\item $\tilde{u}^i(0,0)=0$ for $i=1,\ldots, m$\label{en.PL4}
\item $D_\alpha\tilde{u}^i(0,0)=0$ whenever $|\alpha|=1$ and $i=3,\ldots,2n$\label{en.PL5}
\item $\tilde{u}(s,0)\in \big(\R\times\{0\}\big)^n$ and $\tilde{u}_t(s,0)\in \big(\{0\}\times \R\big)^n$\label{en.PL6}
\item $\sum_{|\alpha|=1}\sum_{i=3}^{2n}\|D_\alpha \tilde{u}^i\|_{C^0(\mathcal{H}_{r_0})}^2 \leq 10^{-20}$\label{en.PL7}
\item For Euclidian coordinates $\rho=(s,t)$, on $\mathcal{H}_{r_0}$, we have\label{en.PL8}
\begin{equation*}
{\textstyle \frac{1}{2}}|\rho|\leq \dist_{u^*g}(0,\rho) \leq 2 |\rho|
\end{equation*}
\item For $|\alpha|=2$ and $i=1,\ldots,2n$, $\|D_\alpha \tilde{u}^i\|_{C^0(\mathcal{H}_{r_0})} \leq 10 $ \label{en.PL9}
\item With $\rho$ as above, we have\label{en.PL10}
\begin{equation*}
    {\textstyle \frac{1}{2}}|\rho|\leq \dist_{g}\big(u(\phi(\rho)),u(\phi(0))\big)\leq 2|\rho|.
\end{equation*}
\end{enumerate}
\noindent Furthermore, letting subscripts $s$ and $t$ denote partial derivatives, and denoting $\gamma_{11}=\la u_s,u_s\ra_g$, $\gamma_{12}=\gamma_{21}=\la u_s,u_t\ra_g$, $\gamma_{22}=\la u_t,u_t\ra_g$, and $\gamma_{ik}\gamma^{k j}=\delta_i^j$, the following inequalities also hold.
\begin{enumerate}[({PL.}1)]
\setcounter{enumi}{10}
\item $\frac{1}{4}|\xi|^2\leq \inf_{\rho\in\mathcal{H}_{r_0}} \gamma^{ij}\xi_i\xi_j$\label{en.PL11}
\item $\|\gamma^{ij}\|_{C^{0,\alpha}(\mathcal{H}_{r_0})}\leq 2\cdot 10^3$ \label{en.PL12}
\end{enumerate}
\end{lemma}
\begin{proof}
After a change in the construction of $\phi$ and choice of $\Phi$, the proof of
Lemma \ref{lem:uniformLocalGraphs2} is essentially no different than that of Lemma \ref{lem:uniformLocalGraphs}.  As such, we will only discuss the necessary modifications, and then refer to the proof of Lemma \ref{lem:uniformLocalGraphs} for verification of the appropriate estimates.

Begin by defining $\delta:=10^{-10}$. Next, since $L$ is a totally geodesic Lagrangian, it follows that there exist $L$-adapted geodesic normal coordinates $\Phi$ such that $\Phi\big(u(\zeta)\big) =0$ and (PL.\ref{en.PL1}) holds. Furthermore, since $(u,S)$ is an immersion (in particular at $\zeta$) and it is adapted to $L$, it follows that for $u^*g$-unit vectors $\tau,\nu \in T_\zeta S$ for which $\tau\in T_\zeta \partial_1 S$, $\la \nu, \tau\ra_{u^*g}=0$, and $\nu$ ``inward pointing'', we see that $\Phi$ can be chosen so that
\begin{equation*}
(\Phi \circ u)_*\tau = (1,0,\ldots,0)\qquad\text{and}\qquad(\Phi\circ u)_* \nu = (0,1,0,\ldots,0);
\end{equation*}
this will guarantee that (PL.\ref{en.PL5}) holds for any parameterization $\phi$.  To construct the map $\phi$, we extend $\nu$ to be a continuous ``inward pointing'' $u^*g$-unit vector field along $\partial_1 S$ which is $u^*g$-orthogonal to $\partial_1 S$, and define an exponential-type map
\begin{equation}\label{eq:defTildeVarPhi}
\tilde{\varphi}:\mathcal{H}_{\delta}\to S \qquad\text{by}\qquad \tilde{\varphi}(s,t):= \exp_{\exp(s\tau)}(t\nu).
\end{equation}
We now claim (and shall prove) that $\tilde{\varphi}$ is well defined and is an immersion.  Indeed, as a consequence of Lemma \ref{lem:geodesicBoundary} it follows that $\partial_1 S$ consists of $u^*g$-geodesics; furthermore since $(u,S)$ is adapted to $(M,J,g,L)$ it follows that the set of non-smooth boundary points of $S$ is precisely $\partial_0 S \cap \partial_1 S\subset \partial_0 S$ and by assumption $\dist_{u^*g}(\partial_0 S, \zeta)\geq \dist_g\big(\partial M,u(\zeta)\big)\geq 2$. Consequently $\tilde{\varphi}\big|_{(-\delta,\delta)\times\{0\}} \to \partial_1 S$ is well defined.  Thus to show $\tilde{\varphi}$ is well-defined on $\mathcal{H}_\delta$, it is sufficient to show that for $\zeta'\in \partial_1 S$ with $\dist_{u^*g}(\zeta,\zeta')< \delta$, the map $\alpha(t):=\exp_{\zeta'}(t \nu)$ is well defined for $t\in [0,\delta)$.  To that end, suppose not; then there exists a smooth path $\beta:=u\circ\alpha:[0,\delta']\to M$ which must satisfy the following conditions: $\|\dot{\beta}\|\equiv 1$, $\beta(t)\in \mathcal{B}_{1/2}\big(u(\zeta)\big)\subset\subset M$ for all $t\in [0,\delta']$, $\beta(0),\beta(\delta')\in L$, $\dot{\beta}(0)\in TL^\bot$, $\delta'\leq 10^{-10}\leq 1/2$, and $\|\nab_{\dot{\beta}}\dot{\beta}\|=\|B_u(\dot{\beta},\dot{\beta})\|\leq 1$.  However, by Lemma \ref{lem:ShortPathsHaveLargeCurvature} we see that
\begin{equation*}
2\leq \sup_{t\in [0,\delta']} \kappa_g(t) = \sup_{t\in [0,\delta']}\|\nab_{\dot{\beta}(t)}\dot{\beta}(t)\| =\sup_{t\in[0,\delta']} \|B_u\big(\dot{\beta}(t),\dot{\beta}(t)\big)\|\leq 1.
\end{equation*}
From this contradiction we conclude that $\tilde{\varphi}:\mathcal{H}_\delta\to S$ given in (\ref{eq:defTildeVarPhi}) is well defined.  The proof that $\tilde{\varphi}$ is an immersion is identical to the proof   that $\exp_\zeta:\{X\in T_\zeta S: \|X\|_{u^*g}\leq \delta\}\to S$ is an immersion in Lemma \ref{lem:uniformLocalGraphs}.

Next, define the coordinate projection $\pr(x^1,\ldots,x^{2n})=(x^1,x^2)$, and the map
\begin{equation*}
\varphi:\mathcal{H}_\delta \to \R^2\qquad\text{by}\qquad\varphi:=\pr\circ \Phi\circ u\circ \tilde{\varphi}.
\end{equation*}
At this point, we observe that since $\tilde{\varphi}:\R\times\{0\}\to \partial_1 S$, $u(\partial_1 S) \subset L$, and $\Phi(L) \subset \big(\R\times \{0\}\big)^n$, it follows that $\varphi\big((-\delta,\delta)\times \{0\}\big)\subset \R\times \{0\}$.  Furthermore, by our choice of $\Phi$, it follows that $(T_0\varphi) (0,1) = (0,1)$, and thus $\varphi:\mathcal{H}_\delta\to\mathcal{H}_\infty$. Thus by defining $\phi:= \tilde{\varphi}\circ\varphi^{-1}$, we see that with the exception of proving (PL.\ref{en.PL6}), and some obvious modifications from working in ``polar'' coordinates to ``bi-polar'' coordinates, the remainder of the proof Lemma \ref{lem:uniformLocalGraphs2} is essentially identical to that of Lemma \ref{lem:uniformLocalGraphs}, with references to $\|g\|_{C_g^k(\mathcal{K})}$, Definition \ref{def:C1BoundOnG}, and Lemma \ref{lem:GammaEst} respectively replaced with $\|g\|_{C_{g,L}^k(\mathcal{K})}$, Definition \ref{def:C1BoundOnG2}, and Lemma \ref{lem:GammaEst2}.

Thus to finish the proof of Lemma \ref{lem:uniformLocalGraphs2}, it is sufficient to prove (PL.\ref{en.PL6}). To that end, we first note that we have already established that $\Phi(L)\subset \big(\R \times \{0\}\big)^n$, so it is sufficient to prove that $\tilde{u}_t(s,0)\in \big(\{0\}\times \R\big)^n$. Next we recall three important facts.
\begin{enumerate}
\item The immersion $\tilde{\varphi}:(\mathcal{H}_\delta,ds^2+dt^2)\to (S,u^*g)$ is isometric along $\{t=0\}$.
\item For vectors $\nu\in TS\big|_{\partial_1 S}$ for which $\nu\bot_{u^*g} \partial_1 S$ we have $u_*\nu \bot_g L$.
\item $\Phi(L)\subset \big(\R\times\{0\}\big)^n$, and along $\Phi(L)$ the metric $\Phi_* g$ expressed in coordinates as a block matrix with $2\times 2$ entries each of the form
 $\left(\begin{smallmatrix} *&0\\0&1\end{smallmatrix}\right)$.
\end{enumerate}
Note that the second part of the third fact was established in Lemma \ref{lem:LAdaptedGeoCoords}. We can now conclude that $(\Phi\circ u\circ \tilde{\varphi})_t(s,0)\in \big(\{0\}\times \R\big)^n$. Since $\varphi = \pr\circ\Phi\circ u\circ \tilde{\varphi}$, it then follows that $\varphi_t(s,0)\in \{0\}\times \R$.  Since $\phi=\tilde{\varphi}\circ\varphi^{-1}$ and $\tilde{u}=\Phi\circ u\circ \phi$, it follows from the above three facts that $\tilde{u}_t(s,0)\in \big(\{0\}\times \R\big)^n$.  This completes the proof of (PL.\ref{en.PL6}), as well as the proof of Lemma \ref{lem:uniformLocalGraphs2}.
\end{proof}

\begin{theorem}[A priori derivative bounds]\label{thm:AprioriDerivataiveBounds2}
Fix constants $m\in \mathbb{N}$, $C>0$, and $\alpha\in (0,1)$.  Then there exists a constant $C'=C'(m,C,\alpha)$ with the following significance.
Let $(M,J,g)$ be a compact almost Hermitian manifold of dimension $2n$ (possibly with boundary), let $L\subset M$ be a totally geodesic Lagrangian submanifold with $\partial L = L\cap \partial M$, and let $\mathcal{K}\subset M$ be a compact set for which $(M,J,g,L,\mathcal{K})$ is an admissible 5-tuple.  Suppose further that
\begin{equation*}
\|g\|_{C_{g,L}^{m-1,\alpha}(L\cap\mathcal{K})}+\|J\|_{C_{g,L}^{m-1,\alpha}(L\cap\mathcal{K})}\leq C,
\end{equation*}
and let $u:S\to M$ be an immersed $J$-curve of Type 1 for which
\begin{equation*}
\sup_{\zeta\in S} \|B_u(\zeta)\|_g \leq 1.
\end{equation*}
Letting $r_0$ be the constant guaranteed by Lemma \ref{lem:uniformLocalGraphs2}, we fix $\zeta\in u^{-1}(L\cap\mathcal{K})$ and let $\phi:\mathcal{H}_{r_0}\to S$ and $\Phi:\mathcal{B}_{2r_0}\times\mathcal{B}_{2r_0}^\bot\to \R^{2n}$ denote the maps also guaranteed by that
Lemma.  \emph{Then} for $\tilde{u}:=\Phi\circ u \circ \phi$, the following estimate holds.
\begin{equation}
\|\tilde{u}\|_{C^{m,\alpha}(\mathcal{H}_{r_0/2})}\leq C'.\label{eq:EllipticRegularity2}
\end{equation}
\end{theorem}

\begin{proof}
The proof of Theorem \ref{thm:AprioriDerivataiveBounds2} parallels the proof of Theorem \ref{thm:AprioriDerivataiveBounds} exactly. Indeed, the only modification of any non-trivial importance is that the local parameterizations yielded by Lemma \ref{lem:uniformLocalGraphs2} satisfy an elliptic partial differential equation with either a partial Dirichlet or partial Neumann boundary condition.  Nevertheless, elliptic bootstrapping at the boundary still applies, and the desired result is immediate.
\end{proof}

\section{Regularity of the second fundamental form}\label{sec:CurvatureRegularity}
The purpose of this section is to establish several regularity results for the second fundamental form of an immersed $J$-curve $u:S\to M$ of either Type 0 or Type 1.  The key results of this section are Theorem \ref{thm:EpsRegularity0} and Theorem \ref{thm:EpsRegularity1} which establish $\epsilon$-regularity of $\|B\|^2$ for curves with or without partial Lagrangian boundary. Of perhaps more practical importance are
Corollary \ref{cor:CurvatureThreshold} and Corollary \ref{cor:CurvatureThreshold1} which guarantee that if the curvature of a $J$-curve is sufficiently large at a point $\zeta\in S$, then there is a threshold amount of ``total'' curvature in a small neighborhood of $\zeta$. Note that these theorems and corollaries depend heavily upon other regularity results, namely Proposition \ref{prop:nabBbounded} and Proposition \ref{prop:nabBbounded2}, which establish elliptic regularity for the second fundamental form. These two propositions can be regarded as a parametrization-independent version of elliptic regularity for $J$-curves.

\subsection{Interior case}\label{sec:CurvatureRegularity0} Here we establish the aforementioned regularity results for immersed $J$-curves without partial Lagrangian boundary conditions, or more precisely for those $J$-curves which are Type 0 (see Definition \ref{def:Type0Map}).  To that end, we begin with an useful elementary result regarding covariant derivatives of the two second fundamental forms $A\in \Hom\big(\mathcal{N},\Hom(\mathcal{T},\mathcal{T}) \big)$ and $B\in \Hom\big(\mathcal{T}\otimes\mathcal{T},\mathcal{N} \big)$ of an immersion $u:S\to M$.  Recall $A$ and $B$ are respectively defined by $A^\nu(\tau_1):=(\nab_{\tau_1} \nu)^\top$ and $B(\tau_1,\tau_2) = (\nab_{\tau_1}\tau_2)^\bot$.

\begin{lemma}\label{lem:covDerivOfAandB}
Let $u:S\to M$ be an immersion into a Riemannian manifold $(M,g)$, and let $A$ and $B$ be the associated second fundamental forms.  Then for each $\zeta\in S$, $k\in \mathbb{N}$, and $\tau_1,\tau_2, X_1,\ldots,X_k \in \mathcal{T}_\zeta$ and $\nu\in \mathcal{N}_\zeta$, the following holds
\begin{equation}\label{eq:covDerivOfAandB}
0=\big\langle (\nab_{X_1,\ldots,X_k}^k B)(\tau_1,\tau_2),\nu\big\rangle + \big\langle \tau_2, (\nab_{X_1,\ldots,X_k}^k A)^\nu (\tau_1),  \nu\big\rangle
\end{equation}
where $\nab^k A$ and $\nab^k B$ denote the $k^{th}$ covariant derivative of $A$ and $B$ respectively. Furthermore
\begin{equation}\label{eq:covDerivOfAandB2}
\|\nab^k A\| = \|\nab^k B\|,
\end{equation}
for all $k\in \mathbb{N}$.
\end{lemma}
\begin{proof}
We begin by letting $\tau_1,\tau_2,X_1\in \Gamma(\mathcal{T})$ and $\nu\in \Gamma(\mathcal{N})$ be extensions of the given vectors so that
\begin{equation*}
0=(\nab_{X_1} \tau_1)^\top\big|_{\zeta}=(\nab_{X_1} \tau_2)^\top\big|_{\zeta} = (\nab_{X_1} \nu)^\bot\big|_{\zeta}.
\end{equation*}
Then by taking successive derivatives we find
\begin{align*}
0&=\la \tau_2, \nu\ra\\
&\Rightarrow\\
0&=\la (\nab_{\tau_1} \tau_2)^\bot,\nu\ra + \la \tau_2 , (\nab_{\tau_1} \nu)^\top \ra\\
&=\la B(\tau_1,\tau_2), \nu\ra + \la \tau_2, A^\nu(\tau_1)\ra\\
&\Rightarrow\\
0&= \la (\nab_{X_1}B)(\tau_1,\tau_2), \nu\ra + \la B(\nab_{X_1}^\top\tau_1,\tau_2), \nu\ra\\
&\qquad+ \la B(\tau_1,\nab_{X_1}^\top\tau_2), \nu\ra + \la B(\tau_1,\tau_2), \nab_{X_1}^\bot\nu\ra\\
&\qquad+ \la \nab_{X_1}^\top \tau_2, A^\nu(\tau_1)\ra
+ \la \tau_2, (\nab_{X_1}A)^\nu(\tau_1)\ra\\
&\qquad + \la \tau_2, A^{\nab_{X_1}^\bot \nu}(\tau_1)\ra+ \la \tau_2, A^\nu(\nab_{X_1}^\top\tau_1)\ra,
\end{align*}
and thus by evaluating at $\zeta$ we find that
(\ref{eq:covDerivOfAandB}) holds for $k=0,1$.  Since the quantities on the right hand side of that equation do not depend on the extensions of $\tau_1,\tau_2$, and $\nu$, we conclude that (\ref{eq:covDerivOfAandB}) holds for $k=0,1$ in general for arbitrary vectors $X_1,\tau_1,\tau_2$ and $\nu$. Iterating this procedure, and making sure to take extensions of the $X_1,\ldots,X_{k-1}$ so that
\begin{equation*}
0=(\nab_{X_k} \tau_1)^\top\big|_{\zeta}=(\nab_{X_k} \tau_2)^\top\big|_{\zeta} = (\nab_{X_k} \nu)^\bot\big|_{\zeta} = (\nab_{X_k} X_j)^\top\big|_{\zeta}
\end{equation*}
for $j=1,\ldots,k-1$, will prove (\ref{eq:covDerivOfAandB}) for all $k\in \mathbb{N}$.  Equation (\ref{eq:covDerivOfAandB2}) then follows immediately.
\end{proof}

We are now prepared to prove elliptic regularity of the second fundamental form.

\begin{proposition}[elliptic regularity of $B$]\label{prop:nabBbounded}
Fix a constant $\alpha\in (0,1)$, and an increasing sequence $\{C_k\}_{k=0}^{\infty}\subset \R^+$, then there exists another increasing sequence $\{C_k'\}_{k=0}^\infty\subset \R^+$ for which $C_k'=C_k'(\alpha,C_1,\ldots,C_{k+1})$ and which has the following significance.  Let $(M,J,g)$ be a smooth compact almost Hermitian manifold of dimension $2n$ which possibly has boundary.  Define the set $\tilde{M}:=\{p\in M:\dist (p,\partial M)\geq 2\}$, and suppose $(M,g,\tilde{M})$ is an admissible triple.  Suppose further that
\begin{equation*}
\|g\|_{C_g^{k,\alpha}(M\setminus\partial M)}+ \|J\|_{C_g^{k,\alpha}(M\setminus\partial M)}\leq C_k,
\end{equation*}
for each $k\in \mathbb{N}$. Then for any immersed $J$-curve $u:S\to M$ of Type 0 for which
\begin{equation*}
\sup_{\zeta\in S}\|B_u\|_g\leq 1,
\end{equation*}
the following also holds
\begin{equation*}
\|\nab^k B_u\|_{C_g^{0,\alpha}(u^{-1}(\tilde{M}))} \leq C_k',
\end{equation*}
where $\nab^k B_u \in \Gamma \big( \Hom(\bigotimes^{k+2}\mathcal{T},\mathcal{N})\big)$ denotes the $k^{th}$ covariant derivative of the second fundamental form $B_u$ of the immersion $u$.
\end{proposition}
\begin{proof}

For any smooth immersion $\tilde{u}:\R^2\to \R^m$ we will use the following notation for coordinate tangent vectors:
$D_1 \tilde{u} = \tilde{u}_{,1}=\tilde{u}_s = \tilde{u}_*\partial_s$, $D_2 \tilde{u} =\tilde{u}_{,2}=\tilde{u}_t = \tilde{u}_*\partial_t$.  Suppose that $\R^m$ is equipped with coordinates $(x^1,\ldots,x^m)$, and a Riemannian metric $g$, with components $g_{ij}$ and associated inverse components $g^{ij}$ for which $\delta_i^j=g_{i\ell} g^{\ell j}$.  Then  for any tangent vector fields of the form $\tau_i= \tilde{u}_{,j_i}$, repeated use of the Leibniz rule and the fact that $\nabla = \nabla^\top + \nabla^\bot$ shows that
\begin{equation}\label{eq:FormulaA}
(\nabla_{\tau_{k+2},\ldots,\tau_{3}}^kB)(\tau_2,\tau_1) = \mathcal{F}_{k,j_1,\ldots,j_{k+2}}^\ell \partial_{x^\ell}
\end{equation}
where the $\mathcal{F}_{k,j_1,\ldots,j_{k+2}}^\ell$ are polynomials of their arguments which we denote as
\begin{equation}\label{eq:FormulaB}
\mathcal{F}=\mathcal{F}\big(\bigcup_{i=1}^m\bigcup_{|\mathbf{a}|=0}^{k+2}D_{\mathbf{a}}\tilde{u}^i,
\;\;\;\bigcup_{i,j=1}^2 \gamma^{ij},
\;\;\;\bigcup_{i,j=1}^m g^{ij},
\;\;\;\bigcup_{i,j=1}^m \bigcup_{|\mathbf{a}| =0}^{k+1} D_{\mathbf{a}}g_{ij}\big);
\end{equation}
Here the $\gamma^{ij}$ satisfy $\delta_i^j=\gamma_{i\ell}\gamma^{\ell j}$ and $\gamma_{ij}=g_{\alpha\beta}\tilde{u}_{,i}^\alpha, \tilde{u}_{,j}^\beta$.  For an immersed $J$-curve $u:S\to M$, we consider the maps $\tilde{u}=\Phi\circ u \circ \phi^{-1}$ as in Lemma \ref{lem:uniformLocalGraphs}.  A consequence of Theorem \ref{thm:AprioriDerivataiveBounds} is that that for fixed $k$, the $C^{k+2,\alpha}$-norms of the $\tilde{u}$ are uniformly bounded in terms of $\|g,J\|_{C_g^{k+1,\alpha}(M\setminus \partial M)}$, which are bounded by $C_{k+1}$.  A consequence of Lemma \ref{lem:uniformLocalGraphs} is that the $\|\gamma^{ij}\|_{C^{0,\alpha}}$ are uniformly bounded by a universal constant. A consequence of Lemma \ref{lem:GammaEst} is that the $\|g^{ij}\|_{C^0}$ are bounded by a universal constant, and it is straight forward to show that the $\|g^{ij}\|_{C^{0,\alpha}}$ are bounded in terms of $C_0$. Combining (\ref{eq:FormulaA}), (\ref{eq:FormulaB}), and these estimates then proves the desired result.
\end{proof}

It will often be convenient to locally consider $J$-curves in small balls $\mathcal{B}_\epsilon(p)\subset M$ (in particular for $\epsilon<2$).  In such cases the following result is useful.
\begin{corollary}\label{cor:LocalInteriorCurvatureBounds}
Fix $\epsilon\in(0,2]$ and let $M$, $J$, $g$, $\tilde{M}$, $\alpha$, $C_k$ and $C_k'$ be as in Proposition \ref{prop:nabBbounded}; in particular assume
\begin{equation*}
\|g\|_{C_g^{k,\alpha}(M\setminus\partial M)}+\|J\|_{C_g^{k,\alpha}(M\setminus\partial M)}\leq C_k
\end{equation*}
for each $k\in \mathbb{N}$.  Fix $p\in \tilde{M}$, and suppose $u:S\to \overline{\mathcal{B}_{\epsilon}^g(p)}\subset M$ is an immersed $J$-curve of Type 0 in $\overline{\mathcal{B}_\epsilon^g(p)}$ which satisfies
\begin{equation*}
\sup_{S}\|B_u\|_g\leq 1.
\end{equation*}
Then
\begin{equation*}
\sup_{\zeta\in u^{-1}(\overline{\mathcal{B}_{\epsilon/3}^g(p)} )}\|\nab^k B_u\|_g \leq (3/\epsilon)^{k+1}C_k'.
\end{equation*}
\end{corollary}
\begin{proof}
We begin by defining $N:=\overline{\mathcal{B}_\epsilon^g(p)}$, $\tilde{N}:=\overline{\mathcal{B}_{\epsilon/3}^g(p)}\big)$ and $\tilde{g}:=(3/\epsilon)^2 g$. Then observe that $(N,\tilde{g},\tilde{N})$  is an admissible triple. Furthermore, since $\epsilon/3< 1$ it follows from Lemma \ref{lem:RescalingProperties} that
\begin{equation*}
\|\tilde{g}\|_{C_{\tilde{g}}^{k,\alpha}(N\setminus\partial N)}+\|J\|_{C_{\tilde{g}}^{k,\alpha}(N\setminus\partial N)}\leq C_k,
\end{equation*}
and $\sup_{\zeta\in S} \|B_u^{\tilde{g}}\|_{\tilde{g}}\leq 1$. By Proposition \ref{prop:nabBbounded}, it then follows that
\begin{equation*}
\sup_{\zeta \in u^{-1}(\tilde{N})}\|\nab^kB_u^{\tilde{g}}\|_{\tilde{g}} \leq C_k'.
\end{equation*}
However, by the definition of $\tilde{N}$, the fact that $\nab^k B_u^{\tilde{g}} = \nab^k B_u^{g}$, and another application of Lemma \ref{lem:RescalingProperties}, we then find
\begin{equation*}
\sup_{\zeta\in u^{-1}(\overline{\mathcal{B}_{\epsilon/3}^g(p)})}\|\nab^kB_u^{{g}}\|_{{g}} \leq (3/\epsilon)^{k+1}C_k',
\end{equation*}
which is precisely the desired estimate.
\end{proof}

With Proposition \ref{prop:nabBbounded} established, we are now prepared to state and prove the main result of Section \ref{sec:CurvatureRegularity0}, namely Theorem \ref{thm:EpsRegularity0} below.  To that end, it will be convenient to have the following notation at our disposal.  If $(M,g)$ is a compact Riemannian manifold which possibly has boundary, then for each $\delta\geq 0$ we define $M^\delta$ to be the compact set given by $M^\delta:=\{q\in M:\dist(q,\partial M)\geq \delta\}$. In the case that $\partial M = \emptyset$, then $M^\delta:=M$ for all $\delta\geq0$.

\begin{theorem}[$\epsilon$-regularity of $\|B\|^2$]\label{thm:EpsRegularity0}
Fix positive constants $C,\delta, c>0$ and $n\in \mathbb{N}$. Then there exists a positive constant $\hbar=\hbar(C,\delta,c,n)< \delta$ with the following significance. Let $(M,J,g)$ be a compact almost Hermitian manifold of dimension $2n$ which possibly has boundary, and which satisfies
\begin{equation*}
\|g\|_{C_g^{2,\alpha}(M\setminus\partial M)} + \|J\|_{C_g^{2,\alpha}(M\setminus\partial M)} \leq C\qquad\text{and}\qquad\inf_{q\in M^{2\delta}}\inj(q)\geq \delta.
\end{equation*}
Let $u:S\to M$ be an immersed $J$-curve of Type 0.
Then for each $\zeta\in u^{-1}(M^{3\delta})\subset S$ and $r\in(0,\hbar)$ for which
\begin{equation*}
\int_{S_r(\zeta)}\|B_u^g\|^2 \leq \hbar
\end{equation*}
the following inequality also holds
\begin{equation*}
\max_{0\leq \sigma \leq r} \big(\sigma^2
\sup_{S_{r-\sigma}(\zeta)}\|B_u^g\|^2 \big)\leq c^2.
\end{equation*}
\end{theorem}
\begin{proof}The proof is similar to the Choi and Schoen's proof of $\epsilon$-regularity of $\|B\|^2$ for minimal surfaces as in \cite{ChSr85}. That is, we argue by contradiction and rescaling methods. Thus we proceed by assuming the theorem is false. Then there exist sequences of almost Hermitian manifolds $(M_k,J_k,g_k)$, and $J_k$-curves $u_k:S_k\to M_k$ which satisfy the hypotheses of the Theorem, and there also exists a sequence $\epsilon_k\to 0$ and a sequence of points
$\zeta_k\in u^{-1}(M_k^{3\delta})\subset S_k$ and a sequence $r_k\in(0,\epsilon_k]$ for which
\begin{equation*}
\int_{S_{r_k}(\zeta_k)}\|B_{u_k}^{g_k}\|_{g_k}^2 = \epsilon_k \qquad\text{and}\qquad \max_{0\leq \sigma \leq r_k} \big(\sigma^2
\sup_{S_{r_k-\sigma}(\zeta_k)}\|B_{u_k}^{g_k}\|_{g_k}^2 \big)\geq c^2.
\end{equation*}
Next we let $\sigma_k \in (0, r_k]$ be chosen so that
\begin{equation*}
\sigma_k^2 \sup_{S_{r_k-\sigma_k}(\zeta_k)}\|B_{u_k}^{g_k}\|_{g_k}^2 = \max_{\sigma\in [0,r_k]} \big(\sigma^2 \sup_{S_{r_k-\sigma}(\zeta_k)}\|B_{u_k}^{g_k}\|_{g_k}^2\big),
\end{equation*}
and we let $\zeta_k'\in S_{r_k-\sigma_k}(\zeta_k)$ be chosen so that
\begin{equation*}
\|B_{u_k}^{g_k}(\zeta_k')\|_{g_k}^2 = \sup_{S_{r_k-\sigma_k}(\zeta_k)}\|B_{u_k}^{g_k}\|_{g_k}^2.
\end{equation*}
Consequently, $c^2\leq \max_{0\leq\sigma\leq r_k}\big( \sigma^2 \sup_{S_{r_k-\sigma}(\zeta_k')}\|B_{u_k}^{g_k}\|_{g_k}^2\big)
=\sigma_k^2\|B_{u_k}^{g_k}(\zeta_k')\|_{g_k}^2$. Furthermore,
\begin{align*}
\big(\frac{\sigma_k}{2}\big)^2\sup_{S_{\sigma_k/2}(\zeta_k')}\|B_{u_k}^{g_k}\|_{g_k}^2 &\leq \big(\frac{\sigma_k}{2}\big)^2\sup_{S_{r_k-\sigma_k/2}(\zeta_k')}\|B_{u_k}^{g_k}\|_{g_k}^2\\
&\leq \max_{0\leq\sigma\leq r_k}\big( \sigma^2 \sup_{S_{r_k-\sigma}(\zeta_k')}\|B_{u_k}^{g_k}\|_{g_k}^2\big),
\end{align*}
from which we conclude
\begin{equation}\label{eq:epsReg2}
\sup_{S_{\sigma_k/2}(\zeta_k')}\|B_{u_k}^{g_k}\|_{g_k}^2 \leq 4\|B_{u_k}^{g_k}(\zeta_k')\|_{g_k}^2.
\end{equation}
Also observe that since $\sigma_k\leq r_k\leq \epsilon_k\to 0$ it follows that
\begin{equation}\label{eq:epsReg1}
\frac{c^2}{c_k^{2}}:=\|B_{u_k}^{g_k}(\zeta_k')\|_{g_k}^2\geq \frac{c^2}{\sigma_k^2}\to\infty.
\end{equation}
Next we define the metrics $\tilde{g}_k:=(\ell/c_k)^{2} g_k$ where $\ell:=\max(6,2c)$.  Observe that the $(M_k,\tilde{g}_k,J_k)$ are again almost Hermitian manifolds. Also, it is straightforward to show that $B_{u_k}^{g_k}=B_{u_k}^{\tilde{g}_k}$ and that the following point-wise estimate holds:
\begin{equation}\label{eq:epsReg3}
\|B_{u_k}^{\tilde{g}_k}\|_{\tilde{g}_k} = \frac{c_k}{\ell}\|B_{u_k}^{g_k}\|_{g_k};
\end{equation}
consequently $\|B_{u_k}^{\tilde{g}_k}(\zeta_k')\|_{\tilde{g}_k}=c/\ell$.  For clarity, for each $r>0$ and $\zeta_k$, define the set $\tilde{S}_r(\zeta_k')$ to be
the connected component of $\{\zeta'\in S_k:u_k(\zeta')\in \mathcal{B}^{\tilde{g}_k}_r\big(u_k(\zeta_k')\big)\}$ which contains $\zeta_k'$.  It is straightforward to show that
\begin{equation*}
\mathcal{B}_{r}^{g_k}(p)=\mathcal{B}_{\ell r/{c_k}}^{\tilde{g}_k}(p)\qquad\text{and}\qquad S_{r}(\zeta_k') = \tilde{S}_{\ell r/{c_k}}(\zeta_k').
\end{equation*}
We then conclude from  (\ref{eq:epsReg2}) - (\ref{eq:epsReg3}), $c_k\leq \sigma_k$, and $\ell:=\max(6,2c)$ that
\begin{equation*}
\sup_{\tilde{S}_{3}(\zeta_k')}\|B_{u_k}^{\tilde{g}_k}\|_{\tilde{g}_k}=\frac{c_k}{\ell}\sup_{S_{3c_k/\ell}(\zeta_k')}\|B_{u_k}^{g_k}\|_{g_k}
\leq \frac{c_k}{\ell}\sup_{S_{\sigma_k/2 }(\zeta_k')}\|B_{u_k}^{g_k}\|_{g_k} \leq \frac{2c}{\ell}\leq 1.
\end{equation*}
Next since $u_k(\zeta_k)\in M_k^{3\delta}$, and $c_k,\epsilon_k\to 0$, it follows that for all sufficiently large $k$, we have
\begin{equation*}
\mathcal{B}_3^{\tilde{g}_k}\big(u_k(\zeta_k')\big) = \mathcal{B}_{3c_k/\ell}^{g_k}\big(u_k(\zeta_k')\big) \subset \mathcal{B}_{3c_k/\ell+\epsilon_k}^{g_k} \big(u_k(\zeta_k)\big)\subset M_k^{2\delta},
\end{equation*}
so that $\inf_{q\in \mathcal{B}_3^{\tilde{g}_k}(u_k(\zeta_k'))}\inj^{g_k}(q)\geq \delta.$
Since $\tilde{g}_k = (\ell/c_k)^2g_k$ and $c_k\to 0$, we conclude
\begin{equation*}
\inf_{q\in \mathcal{B}_1^{\tilde{g}_k}(u_k(\zeta_k'))}\inj^{\tilde{g}_k}(q)\geq 2.
\end{equation*}
We then define the compact manifolds (with boundary) $N_k$ and $\mathcal{K}_k$ to be the closures of $\mathcal{B}_3^{\tilde{g}_k} \big(u_k(\zeta_k')\big)$ and $\mathcal{B}_1^{\tilde{g}_k} \big(u_k(\zeta_k')\big)$ respectively.  We conclude that the triples $(N_k,\tilde{g}_k,\mathcal{K}_k)$ are admissible for all sufficiently large $k$, and the $J_k$-curves with restricted domains given by $\tilde{u}_k:=u_k:\tilde{S}_3(\zeta_k')\to N_k$, with image in the almost Hermitian manifolds  $\big(N_k,J_k,\tilde{g}_k\big)$, are immersed and of Type 0 and satisfy
\begin{equation*}
\|B_{\tilde{u}_k}^{\tilde{g}_k}(\zeta_k')\|_{\tilde{g}_k} = c\ell^{-1}\qquad\text{and}\qquad\sup_{\tilde{S}_{3}(\zeta_k')}\|B_{\tilde{u}_k}^{\tilde{g}_k}\|_{\tilde{g}_k}\leq 1.
\end{equation*}
By Proposition \ref{prop:nabBbounded} we conclude that there exists a $C_1>0$ which is independent of $k$, for which
\begin{equation*}
\sup_{\tilde{S}_1(\zeta_k')} \|\nab B_{\tilde{u}_k}^{\tilde{g}_k}\|_{\tilde{g}_k}\leq C_1.
\end{equation*}
Then either arguing directly, or applying Lemma \ref{lem:uniformLocalGraphs}, one finds that there exist constants $k_0\in \mathbb{N}$ and  $\hat{c} >0$ such that for all $k\geq k_0$ the following holds:
\begin{align}
\hat{c}&\leq \min \big(1, {\textstyle \frac{1}{2}} \inj^{\tilde{u}_k^*\tilde{g}_k}(\zeta_k'),c(2\ell C_1)^{-1}\big)\label{eq:EpsReg4}\\
\pi \hat{c}^2/2&\leq \Area_{\tilde{u}_k^*\tilde{g}_k}(\Sigma_k) \quad\text{where}\quad \Sigma_k:=\{ \zeta\in \tilde{S}_3(\zeta_k'): \dist_{\tilde{u}_k^*g_k}(\zeta,\zeta_k')\leq \hat{c}\}.\label{eq:EpsReg5}
\end{align}
However, for each $\zeta \in \Sigma_k$ there exists a $\tilde{u}_k^*g_k$-unit speed geodesic $\gamma$ for which $\gamma(0)=\zeta_k'$ and $\gamma(\ell)=\zeta$ and $\ell\in (0,\hat{c}]$.  Consequently, by inequality (\ref{eq:EpsReg4})
\begin{align*}
\big|c\ell^{-1} - \|B_{\tilde{u}}^{\tilde{g}_k}(\zeta)\|_{\tilde{g}_k}\big| & = \big|\|B_{\tilde{u}}^{\tilde{g}_k}(\zeta_k')\|_{\tilde{g}_k} - \|B_{\tilde{u}}^{\tilde{g}_k}(\zeta)\|_{\tilde{g}_k}\big|\\
&\leq \int_0^\ell \big|{\textstyle \frac{d}{dt}} \|B_{\tilde{u}_k}^{\tilde{g}_k}(\gamma(t))\|_{\tilde{g}_k}\big| dt \leq \int_0^{\hat{c}} \|(\nab B_{\tilde{u}_k}^{\tilde{g}_k})(\gamma(t))\|_{\tilde{g}_k} dt\\
&\leq c(2\ell)^{-1},
\end{align*}
and thus $\inf_{\zeta\in \Sigma_k}\|B_{\tilde{u}_k}^{\tilde{g}_k}\|_{\tilde{g}_k}\geq c(2\ell)^{-1}$.  Let $d{\mathcal{H}}_k^2$ and $d\tilde{\mathcal{H}}_k^2$ denote that Hausdorff measures associated to the metrics $u_k^*g_k$ and $\tilde{u}_k^* \tilde{g}_k = u_k^*\tilde{g}_k$ respectively, and observe that $d\tilde{\mathcal{H}}_k^2=(\ell/c_k)^2d{\mathcal{H}}_k^2$.  Then by inequality (\ref{eq:EpsReg5}) we find
\begin{align*}
\frac{\pi \hat{c}^2}{2}\big(\frac{c}{2\ell}\big)^2 &\leq \int_{\Sigma_k}\|B_{\tilde{u}_k}^{\tilde{g}_k}\|_{\tilde{g}_k}^2 d\tilde{\mathcal{H}}_k^2\leq \int_{\tilde{S}_3(\zeta_k')}\|B_{\tilde{u}_k}^{\tilde{g}_k}\|_{\tilde{g}_k}^2 d\tilde{\mathcal{H}}_k^2 = \int_{S_{3c_k/\ell}(\zeta_k')}\|B_{{u}_k}^{g_k}\|_{g_k}^2 d{\mathcal{H}}_k^2\\
&\leq \int_{S_{\sigma_k/2}(\zeta_k')}\|B_{u_k}^{g_k}\|_{g_k}^2 d{\mathcal{H}}_k^2 \leq \int_{S_{r_k}(\zeta_k)}\|B_{u_k}^{g_k}\|_{g_k}^2 d{\mathcal{H}}_k^2= \epsilon_k \to 0.
\end{align*}
Since $\hat{c},c,\ell>0$ are independent of $k$, we have obtained the desired contradiction.  This completes the proof of Theorem \ref{thm:EpsRegularity0}.
\end{proof}

While useful as stated above, it will be convenient to state the ``curvature threshold'' version of Theorem \ref{thm:EpsRegularity0} above, namely Corollary \ref{cor:CurvatureThreshold} below.

\begin{corollary}[Curvature Threshold]\label{cor:CurvatureThreshold}
Fix real constants $C,\delta>0$ and $n\in \mathbb{N}$. Then there exists a positive constant $\hbar=\hbar(C,\delta,n)< \delta$ with the following significance. Let $(M,J,g)$ be a compact almost Hermitian manifold of dimension $2n$ which possibly has boundary, and which satisfies
\begin{equation*}
\|g\|_{C_g^{2,\alpha}(M\setminus\partial M)} + \|J\|_{C_g^{2,\alpha}(M\setminus\partial M)} \leq C\qquad\text{and}\qquad\inf_{q\in M^{2\delta}}\inj(q)\geq \delta.
\end{equation*}
Let $u:S\to M$ be an immersed $J$-curve of Type 0.
Then for each $\zeta\in u^{-1}(M^{3\delta})\subset S$ the following statement holds for every $0<r<\hbar$:
\begin{equation*}
\text{if}\qquad \|B_u^g(\zeta)\|_g \geq \frac{1}{r}\qquad \text{then}\qquad \int_{S_r(\zeta)}\|B_u^g\|^2 \geq \hbar.
\end{equation*}
\end{corollary}
\begin{proof}
Let $\hbar>0$ be the constant yielded from Theorem \ref{thm:EpsRegularity0} for $c=1/3$, and $C$, $\delta$, and $n$ as above.  Then
\begin{equation*}
\frac{1}{9}< \frac{1}{4}\leq (r/2)^2\|B_u^g(\zeta)\|_g^2\leq (r/2)^2 \sup_{S_{r/2}(\zeta)}\|B_u^g\|_g^2 \leq \max_{0\leq\sigma\leq r}\big(\sigma^2  \sup_{S_{r-\sigma}(\zeta)}\|B_u^g\|^2\big).
\end{equation*}
We conclude from Theorem \ref{thm:EpsRegularity0} that it cannot be the case that $\int_{S_r(\zeta)}\|B\|^2 \leq \hbar$, so the desired conclusion is immediate.
\end{proof}

\subsection{Lagrangian boundary case}\label{sec:CurvatureRegularity1} Here we establish the aforementioned regularity results for immersed $J$-curves with partial Lagrangian boundary conditions, or more precisely for those $J$-curves which are Type 1 (see Definition \ref{def:Type1Map}).  It should not be surprising that many of the results below are proved by considering two types of points $\zeta \in S$, namely those that are
uniformly bounded away from $\partial_1 S$, and those which are sufficiently close to $\partial_1 S$.  For the former points, we apply results of Section \ref{sec:CurvatureRegularity0} directly, and for the latter points we mimic the proofs in Section \ref{sec:CurvatureRegularity0} while making references to results in Section \ref{sec:BoundaryCase} instead of Section \ref{sec:InteriorCase}.  To help establish this dichotomy, it will be convenient to have the following estimate comparing extrinsic and intrinsic distances.

\begin{lemma}[extrinsic/intrisic estimate]\label{lem:IntrinsicExtrinsicEstimate}
Let $(M,g)$ be a compact Riemannian manifold of dimension $2n$ which possibly has boundary, define $\tilde{M}:=\{q\in M: \dist_g(q,\partial M)\geq 2\}$, and suppose that $(M,g,\tilde{M})$ is an admissible triple (in the sense of Definition \ref{def:admissible}).  Suppose further that $u:S\to M$ is an immersion, for which $S$ is compact but possibly has boundary, and assume that $\sup_S\|B_u\|\leq 1$.  Then for each $\zeta\in S\setminus \partial S$ with the property that $u(\zeta)\in \tilde{M}$, the following also holds for all $0<r\leq \min\big(1/10,\dist_{u^*g}(\zeta,\partial S)\big)$:
\begin{equation*}
S_{r/2}(\zeta)\subset \{\zeta'\in S: \dist_{u^*g}(\zeta',\zeta)<r\};
\end{equation*}
here $S_r(\zeta)$ is defined to be the connected component of $u^{-1}\big(\mathcal{B}_r^g(u(\zeta))\big)$ which contains $\zeta$.
\end{lemma}
\begin{proof}
Begin by considering a $u^*g$-unit speed geodesic $\gamma:[0,r]\to S$ such that $\gamma(0) = \zeta$. Next, choose geodesic-normal coordinates $\Phi=(x^1,\ldots,x^{2n}):\mathcal{B}_1\big(u(\zeta)\big)\to \R^{2n}$ so that $\Phi\big(u(\zeta)\big)=0$, and $\frac{d}{dt}\big|_{t=0}x^1\big(u(\gamma(t))\big) =1$.  For convenience, define $\bar{g}:=dx^i\otimes dx^i$, and the path $\alpha:=(\alpha^1,\ldots,\alpha^{2n}) =\Phi\circ u\circ \gamma$, so that $\alpha(0)=0$ and $\dot{\alpha}^1(0) =1$.  We will also abuse notation by letting $g$ denote the pushed forward metric $\Phi_*g$.

Observe that by definition $\gamma$ is a $u^*g$-geodesic in $S$, and thus $\nab_{(u\circ \gamma)'}(u\circ\gamma)' = (\nab_{(u\circ \gamma)'}(u\circ\gamma)')^\bot = B_u(\dot{\gamma},\dot{\gamma})$, which is bounded in norm by 1 by assumption. Consequently $\|\nab_{\dot{\alpha}} \dot{\alpha}\|_g\leq 1$ which implies $\|\ddot{\alpha} + \Gamma_{\alpha}(\dot{\alpha},\dot{\alpha})\|_g\leq 1,$
where $\Gamma$ is the locally defined bi-linear form on $\R^{2n}$ associated to the Christoffel symbols in the coordinates $\Phi$.  However, since $r\leq \frac{1}{10}$ and $\|\dot{\alpha}\|_g\equiv 1$, it follows that $|\alpha(t)|:=\dist_g\big(\alpha(0),\alpha(t)\big)\leq \frac{1}{10}$ for all $t\in [0,r]$.  Furthermore, since $(M,g,\tilde{M})$ is admissible and $u(\zeta)\in \tilde{M}$, it follows by Lemma \ref{lem:GammaEst} that $\|\Gamma(\dot{\alpha},\dot{\alpha})\|_g\leq \frac{1}{2}$, and thus for all $t\in [0,r]$ we have
\begin{equation*}
|\ddot{\alpha}^1| \leq \|\ddot{\alpha}\|_{\bar{g}} \leq 2 \|\ddot{\alpha}\|_g \leq 3.
\end{equation*}
Since $\dot{\alpha}^1(0) = 1$, the mean value theorem then guarantees that for $t\in [0,1/10]$ we have $\dot{\alpha}^1(t)\geq 1/2$. Integrating up, we then find that $\alpha^1(t) \geq t/2$.  Since $\dist_g\big(\alpha(0),\alpha(t)\big)^2 = \sum_{i} \big(\alpha^i(t)\big)^2$, we conclude that $\dist_g\big(u\circ \gamma (0),u\circ\gamma (t)\big)\geq t/2$ for all $t\in [0,r]$.  The desired result is immediate.
\end{proof}

We next move on to our first main result, namely elliptic regularity for the second fundamental form.

\begin{proposition}[higher order curvature bounds]\label{prop:nabBbounded2}
Fix a constant $\alpha\in (0,1)$, and an increasing sequence $\{C_k\}_{k=0}^{\infty}\subset \R^+$, then there exists another increasing sequence $\{C_k'\}_{k=0}^\infty\subset \R^+$ for which $C_k'=C_k'(\alpha,C_1,\ldots,C_{k+1})$ and which has the following significance.  Let $(M,J,g)$ be a smooth compact almost Hermitian manifold of dimension $2n$ which possibly has boundary, and let $L\subset M$ be a compact embedded totally geodesic Lagrangian submanifold for which $\partial L = L\cap \partial M$.  For each $\epsilon>0$ define the set $\mathcal{K}^\epsilon:=\{p\in M:\dist (p,\partial M)\geq \epsilon\}$, and suppose $(M,J,g,L,\mathcal{K}^2)$ is an admissible 5-tuple.  Suppose further that
\begin{equation*}
\|g\|_{C_g^{k,\alpha}(M\setminus\partial M)}+ \|g\|_{C_{g,L}^{k,\alpha}(L\cap M\setminus\partial M)} + \|J\|_{C_g^{k,\alpha}(M\setminus\partial M)}+\|J\|_{C_{g,L}^{k,\alpha}(L\cap M\setminus \partial M)}\leq C_k,
\end{equation*}
for each $k\in \mathbb{N}$. Then for any immersed $J$-curve $u:S\to M$ of Type 1 for which
\begin{equation*}
\sup_{\zeta\in S}\|B_u\|_g\leq 1,
\end{equation*}
the following also holds
\begin{equation*}
\|\nab^k B_u\|_{C^{0,\alpha}(u^{-1}(\mathcal{K}^3))} \leq C_k',
\end{equation*}
where $\nab^k B_u \in \Gamma \big( \Hom(\bigotimes^{k+2}\mathcal{T},\mathcal{N})\big)$ denotes the $k^{th}$ covariant derivative of the second fundamental form $B_u$ of the immersion $u$.
\end{proposition}
\begin{proof}
We begin by letting $r_0$ be the positive constant guaranteed by Lemma \ref{lem:uniformLocalGraphs2}. Since $(M,J,g,L,\mathcal{K}^2)$ is an admissible 5-tuple, it also follows from Lemma \ref{lem:uniformLocalGraphs2} that for each point $\zeta\in u^{-1}(\mathcal{K}^2)\cap \partial_1 S$, there exists a map $\phi_{\zeta}:\mathcal{H}_{r_0/2}\to S$ which is a diffeomorphism with its image, satisfies $\phi_\zeta(0)=\zeta$,  and  satisfies the conclusions of that lemma. Arguing as in the proof of Proposition \ref{prop:nabBbounded}, one proves the existence of constants $\tilde{C}_k' = \tilde{C}_k'(\alpha,C_1,\ldots,C_{k+1})$ with the property that
$\|\nab^k B_u\|_{C^{0,\alpha}(\mathcal{O})}\leq \tilde{C}_k'$ whenever
\begin{equation}\label{eq:nabBbdd2eq1}
\mathcal{O}\subset \bigcup_{\zeta\in u^{-1}(\mathcal{K}^2)\cap \partial_1 S}\phi_{\zeta}(\mathcal{H}_{r_0/2}).
\end{equation}
For each $\zeta\in S$, we define $S_r(\zeta)$ to be the connected component of $u^{-1}\big(\mathcal{B}_r(u(\zeta))\big)$ that contains $\zeta$.  We then note that a consequence of Corollary \ref{cor:LocalInteriorCurvatureBounds} is that there exist constants $\hat{C}_k:=\hat{C}_k(\alpha,C_1,\ldots,C_{k+1})$ such that $\|\nab^k B_u\|_{C^{0,\alpha}(\mathcal{O})}\leq \hat{C}_k'$ whenever
\begin{equation}\label{eq:nabBbdd2eq2}
\partial S\cap \bigcup_{\zeta\in \mathcal{O}}S_{r_0/10}(\zeta) =\emptyset.
\end{equation}
Letting $C_k':=\max(\tilde{C}_k',\hat{C}_k')$, we see that to finish the proof of Proposition \ref{prop:nabBbounded2}, it is sufficient to show that for each $\zeta'\in u^{-1}(\mathcal{K}^3)$, there exists an open neighborhood $\mathcal{O}$ of $\zeta'$ such that either (\ref{eq:nabBbdd2eq1}) holds or (\ref{eq:nabBbdd2eq2}) holds.  To that end, fix $\zeta'\in u^{-1}(\mathcal{K}^3)$ and consider two cases.\\

\noindent\emph{Case I: $\dist_{u^*g}(\zeta',\partial_1 S)\leq r_0/5$.}\\
\indent Since $\dist_{u^*g}(\zeta',\partial_0 S) \geq 3$ and $r_0/5\leq 1$, there must exist $\zeta\in u^{-1}(\mathcal{K}^2)\cap \partial_1 S$ and a $u^*g$-unit speed geodesic $\gamma:[0,\ell]\to S$ such that $\gamma(0)=\zeta$, $\gamma(\ell) = \zeta'$,  and $0\leq \ell \leq r_0/5$.  Consequently $\dist_{u^*g}(\zeta,\zeta')\leq \ell \leq r_0/5$.  However by property (PL.\ref{en.P8}) of Lemma \ref{lem:uniformLocalGraphs2} it follows that there exists an open neighborhood $\mathcal{O}$ of $\zeta'$ such that $\mathcal{O}\subset \phi_\zeta(\mathcal{H}_{r_0/2})$, and thus (\ref{eq:nabBbdd2eq1}) is satisfied.\\

\noindent\emph{Case II: $\dist_{u^*g}(\zeta',\partial_1 S)> r_0/5$.}\\
\indent Since $\zeta'\in u^{-1}(\mathcal{K}^3)$, we can then conclude that $\dist_{u^*g}(\zeta',\partial S)> r_0/5$. Since $r_0\leq 1/2 $, we can then conclude from Lemma \ref{lem:IntrinsicExtrinsicEstimate} that there exists an open neighborhood $\mathcal{O}$ of $\zeta'$ such that
\begin{equation*}
S_{r_0/10}(\zeta'')\subset \{\zeta \in S: \dist_{u^*g}(\zeta,\zeta'') < r_0/5\},
\end{equation*}
for all $\zeta''\in \mathcal{O}$.
Since $\{\zeta \in S: \dist_{u^*g}(\zeta,\zeta') \leq r_0/5\}\cap \partial S = \emptyset$, we see that (\ref{eq:nabBbdd2eq2}) holds.   This completes the proof of Proposition \ref{prop:nabBbounded2}.
\end{proof}

As in Section \ref{sec:CurvatureRegularity0}, the establishment of elliptic regularity of $B$ allows one to prove $\epsilon$-regularity of $\|B\|^2$ and a ``curvature threshold.''  We accomplish this at present.

\begin{theorem}[$\epsilon$-regularity of $\|B\|^2$]\label{thm:EpsRegularity1}
Fix real constants $C,\delta, c>0$ and $n\in \mathbb{N}$. Then there exists a positive constant $\hbar=\hbar(C,\delta,c,n)< \delta$ with the following significance. Let $(M,J,g)$ be a compact almost Hermitian manifold of dimension $2n$ which possibly has boundary, and which satisfies
\begin{equation*}
\|g\|_{C_g^{2,\alpha}(M\setminus\partial M)} + \|J\|_{C_g^{2,\alpha}(M\setminus\partial M)} \leq C\qquad\text{and}\qquad\inf_{q\in M^{2\delta}}\inj(q)\geq \delta.
\end{equation*}
Suppose further that $L\subset M$ is a compact embedded totally geodesic Lagrangian submanifold for which $\partial L=L\cap \partial M$, and
\begin{equation*}
\|g\|_{C_{g,L}^{2,\alpha}(L\setminus\partial L)} + \|J\|_{C_{g,L}^{2,\alpha}(L\setminus\partial L)} \leq C\qquad\text{and}\qquad \inf_{q\in L\cap M^{2\delta}}\inj^L(q)\geq \delta.
\end{equation*}
Let $u:S\to M$ be an immersed $J$-curve of Type 1.
Then for each $\zeta\in u^{-1}(M^{3\delta})$ and $r\in(0,\hbar)$ for which
\begin{equation*}
\int_{S_r(\zeta)}\|B_u^g\|^2 \leq \hbar
\end{equation*}
the following inequality also holds
\begin{equation*}
\max_{0\leq \sigma \leq r} \big(\sigma^2
\sup_{S_{r-\sigma}(\zeta)}\|B_u^g\|^2 \big)\leq c^2.
\end{equation*}
\end{theorem}
\begin{proof}
The proof is very similar to that of Theorem \ref{thm:EpsRegularity0}.  Indeed, suppose not, then argue that there must exist
sequences of almost Hermitian manifolds $(M_k,J_k,g_k)$, Lagrangian submanifolds $L_k\subset M_k$, and $J_k$-curves $u_k:S_k\to M_k$ which satisfy the hypotheses of the Theorem, and there also exists a sequence $\epsilon_k\to 0$ and a sequence of points
$\zeta_k\in u^{-1}(M_k^{3\delta})\subset S_k$ and a sequence $r_k\in(0,\epsilon_k]$ for which
\begin{equation*}
\int_{S_{r_k}(\zeta_k)}\|B_{u_k}^{g_k}\|_{g_k}^2 = \epsilon_k \qquad\text{and}\qquad \max_{0\leq \sigma \leq r_k} \big(\sigma^2
\sup_{S_{r_k-\sigma}(\zeta_k)}\|B_{u_k}^{g_k}\|_{g_k}^2 \big)\geq c^2.
\end{equation*}
Choose $\sigma_k\in (0,r_k]$ and $\zeta_k'\in S_{r_k-\sigma_k}(\zeta_k)$ as in the proof of Theorem \ref{thm:EpsRegularity0}, in particular so that
$\max_{0\leq\sigma\leq r_k}\big( \sigma^2 \sup_{S_{r_k-\sigma}(\zeta_k')}\|B_{u_k}^{g_k}\|_{g_k}^2\big)
=\sigma_k^2\|B_{u_k}^{g_k}(\zeta_k')\|_{g_k}^2\geq c^2$; define $c_k := c\|B_{u_k}^{g_k}(\zeta_k')\|_{g_k}^{-1}$ and $S_r(\zeta_k')$ as before, so that
\begin{equation}
\sup_{S_{\sigma_k/2}(\zeta_k')}\|B_{u_k}^{g_k}\|_{g_k} \leq 2\|B_{u_k}^{g_k}(\zeta_k')\|_{g_k}=2c/c_k.
\end{equation}
Also as before, define $\tilde{g}_k:(\ell/c_k)^2 g_k$, where $\ell:=\max(8,2c)$, so that
\begin{equation}\label{eq:epsRegA}
\|B_{u_k}^{\tilde{g}_k}(\zeta_k')\|_{\tilde{g}_k} = c\ell^{-1}\qquad\text{and}\qquad\sup_{\tilde{S}_{4}(\zeta_k')}\|B_{u_k}^{\tilde{g}_k}\|_{\tilde{g}_k}\leq 1.
\end{equation}
Letting $N_k$ denote the closure of $\mathcal{B}_4^{\tilde{g}_k}\big(u_k(\zeta_k')\big)$, we see that the restricted maps $\tilde{u}_k:\tilde{S}_4(\zeta_k')\to N_k$ are immersed $J_k$-curves of Type 1 which satisfy (\ref{eq:epsRegA}).  It follows from Proposition \ref{prop:nabBbounded2} that there exists a $C_1>0$ with the property that for all sufficiently large $k$, the following holds
\begin{equation*}
\sup_{\tilde{S}_1(\zeta_k')} \|\nab B_{\tilde{u}_k}^{\tilde{g}_k}\|_{\tilde{g}_k}\leq C_1.
\end{equation*}
At this point we pass to a subsequence so that precisely one of the two following cases holds.

\noindent\emph{Case I: There exists $\epsilon'>0$ such that $\dist_{\tilde{u}_k^*\tilde{g}_k}(\zeta_k',\partial_1 S_k) \geq \epsilon'$.}\\
\indent In this case one employs Lemma \ref{lem:uniformLocalGraphs} to find a $\hat{c}>0$ such that for
\begin{equation}\label{eq:epsRegB}
\Sigma_k:=\{ \zeta\in \tilde{S}_4(\zeta_k'): \dist_{\tilde{u}_k^*g_k}(\zeta,\zeta_k')\leq \hat{c}\},
\end{equation}
the following hold
\begin{equation*}
\pi \hat{c}^2/2 \leq \Area_{\tilde{u}_k^*\tilde{g}_k}(\Sigma_k) \qquad\text{and}\qquad  \inf_{\zeta\in \Sigma_k}\|B_{\tilde{u}_k}^{\tilde{g}_k}\|_{\tilde{g}_k}\geq c(2\ell)^{-1}.
\end{equation*}
As the proof of Theorem \ref{thm:EpsRegularity0}, the desired contradiction follows from the scale invariance of the ``total'' curvature $\int \|B\|^2 d\mathcal{H}^2$, and  the fact that by construction $\Sigma_k\subset S_{r_k}(\zeta_k)$, on which the total curvature is assumed to be arbitrarily small.\\

\noindent\emph{Case II: $\dist_{\tilde{u}_k^*\tilde{g}_k}(\zeta_k',\partial_1 S_k)\to 0$.}\\
\indent In this case one employs Lemma \ref{lem:uniformLocalGraphs2} to find a $\hat{c}>0$ such that for $\Sigma_k$ defined as in (\ref{eq:epsRegB}), the following hold
\begin{equation*}
\pi\hat{c}^2/4 \leq \Area_{\tilde{u}_k^*\tilde{g}_k}(\Sigma_k) \qquad\text{and}\qquad  \inf_{\zeta\in \Sigma_k}\|B_{\tilde{u}_k}^{\tilde{g}_k}\|_{\tilde{g}_k}\geq c(2\ell)^{-1}.
\end{equation*}
The remainder of the proof then follows as in Case I.
\end{proof}

As in the previous section, we now state the associated ``curvature threshold'' version of the above result for the Lagrangian boundary case.

\begin{corollary}\label{cor:CurvatureThreshold1}
Fix real constants $C,\delta>0$ and $n\in \mathbb{N}$. Then there exists a positive constant $\hbar=\hbar(C,\delta,n)< \delta$ with the following significance. Let $(M,J,g)$ be a compact almost Hermitian manifold of dimension $2n$ which possibly has boundary, and which satisfies
\begin{equation*}
\|g\|_{C_g^{2,\alpha}(M\setminus\partial M)} + \|J\|_{C_g^{2,\alpha}(M\setminus\partial M)} \leq C\qquad\text{and}\qquad\inf_{q\in M^{2\delta}}\inj(q)\geq \delta.
\end{equation*}
Suppose further that $L\subset M$ is a compact embedded totally geodesic Lagrangian submanifold for which $\partial L=L\cap \partial M$, and
\begin{equation*}
\|g\|_{C_{g,L}^{2,\alpha}(L\setminus\partial L)} + \|J\|_{C_{g,L}^{2,\alpha}(L\setminus\partial L)} \leq C\qquad\text{and}\qquad \inf_{q\in L\cap M^{2\delta}}\inj^L(q)\geq \delta.
\end{equation*}
Let $u:S\to M$ be an immersed $J$-curve of Type 1.
Then for each $\zeta\in u^{-1}(M^{3\delta})$ and $0<r<\hbar$ the following statement holds
\begin{equation*}
\text{if}\qquad \|B_u^g(\zeta)\|_g \geq \frac{1}{r}\qquad\text{then}\qquad\int_{S_r(\zeta)}\|B_u^g\|^2 \geq \hbar.
\end{equation*}
\end{corollary}
\begin{proof}
The proof is essentially identical to that of Corollary \ref{cor:CurvatureThreshold}, with reference to Theorem \ref{thm:EpsRegularity0} replaced with reference to Theorem \ref{thm:EpsRegularity1}.
\end{proof}

\appendix

\section{The inhomogeneous equation in local coordinates}  The purpose of this section is to prove Proposition \ref{prop:GraphicalMinimalSurfaceSystem}, which is a quasi-linear elliptic partial differential equation for graphically (but not conformally) parameterized $J$-curves. The following argument closely follows Micallef and White's argument for minimal sub-manifolds in the appendix of  \cite{MmWb94} with some additional details and modifications for the inhomogeneous mean curvature equation which $J$-curves satisfy.

We begin by establishing some notation.  Indeed, let $(M,J,g)$ be an almost Hermitian manifold, $x:M\to \R^{2n}$ local coordinates on $M$, $(u,S)$ a generally immersed (but not necessarily pseudo-holomorphic) curve, with local coordinates $y:S\to \R^2$.  Suppose further that the map $x\circ u\circ y^{-1}:\mathcal{O}\subset \R^2\to\R^{2n}$ can be written as
\begin{align*}
x^1(y^1,y^2)&:=\big(\pi_1\circ x\circ u \circ y^{-1}\big) (y^1,y^2)\\
x^2(y^1,y^2)&:=\big(\pi_2\circ x\circ u \circ y^{-1}\big) (y^1,y^2)\\
&\;\;\vdots\\
x^{2n}(y^1,y^2)&:=\big(\pi_{2n}\circ x\circ u \circ y^{-1}\big) (y^1,y^2),
\end{align*}
where $\pi_\mu:\R^{2n}\to \R$ denotes the canonical projection of the $\mu^{th}$ component.

Here and throughout, we will employ the Einstein summation notation of summing over repeated indices; furthermore, for clairity, we will use Roman indices to denote summations from $1$ to $2$, and Greek
indices to denote summations from $1$ to $2n$.  As a consequence of this notation, we can express the metrics $g$ and $\gamma:=u^*g$ in local coordinates respectively as $g_{\alpha\beta}dx^\alpha \otimes dx^\beta$ and $\gamma_{ij}dy^i\otimes dy^j$; define $g^{\alpha\beta}$ and $\gamma^{ij}$ respectively by $g_{\alpha \mu}g^{\mu \beta}=\delta_{\alpha \beta}$ and $\delta_{ij}=\gamma_{ik}\gamma^{kj}$, and note that $\gamma_{ij}$ and $g_{\alpha\beta}$ are related by the following
\begin{equation}\label{eq:GAndGamma}
\gamma_{ij}=x_{,i}^\alpha x_{,j}^\beta g_{\alpha \beta},
\end{equation}
where we have employed the notation $x_{,i}^\alpha=\partial_{y^i} (x^\alpha)$. Recall that in these local coordinates, we can express the Levi-Civita connection as
\begin{equation*}
\nab_{\partial_{x^\alpha}}\partial_{x^\beta}=\Gamma_{\alpha\beta}^\mu\partial_{x^\mu}\qquad\text{where}\qquad\Gamma_{\alpha\beta}^\mu= {\textstyle \frac{1}{2}}g^{\nu \mu}(g_{\nu \alpha,\beta}+g_{\beta\nu,\alpha}-g_{\alpha\beta,\nu}),
\end{equation*}
and $g_{\alpha\beta,\mu}:=\partial_{x^\mu}(g_{\alpha\beta})$.  Throughout the remainder of this section, we will let $\Delta$ denote the Laplace-Beltrami operator on $S$, which is given by
\begin{equation*}
\Delta f = \frac{1}{\sqrt{\gamma}} ( \sqrt{\gamma} \gamma^{ij} f_{,i})_{,j},
\end{equation*}
where we have abused notation in standard fashion by letting $\gamma:=\det(\gamma_{ij})$. Lastly, recall the $(2,1)$-tensor $Q$ on $M$ given by $Q(X,Y)=J(\nab_X J)Y$; we then write the components of $Q$ as
\begin{equation*}
Q(\partial_{x^\alpha},\partial_{x^\beta})=Q_{\alpha\beta}^\mu\partial_{x^\mu}.
\end{equation*}

With our notation established, we now move on to proving the main result of this section. To that end we first establish a second order partial differential equation satisfied by arbitrary smooth (or at least $C^2$) maps $u:S\to M$.

\begin{lemma}\label{lem:TangentEquation}
Any immersion $u:S\to M$ expressed in local coordinates as above satisfies the following equation,
\begin{equation*}
\big((\Delta x^\mu) \partial_{x^\mu}+ \gamma^{ij} x_{,i}^\alpha x_{,j}^\beta \Gamma_{\alpha \beta}^\mu\partial_{x^\mu}\big)^\top = 0,
\end{equation*}
where $\partial_{x^\alpha}\mapsto (\partial_{x^\alpha})^\top$ denotes the orthogonal projection onto the tangent bundle $\mathcal{T}$, and $\Delta$, $\gamma^{ij}$, and $\Gamma_{\alpha\beta}^\mu$ are as in the beginning of this section.
\end{lemma}
\begin{proof}
We begin by verifying the following equation:
\begin{equation}\label{eq:IndexJuggling2}
\frac{1}{\sqrt{\gamma}}(\sqrt{\gamma})_{,k}={\textstyle \frac{1}{2}} \gamma^{ij}(x_{,ik}^\alpha x_{,j}^\beta g_{\alpha\beta} + x_{,i}^\alpha x_{,jk}^\beta g_{\alpha\beta} + x_{,i}^\alpha x_{,j}^\beta x_{,k}^\nu g_{\alpha\beta,\nu}).
\end{equation}
Indeed, to see that equation (\ref{eq:IndexJuggling2}) holds, we recall Jacobi's formula for invertible matrices: $d \det (A)= \det(A)A^{ij}dA_{ij}$, from which it follows that $\frac{1}{\gamma} (\gamma)_{,k}=\gamma^{ij}\gamma_{ij,k}$, and hence $\frac{1}{\sqrt{\gamma}}(\sqrt{\gamma})_{,k}=\frac{1}{2}\gamma^{ij}\gamma_{ij,k}$. The validity of equation (\ref{eq:IndexJuggling2}) then follows by differentiating equation (\ref{eq:GAndGamma}).

Next we recall that in our local coordinates, the vector fields $\frac{\partial x}{\partial y^i} =  x_{,i}^\alpha \partial_{x^\alpha}$ for $i=1,2$ form a (non-orthonormal) frame field for the tangent bundle $\mathcal{T}\to S$.  Thus to prove Lemma \ref{lem:TangentEquation}, it suffices to show that
\begin{equation}\label{eq:IndexJuggling3}
\la \Delta x^\mu \partial_{x^\mu}, x_{,\ell}^\nu \partial_{x^\nu}\ra=-\la \gamma^{ij} x_{,i}^\alpha x_{,j}^\beta \Gamma_{\alpha \beta}^\mu\partial_{x^\mu},x_{,\ell}^\nu \partial_{x^\nu}\ra,
\end{equation}
for $\ell=1,2$.  To that end, we compute as follows.
\begin{align*}
\la \Delta x^\mu \partial_{x^\mu}, x_{,\ell}^\nu \partial_{x^\nu}\ra&=(\Delta x^\mu) x_{,\ell}^\nu g_{\mu\nu}\\
&=g_{\mu\nu} x_{,\ell}^\nu \frac{1}{\sqrt{\gamma}}(\sqrt{\gamma} \gamma^{ij} x_{,i}^\mu)_{,j}\\
&=\frac{1}{\sqrt{\gamma}}(\sqrt{\gamma})_{,\ell}+\gamma_{\ell i}\gamma_{,j}^{ij}+g_{\mu\nu} x_{,\ell}^\nu \gamma^{ij} x_{,ji}^\mu\\
&={\textstyle \frac{1}{2}} \gamma^{ij}(x_{,i\ell}^\alpha x_{,j}^\beta g_{\alpha\beta} + x_{,i}^\alpha x_{,j\ell}^\beta g_{\alpha\beta} + x_{,i}^\alpha x_{,j}^\beta x_{,\ell}^\rho g_{\alpha\beta,\rho})\\
&\qquad-\gamma^{ij}(x_{,\ell j}^\alpha x_{,i}^\beta g_{\alpha \beta} + x_{,\ell}^\alpha x_{,ij}^\beta g_{\alpha \beta} + x_{,\ell}^\alpha x_{,i}^\beta g_{\alpha \beta, \rho}x_{,j}^\rho)\\
&\qquad +g_{\mu\nu} x_{,\ell}^\nu \gamma^{ij} x_{,ji}^\mu\\
&={\textstyle \frac{1}{2}} \gamma^{ij}x_{,i}^\alpha x_{,j}^\beta x_{,\ell}^\rho g_{\alpha\beta,\rho}-\gamma^{ij}x_{,\ell}^\alpha x_{,i}^\beta g_{\alpha \beta, \rho}x_{,j}^\rho\\
&=\gamma^{ij} x_{,i}^\alpha x_{,\ell}^\beta x_{,j}^\rho  (\textstyle {\frac{1}{2}}  g_{\alpha \rho,\beta} -   g_{\alpha \beta, \rho}).
\end{align*}
We similarly compute
\begin{align*}
\la \gamma^{ij} x_{,i}^\alpha x_{,j}^\beta \Gamma_{\alpha \beta}^\mu\partial_{x^\mu},x_{,\ell}^\rho \partial_{x^\rho}\ra&=\gamma^{ij}x_{,i}^\alpha x_{,j}^\beta{\textstyle\frac{1}{2}}g^{\nu\mu}(g_{\beta \nu,\alpha}+g_{\alpha \nu,\beta}-g_{\alpha \beta, \nu}) x_{,\ell}^\rho g_{\mu\rho}\\
&=-\gamma^{ij} x_{,i}^\alpha x_{,\ell}^\beta x_{,j}^\rho  (\textstyle {\frac{1}{2}}  g_{\alpha \rho,\beta} -   g_{\alpha \beta, \rho}),
\end{align*}
which proves equation (\ref{eq:IndexJuggling3}) as well as Lemma \ref{lem:TangentEquation}.
\end{proof}

\begin{lemma}\label{lem:NormalEquation}
Any immersion $u:S\to M$ expressed in local coordinates as above satisfies the following mean curvature equation,
\begin{equation}\label{eq:NormalEquation}
H=\big((\Delta x^\mu)\partial_{x^\mu} \big)^\bot +
\big(\gamma^{ij}x_{,i}^\alpha   x_{,j}^\beta
\Gamma_{\alpha\beta}^\mu \partial_{x^\mu}\big)^\bot,
\end{equation}
where $\partial_{x^\mu}\mapsto (\partial_{x^\mu})^\bot$ denotes the orthogonal projection to the normal bundle $\mathcal{N}\to S$, and $\Delta$, $\gamma^{ij}$, and $\Gamma_{\alpha\beta}^\mu$ are defined as in the beginning of this section.
\end{lemma}
\begin{proof}
We begin by choosing an locally defined orthonormal frame field $\{e_1,e_2\}$ of the tangent bundle $\mathcal{T}\to S$, and we define the functions
$ a_{ij}=\la \frac{\partial x}{\partial y^i},e_j \ra $.  Define $a^{ij}$ so that $a_{ik}a^{kj}=\delta_{ij}$; then $e_i = a^{ij}\frac{\partial x}{\partial y^j}$, so
\begin{equation*}
\delta_{\ell k}=a^{\ell
i}a^{k j}\gamma_{ij} \quad\Rightarrow\quad a_{ik}=a^{k j}\gamma_{ij}\quad\Rightarrow \quad\gamma^{ji} a_{ik}=a^{k j} \quad\Rightarrow \quad\gamma^{ij}=a^{k i}a^{k j}.
\end{equation*}
Then
\begin{align*}
H&=\tr_S B = B(e_k,e_k)= \big(\nabla_{e_k}e_k \big)^\bot= a^{k i} a^{k j} \big(\nabla_{\frac{\partial
x}{\partial y^i}} \frac{\partial x}{\partial y^j} \big)^\bot\\
&=\gamma^{ij}\big(x_{,ij}^\mu \partial_{x^\mu} + x_{,i}^\alpha   x_{,j}^\beta
\Gamma_{\alpha\beta}^\mu \partial_{x^\mu}\big)^\bot\\
&=\big((\Delta x^\mu)\partial_{x^\mu} \big)^\bot +
\big(\gamma^{ij}x_{,i}^\alpha   x_{,j}^\beta
\Gamma_{\alpha\beta}^\mu \partial_{x^\mu}\big)^\bot,
\end{align*}
where we have made use of the fact that
\begin{align*}
\big((\Delta x^\mu)\partial_{x^\mu} \big)^\bot &=
\frac{1}{\sqrt{\gamma}}(
\gamma^{ij}\sqrt{\gamma}  x_{,j}^\mu)_{,i} (\partial_{x^\mu} )^\bot\notag\\
&=
\frac{1}{\sqrt{\gamma}}(
\gamma^{ij}\sqrt{\gamma})_{,i} \big(\frac{\partial x}{\partial y^j} \big)^\bot + (
\gamma^{ij}  x_{,ij}^\mu \partial_{x^\mu} )^\bot\\
&= (
\gamma^{ij}  x_{,ij}^\mu \partial_{x^\mu} )^\bot.
\end{align*}
Thus we have established equation (\ref{eq:NormalEquation}) and Lemma \ref{lem:NormalEquation}.
\end{proof}

\begin{proposition}\label{prop:NonGraphicalCoordinateEquation}
Let $(u,S)$ be an immersed $J$-curve expressed in local coordinates as above.  Then
\begin{equation}\label{eq:NonGraphicalCoordinateEquation}
\Delta x^\mu + \gamma^{ij} x_{,i}^\alpha x_{,j}^\beta (\Gamma_{\alpha \beta}^\mu-Q_{\alpha \beta}^\mu)=0,
\end{equation}
where $\Delta$, $\gamma^{ij}$, $\Gamma_{\alpha\beta}^\mu$ and $Q_{\alpha\beta}^\mu$ are defined as in the beginning of this section.
\end{proposition}
\begin{proof}
We begin by recalling Lemma \ref{lem:InhomogeneousMeanCurvatureEq}, which guarantees that $(u,S)$ satisfies the inhomogeneous mean curvature equation $H=\tr_S Q$.  Next we recall that as a consequence of Lemma \ref{Lem:ComputationalAlmostHermition} (equation (\ref{eq:nabJImageComplexOrthogToDomain}) in particular) we see that for any $J$-curve $(u,S)$, the restricted tensor
\begin{equation*}
Q^\top:\mathcal{T}\otimes \mathcal{T} \to \mathcal{T}\qquad\text{given by}\qquad Q^\top(X,Y)=\big(Q(X,Y)\big)^\top,
\end{equation*}
vanishes identically.  Thus with $\{e_1, e_2\}$, $a_{ij}$, and $a^{ij}$ defined as in the proof of Lemma \ref{lem:NormalEquation}, we deduce from Lemmas \ref{lem:TangentEquation}  and  \ref{lem:NormalEquation} that
\begin{align*}
(\Delta x^\mu) \partial_{x^\mu}+ \gamma^{ij} x_{,i}^\alpha x_{,j}^\beta \Gamma_{\alpha \beta}^\mu\partial_{x^\mu}  &= \tr_S Q = Q(e_k,e_k) =a^{k i} a^{k j} Q\big(\frac{\partial x}{\partial y^i},\frac{\partial x}{\partial y^j}\big)\\
&= \gamma^{ij} x_{,i}^\alpha x_{,j}^\beta Q_{\alpha \beta}^\mu\partial_{x^\mu}.
\end{align*}
Equation (\ref{eq:NonGraphicalCoordinateEquation}) and Proposition \ref{prop:NonGraphicalCoordinateEquation} follow immediately.
\end{proof}

\begin{proposition}\label{prop:GraphicalMinimalSurfaceSystem}
Let $(u,S)$ be an immersed $J$-curve parameterized in local coordinates as above.  Suppose further that this (non-conformal) parametrization is graphical, that is
\begin{equation*}
x^1(y^1,y^2)=y^1\qquad\text{and}\qquad x^2(y^1,y^2)=y^2,
\end{equation*}
with notation as in the beginning of this section. Then for $\mu=3,\ldots, 2n$, the $x^\mu$ satisfy the following partial differential equation
\begin{equation*}
\gamma^{ij}D_{ij}x^\mu = \mathcal{F}^\mu,
\end{equation*}
where $D_i=\frac{\partial}{\partial y^i}$, $D_{ij}=\frac{\partial^2}{\partial y^i \partial y^j}$, and $\mathcal{F}^\mu$ is given by
\begin{equation*}
\mathcal{F}^\mu:=\gamma^{ij} (D_i x^\alpha) (D_j x^\beta) \big((D_k x^\mu)\Gamma_{\alpha \beta}^k-(D_k x^\mu) Q_{\alpha \beta}^k +Q_{\alpha \beta}^\mu-\Gamma_{\alpha \beta}^\mu \big),
\end{equation*}
and $\gamma^{ij}$, $\Gamma_{\alpha\beta}^\mu$, and $Q_{\alpha\beta}^\mu$ are defined as in the beginning of this section.
\end{proposition}
\begin{proof}
Begin by observing that in this graphical case we have $x_{,i}^j=\delta_{ij}$ for $i,j=1,2$; substituting this equality into equation (\ref{eq:NonGraphicalCoordinateEquation}) for $\mu=1,2$ then yields
\begin{equation}
\frac{1}{\sqrt{\gamma}} ( \sqrt{\gamma} \gamma^{ij} )_{,j}+ \gamma^{pq} x_{,p}^\alpha x_{,q}^\beta (\Gamma_{\alpha \beta}^i-Q_{\alpha \beta}^i)=0,
\end{equation}
for $i=1,2$.  Consequently, for $\mu=3,\ldots,2n$ we find
\begin{align*}
0&=\frac{1}{\sqrt{\gamma}} ( \sqrt{\gamma} \gamma^{ij} x^\mu_{,i})_{,j}+ \gamma^{ij} x_{,i}^\alpha x_{,j}^\beta (\Gamma_{\alpha \beta}^\mu-Q_{\alpha \beta}^\mu),\\
&=  \gamma^{ij} x^\mu_{,ij}+\frac{1}{\sqrt{\gamma}} ( \sqrt{\gamma} \gamma^{ij})_{,j} x^\mu_{,i}+ \gamma^{ij} x_{,i}^\alpha x_{,j}^\beta (\Gamma_{\alpha \beta}^\mu-Q_{\alpha \beta}^\mu),\\
&=  \gamma^{ij} x^\mu_{,ij}+\gamma^{pq} x_{,p}^\alpha x_{,q}^\beta (Q_{\alpha \beta}^i-\Gamma_{\alpha \beta}^i) x^\mu_{,i}+ \gamma^{ij} x_{,i}^\alpha x_{,j}^\beta (\Gamma_{\alpha \beta}^\mu-Q_{\alpha \beta}^\mu).
\end{align*}
Rearranging and relabeling terms then yields
\begin{equation*}
\gamma^{ij} x^\mu_{,ij}=\gamma^{ij} x_{,i}^\alpha x_{,j}^\beta \big(x^\mu_{,k}\Gamma_{\alpha \beta}^k-x^\mu_{,k}Q_{\alpha \beta}^k +Q_{\alpha \beta}^\mu-\Gamma_{\alpha \beta}^\mu \big),
\end{equation*}
for $\mu=3,\ldots 2n$, which is precisely the desired equation.  This completes the proof of Lemma \ref{prop:GraphicalMinimalSurfaceSystem}.
\end{proof}

\bibliographystyle{amsplain}
%\bibliography{EstimatesForJcurves}
\bibliography{../Bibliography/Bibliography}

\end{document}